\documentclass[aop,MSNbibl,seceqn,dvips]{arximspdf}

\doi{10.1214/13-AOP878} %kopijuoti is PTS
\volume{43}
\issue{1}
\pubyear{2015}
\firstpage{286}
\lastpage{338}

\makeatletter
\def\triangle{\Delta}

\newtheorem{theorem}{Theorem}[section]
\newtheorem{lemma}[theorem]{Lemma}
\newtheorem{proposition}[theorem]{Proposition}

\newproclaim{definition}[theorem]{Definition}
\newproclaim{remark}[theorem]{Remark}

\newcommand{\mc}[1]{{\mathcal #1}}

\newcommand{\bb}[1]{{\mathbb #1}}

\newcommand{\Z}{\mathbb Z}
\newcommand{\E}{\mathbb E}
\newcommand{\R}{\mathbb R}
\newcommand{\N}{\mathbb N}
\renewcommand{\S}{\mathbb S}
\newcommand{\Y}{\mathcal{Y}}
\newcommand{\M}{\mathcal{M}}
\newcommand{\I}{\mathcal{I}}
\newcommand{\B}{\mathcal{B}}
\newcommand{\D}{\mathcal{D}}
\newcommand{\C}{\mathcal{C}}
\newcommand{\A}{\mathcal{A}}

\renewcommand{\P}{\mathbb P}
\newcommand{\G}{\mathcal{G}}

\newcommand{\W}{\mathcal{W}}
\newcommand{\K}{\mathcal{K}}
\newcommand{\HW}{\widetilde{H}}
\renewcommand{\d}{w}

\newcommand{\eqref}[1]{(\ref{#1})}
\makeatother

\begin{document}
\begin{frontmatter}

\title{A stochastic Burgers equation from a class of microscopic interactions}
\runtitle{Stochastic Burgers from microscopic interactions}

\begin{aug}
\author[A]{\fnms{Patr{\'\i}cia} \snm{Gon\c{c}alves}\ead[label=e1]{patricia@mat.puc-rio.br}\thanksref{T1}},
\author[B]{\fnms{Milton} \snm{Jara}\ead[label=e2]{mjara@impa.br}}
\and
\author[C]{\fnms{Sunder} \snm{Sethuraman}\corref{}\ead[label=e3]{sethuram@math.arizona.edu}\thanksref{T2}}
\runauthor{P. Gon\c{c}alves, M. Jara and S. Sethuraman}
\thankstext{T1}{Supported by FCT research project ``Nonequilibrium
Statistical Physics'' PTDC/\break MAT/109844/2009 and PEst-C/MAT/UI0013/2011.}

\thankstext{T2}{Supported in part by NSF DMS-11-59026.}
\affiliation{PUC-RIO and Universidade do Minho, IMPA and University of Arizona}
\address[A]{P. Gon\c{c}alves\\
Departamento de Matemitica\\
UC-RIO\\
Rua Marquzs de Sco Vicente\\
no. 225, 22453-900\\
Givea, Rio de Janeiro\\
Brazil\\
and\\
CMAT\\
Centro de Matem\'atica\\
Universidade do Minho\\
Campus de Gualtar\\
4710-057 Braga\\
Portugal \\
\printead{e1}}
\address[B]{M. Jara\\
IMPA\\
Estrada Dona Castorina, 110\\
Horto, Rio de Janeiro\\
Brasil\\
\printead{e2}}
\address[C]{S. Sethuraman\\
Department of Mathematics\\
University of Arizona\\
617 N. Santa Rita Ave.\\
Tucson, Arizona 85721\\
USA\\
\printead{e3}}
\end{aug}

% HISTORY:
\received{\smonth{9} \syear{2012}}
\revised{\smonth{7} \syear{2013}}

% ABSTRACT
%
\begin{abstract}
We consider a class of nearest-neighbor weakly asymmetric mass
conservative particle systems evolving on $\mathbb{Z}$, which includes
zero-range and types of exclusion processes, starting from a
perturbation of a stationary state. When the weak asymmetry is of order
$O(n^{-\gamma})$ for $1/2<\gamma\leq1$, we show that the scaling limit
of the fluctuation field, as seen across process characteristics, is a
generalized Ornstein--Uhlenbeck process. However, at the critical weak
asymmetry when $\gamma= 1/2$, we show that all limit points
satisfy a martingale formulation which may be interpreted in terms of a
stochastic Burgers equation derived from taking the gradient of the KPZ
equation. The proofs make use of a sharp ``Boltzmann--Gibbs'' estimate
which improves on earlier bounds.
\end{abstract}

% KEYWORDS
\begin{keyword}[class=AMS]
\kwd[Primary ]{60K35}
\kwd[; secondary ]{82C22}
\end{keyword}
\begin{keyword}
\kwd{KPZ equation}
\kwd{Burgers}
\kwd{weakly asymetric}
\kwd{zero-range}
\kwd{kinetically constrained}
\kwd{speed-change}
\kwd{fluctuations}
\end{keyword}

\end{frontmatter}

%s1 #&#
\section{Introduction}
\label{intro_section}
There has been much recent work on the classification of fluctuations
of certain interfaces and currents, corresponding to mass conservative
particle dynamics in one-dimensional
nearest-neighbor interacting particle systems such as simple exclusion
and its variants, with respect to so-called Edwards--Wilkinson (EW) and
Kardar--Parisi--Zhang (KPZ) classes (cf. \cite{Corwinreview} for a
review and references).
Following recent sensibilities, a $d=1$ particle system is in the EW
class if the standard deviation of the associated ``height'' function
$h_t(x)$ of the interface at time $t$ and space point $x$, or the
integrated current at time $t\geq0$ across the space point $x\in\R$,
is of order $t^{1/4}$, and also spatial correlations are nontrivial at
range $t^{1/2}$. Examples in this class are independent random walk
systems, random averaging and reversible simple exclusion processes
starting from a stationary state or even in nonstationary states \cite
{Balazs-R-S,Ferrari-Fontes,Jara-Landim,Kumar,Seppalainen}.

On the other hand, a system is in the KPZ class if its ``height''
function and integrated current have standard deviation of order
$t^{1/3}$, and nontrivial spatial correlations at range $t^{2/3}$. A
well-studied particle system model in this class is the asymmetric
simple exclusion process starting from deterministic initial
configurations such as step profile and alternating conditions, or from
a stationary state (cf.
\cite{BDJ,Baik-Rains,Balazs-Seppalainen1,Balazs-Seppalainen,BCS,BFS,Ferrari-Spohn,Lee,Praehofer-Spohn,Quastel-Valko,TWreview} and references therein).

These two classes can be seen in the study of the famous KPZ stochastic
partial differential equation first mentioned in \cite{KPZ}:
%
%e1.1 #&#
\begin{equation}
\label{KPZ_eqn} \partial_t h_t(x) = D\Delta
h_t(x) + a \bigl(\nabla h_t(x) \bigr)^2 +
\sigma\dot\W_t(x),
\end{equation}
where $\dot\W_t(x)$ is a space--time
white noise with unit variance. When $a=0$ and
$D,\sigma> 0$, then $h_t(x)$ is a generalized Ornstein--Uhlenbeck
process in EW class. However, when
$a\neq0$ and $D, \sigma> 0$, a physical argument indicates
that
$h_t(x)$ is in the KPZ class (cf. \cite{BKS,KPZ}). Also, in
another sense, it has been shown that the ``Cole--Hopf'' solution of the
KPZ equation, starting from certain initial conditions, interpolates
between the two classes when the centered solution is examined in
different asymptotic scaling regimes, that is, when normalized by
$t^{1/3}$ as $t\uparrow\infty$ or when normalized by $t^{1/4}$ as
$t\downarrow0$, nontrivial limits are obtained (cf. \cite{ACQ,BQS}).

Moreover, it is believed that in many ``critical'' weakly asymmetric,
$d=1$ particle systems, that is, when the weak asymmetry is scaled at a
critical level, the diffusively scaled ``height'' function or integrated
current should converge to the solution of the KPZ equation with
parameters depending on the structure of particle interactions and
initial conditions. Recently, much progress has been made in making
clear this convergence. Part of the difficulty is that, since
``solutions'' to the KPZ equation are expected to be
distribution-valued, the nonlinear term in the equation does not make
sense, and so the equation is ill-posed. Hence, what does it mean to
solve the KPZ equation and also, when properly interpreted, how to
derive the KPZ equation from microscopic particle interactions?

One way to approach these questions is to observe that the Cole--Hopf
transformation $z_t(x): = \exp\{(a/D)h_t(x)\}$ linearizes the KPZ
equation to a stochastic heat equation
%
%e1.2 #&#
\begin{equation}
\label{linearization} \partial_t z_t(x) = D\Delta
z_t(x) + (a\sigma/D) z_t(x)\dot\W_t(x),
\end{equation}
which can be solved uniquely starting from a class of initial
conditions and is also strictly positive for times $t>0$ \cite
{Mueller,Walsh}. Then the ``Cole--Hopf'' solution is defined as
$h_t(x):= \log z_t(x)$. In \cite{BG}, starting from near stationary
measures in a certain weakly asymmetric simple exclusion process
observed in diffusive scale, this sentiment was made rigorous. Namely,
it was proved that the microscopic Cole--Hopf transform of the
microscopic height function, using a clever device in~\cite{Gartner}
which linearizes the simple exclusion dynamics to a more manageable
system, converges to the Cole--Hopf transform of the KPZ equation, the
solution to the stochastic
heat equation \eqref{linearization}. More recently, in \cite{ACQ,Sasamoto-Spohn} this notion of solution further gained traction in
that the result in \cite{BG} was nontrivially generalized to step
profile deterministic initial configurations. At the same time, in \cite
{Hairer}, it has been shown that $\log z_t(x)$ is the unique solution
of a well-posed equation on a torus,
derived from a ``rough paths'' approximation of \eqref{KPZ_eqn}, so that
it is clear what sort of KPZ equation the ``Cole--Hopf'' solution
actually solves.

In this article, another approach is considered which allows to
generalize the types of microscopic particle interactions considered,
given that the device in \cite{Gartner} seems limited to simple
exclusion and a few variants such as $q$-TASEP dynamics~\cite
{Borodin-Corwin}. At the microscopic level, the height function
$H_t(x)$, evaluated for $t\geq0$ and $x\in\Z$, takes form
%
%e1.3 #&#
\begin{equation}
\label{current-relation} H_t(x) = \cases{ %
\displaystyle J_0(t) - \sum_{y=0}^{x-1}
\eta_t(y),& \quad$\mbox{for }x\geq1,$
\vspace*{2pt}\cr
J_0(t), &\quad $\mbox{for }x=0,$
\vspace*{2pt}\cr
\displaystyle J_0(t) + \sum_{y=x}^{-1}
\eta_t(y), & \quad$\mbox{for } x\leq-1,$}
\end{equation}
where $J_y(t)$ is the current across bond $(y-1,y)$ and $\eta_t(y)$ is
the particle number at $y\in\Z$ at time $t\geq0$. Then the discrete
gradients of the microscopic height function are the particle numbers,
$H_t(x+1) - H_t(x) = \eta_t(x)$, and the corresponding fluctuation
field examined in diffusive scale, that is, when time is scaled in
terms of $n^2$ and space is scaled by $n$, is the particle density
fluctuation field $\Y^n_t$. The guiding idea is that $\Y^n_t$ should
converge to $\Y_t = \nabla h_t$ in some sense.

Formally, by carrying through the ``$\nabla$'' operation, $\Y_t$
satisfies a type of stochastic Burgers equation,
%
%e1.4 #&#
\begin{equation}
\label{conservative_eqn} \partial_t\Y_t(x) = D\Delta
\Y_t(x) + a\nabla \bigl(\Y_t(x) \bigr)^2 +
\sigma\nabla\dot\W_t(x),
\end{equation}
which again for the same reasons as for the original KPZ equation is
ill-posed when $a\neq0$. If $a=0$, however, it is a type of
Ornstein--Uhlenbeck equation which possesses a unique solution when
starting from a large class of initial distributions (cf. \cite{BC,Walsh}).

A main contribution of the article is to understand the derived
stochastic Burgers equation \eqref{conservative_eqn}
in the context of
a general class of nearest-neighbor weakly asymmetric interacting
particle systems on $\Z$, starting from perturbations of the invariant
measure $\nu_\rho$. This class is composed of systems with ``gradient''
dynamics, not necessarily product invariant measures, sufficient
spectral gap and ``equivalence of ensembles'' estimates among other
technical conditions (cf. Section~\ref{notation}), which include in
particular the already studied simple exclusion process, and also
zero-range and exclusion models with kinetically constrained or
speed-change interactions, which have varying and sometimes slow mixing
behaviors. The initial distributions consist of ``bounded entropy''
perturbations of the invariant measure~$\nu_\rho$ (cf. Section~\ref{notation} for a precise statement).

Our results describe the limit points of the fluctuation field $\Y
^{n,\gamma}_t$ in diffusive scale, in a reference frame moving with a
process characteristic velocity $\upsilon_n(t) \sim n^{-1}\lfloor
n^2(p_n-q_n)\upsilon t\rfloor$. Here, $p_n-q_n$ is the difference of
the single particle jump rates which identifies the strength of the
weak asymmetry considered, and $\upsilon$ is a homogenized velocity
parameter depending on the particle dynamics. Given the size of
$p_n-q_n$, a dichotomy emerges in the form of the limits derived.
Namely, for $p_n-q_n = O(n^{-\gamma})$, when $1/2<\gamma\leq1$, we
show a ``crossover result'' (Theorem~\ref{th crossover}) that $\Y
^{n,\gamma}_t$ converges to an Ornstein--Uhlenbeck field
with certain homogenized parameters. When $\gamma=1$, convergence of
$\Y^{n,\gamma}_t$ to the same Ornstein--Uhlenbeck field has been known
for many particle systems since the work \cite{BR}. For discussions of
``crossover'' results with respect to simple exclusion, see \cite
{Sasamoto-Spohn-crossover,GJ5}.

However, when $\gamma= 1/2$, a critical value, we prove (Theorem~\ref
{th kpz scaling}) that limit points of $\Y^{n,\gamma}_t$
satisfy a martingale formulation, which we dub as an ``energy''
formulation (cf. Theorem~\ref{th kpz scaling}),
also with homogenized constants, which interprets the stochastic
Burgers equation: Namely, the nonlinear term in \eqref
{conservative_eqn} is understood in terms of a certain Cauchy limit of
a function of the fluctuation field acting on an approximation of a
point mass as the approximation becomes more refined.
We remark, however, with respect to simple exclusion processes,
convergence of $\Y^{n,\gamma}_t$ to a unique limit when $\gamma=1/2$ is
already known, and this limit is understood in the ``Cole--Hopf'' sense
as mentioned above \cite{BG}. Therefore, our results imply that the
``Cole--Hopf'' limit of the fluctuation field satisfies also our ``energy''
formulation. In this context, we note \cite{Assing} further clarifies the
``energy'' formulation of the simple exclusion limit starting from the
invariant state $\nu_\rho$
(cf. point $2$ of Remark~\ref{kpz_remark}).
Also, we note another
martingale formulation
was given with respect to the Burgers equation in \cite{Assingmartingale}.

In our general framework, convergence of $\Y^{n,\gamma}_t$ to a unique
limit when $\gamma=1/2$
has not been shown,
an important open question (cf. Remark~\ref{kpz_remark}).
However, one may still try to characterize limit points of the height
function across process characteristics,
$H^{n,\gamma}_t(x):= n^{-1/2}H_{n^2t} (nx-n\upsilon_n(t) )$, via
\eqref{current-relation} given subsequential convergence of $\Y
^{n,\gamma}_t$. Although this is not the purpose of this paper, we
indicate how this might be accomplished to be more complete. Indeed, by
\eqref{current-relation} and $J_0(t) -J_{x}(t) = \sum_{y=0}^{x-1}
(\eta_t(y) - \eta_0(y) )$, one has $H^{n,\gamma}_t(x) =
n^{-1/2}J_{nx-n\upsilon_n(t)}(n^2t) - n^{-1/2}\sum_{y=0}^{nx-n\upsilon
_n(t)-1} \eta_0(y)$, say for $x>\upsilon_n(t)$. To write the current in
terms of the fluctuation field, formally, $n^{-1/2}J_{ nx-n\upsilon
_n(t)}(n^2t) = \Y^{n,\gamma}_t(1_{[x,\infty)}) - \Y^{n,\gamma
}_0(1_{[x,\infty)}) + o(1)$, although as there are an infinite number
of particles and $1_{[x,\infty)}$ is not a compactly supported function
some sort of truncation is needed to make a rigorous argument. Using
the method in \cite{Rost-Vares} and \cite{Jara-Landim}, one can
approximate $n^{-1/2}J_{nx-n\upsilon_n(t)}(n^2t)$ by $\Y^{n,\gamma
}_t(G_{k,x}) - \Y^{n,\gamma}_0(G_{k,x})$ for large $k$ where
$G_{k,x}(z) = (1-(z-x)_+/k)_+$, and so it is possible to take
subsequential limits of $H^{n,\gamma}_t$.

Finally, we remark if uniqueness of solution for the $\gamma=1/2$
``energy''
formulation were known in our more general framework, one would be able
to identify the solution, modulo parameters, as the limit already
identified for simple exclusion through the Cole--Hopf apparatus. In
this way, one should be able to determine that the height function
limits, with respect to a general class of interactions starting from
nearly the invariant measure, are in the KPZ class for instance.

We now remark on the argument for Theorems \ref{th crossover} and \ref
{th kpz scaling}. We take a stochastic differential of $\Y^{n,\gamma
}_t$, namely
\[
d\Y^{n,\gamma}_t = \bigl(\partial_t
\Y^{n,\gamma}_t + L_n \Y ^{n,\gamma}_t
\bigr) \,dt + d\M^{n,\gamma}_t,
\]
where $L_n$ is the system infinitesimal generator and $\M^{n,\gamma}_t$
is a martingale. We note, because the reference frame moves with
velocity $\upsilon_n(t)$, the term $\partial_t \Y^{n,\gamma}_t$ does
not vanish. Beginning in a perturbed invariant measure, the martingale
term can be handled by an ergodic theorem. However, to write the drift term
$\partial_t \Y^{n,\gamma}_t + L_n \Y^{n,\gamma}_t$, in terms of the
fluctuation field itself and, therefore, ``close'' the equation, is a
more difficult task, and requires what has been known as a
``Boltzmann--Gibbs'' principle. Such a principle was first proved in \cite
{BR} when $\gamma=1$. In our context, we would like to recover a
second-order term, and the principle would replace
\[
\int_0^t \frac{1}{n^{\gamma- 1/2}}\sum
_{x\in\Z} \nabla G(x/n)\tau_x V(
\eta_{n^2s}) \,ds
\]
with
\begin{eqnarray*}
&&\frac{\varphi''_{V}(\rho)}{2}\int_0^t
\frac{1}{n^{\gamma+ 1/2}}\sum_{x\in\Z} \nabla G(x/n)
\\
&&\hspace*{23pt}\qquad{} \times \biggl\{\Y^{n,\gamma}_s \biggl(\frac
{1}{2\varepsilon}1_{[x-\varepsilon,x+\varepsilon]}
\biggr)^2 - \E_{\nu_\rho} \biggl[\Y^{n,\gamma}_s
\biggl(\frac{1}{2\varepsilon
}1_{[x-\varepsilon,x+\varepsilon]} \biggr)^2 \biggr] \biggr
\}\,ds
\end{eqnarray*}
in $L^2(\mathbb{ P}_{\nu_\rho})$ as $n\uparrow\infty$ and $\varepsilon
\downarrow0$.
Here, $G$ is a function in the Schwarz class, $\tau_x$ is the $x$-shift
operator, $V$ is a mean-zero function
with the property that the derivative of its ``tilted mean'' $\varphi
_{V}(z)$ vanishes at $z=\rho$ [cf. definition near \eqref
{varphi_derivatives}]. Given such a replacement principle (cf.
Section~\ref{BG_statement} for precise statements), one can prove
the sequence $\Y^{n,\gamma}_t$ is tight and derive martingale
formulations of limit points as desired.

The case $\gamma=1/2$ is the most difficult since there is no spatial
averaging at all. However, there is much cancelation with respect to
the time integral which helps to prove the estimate needed. We show the
cases $1/2<\gamma\leq1$ would follow from the $\gamma= 1/2$
replacement. A similar replacement for symmetric simple exclusion,
using specific duality methods, was performed in \cite{Assingreplacement}.

The method given here, in our general framework, is quite different.
The main idea is to use an involved $H_{-1}$ renormalization scheme to
bound errors in the replacement. Such a scheme makes good use of three
assumed ingredients [cf. precise statements (R), (G), (EE) in
Section~\ref{notation}]: First, the measure $\nu_\rho$ is invariant
with respect to all asymmetric and symmetric versions of the process,
the main reason for the ``gradient dynamics'' condition. Second, a
spectral gap lower bound for the symmetric process localized on a
interval $\Lambda_\ell$ with width $\ell$ and $\sum_{x\in\Lambda_\ell
}\eta(x)$ particles which, after averaging with respect to $\nu_\rho$,
is of order $O(\ell^{-2})$. Also, third, an ``equivalence of ensembles''
estimate holds with respect to canonical $\nu_\rho(\cdot| \sum_{|x|\leq\ell}\eta(x) = k)$ and grand canonical $\nu_\rho$ measures.

We note the current article is an evolution of the arXiv paper \cite
{GJ6}, encompassing the work there on a type of exclusion model
starting from a
Bernoulli product invariant measure and a model specific Boltzmann--Gibbs
principle. See also \cite{Assing} for a different type of resolvent
method specific to simple exclusion.
In this context, the current article is a nontrivial generalization to
more diverse
models, starting from perturbations of the stationary state, using a
more general $H_{-1}$ renormalization scheme. We remark that part of
this improvement, of its own interest, is that the
Boltzmann--Gibbs principle (Theorem~\ref{gbg_L2}) shown here does not
rely on
the independence structure of a product measure, or on a sharp spectral
gap estimate, or on a process ``duality.''
Finally, we note
elements of our $H_{-1}$ renormalization scheme go back to \cite{G.}
and \cite{SX} in different contexts.

We now give the structure of the article. In Section~\ref{notation-results}, the general class of models studied, results and
specific systems satisfying the class assumptions are discussed. Then,
in Section~\ref{proofs}, we outline the proof of the main results,
Theorems \ref{th crossover} and \ref{th kpz scaling}, stating the form
of Boltzmann--Gibbs principle used. In Section~\ref{BG}, this principle
is proved. Finally, in Section~\ref{EE_section}, we prove for a class
of systems, including the specific processes discussed in Section~\ref{notation-results}, the ``equivalence of ensembles'' estimate assumed for
the proofs in Section~\ref{proofs}.

%s2 #&#
\section{Abstract framework,
results and models}
\label{notation-results}
We now discuss the abstract framework we work with in Section~\ref{notation}, and state results in this framework in Section~\ref{results}. This framework covers a wide class of models such as
zero-range models and different types of exclusion processes which we
detail in Sections~\ref{zero-range}--\ref{speed_change_model}. A
reader focusing on one of these models, might skip to its subsection
while referring to Section~\ref{notation}, and then proceed to
results in Section~\ref{results}.

%s2.1 #&#
\subsection{Notation and assumptions}
\label{notation}

We consider a sequence of ``weakly asymmetric'' nearest-neighbor ``mass
conservative'' particle systems $\{\eta^n_t\dvtx t\geq0\}$ on the state
space $\Omega= \N_0^\Z$ where $\N_0=\{0,1,2,\ldots\}$. The
configuration of the system $\eta_t = \{\eta_t(x)\dvtx x\in\Z\}$ is a
collection of occupation numbers $\eta_t(x)$ which counts the numbers
of particles at sites $x\in\Z$ at time $t\geq0$. In some of the
examples we will consider, the occupation number is bounded by $1$, in
which case
the effective state space reduces to $\{0,1\}^\Z$.

\textit{``Gradient'' dynamics.}
The dynamics will be of ``gradient'' type. That is, we suppose there are
functions $\{b^{R,n}_x\}_{n\geq1}$ and $\{b^{L,n}_x\}_{n\geq1}$
satisfying the following conditions (R1) and (R2). Let $\tau_x$ be the
shift operator where $(\tau_x\eta)(z) = \eta(x+z)$ and $\tau_x f(\eta)
= f(\tau_x\eta)$ for $x\in\Z$. Let also $\Lambda_k = \{j\dvtx |j|\leq k\}
\subset\Z$ for $k\geq1$.
\begin{longlist}[(R1)]
\item[(R1)] For all $n\geq1$, $b^{R,n}_x=\tau_x b_0^{R,n}$ and
$b^{L,n}_x=\tau_x b^{L,n}_0$ are nonnegative finite-range functions on
$\Omega$ such that $b_0^{R,n}$ and $b_0^{L,n}$ are supported on $\{\eta
(y)\dvtx y\in\Lambda_R\}$ for some $R\geq1$. We suppose uniformly in $n$ that
$|b_0^{R,n}(\eta)| + |b_0^{L,n}(\eta)| \leq C\sum_{y\in\Lambda_R}\eta
(y)$. Moreover, there are nonnegative functions $c^n_x = \tau_x c^n$ on
$\Omega$, supported on $\{\eta(x)\dvtx x\in\Lambda_R\}$ such that
\[
b^{R,n}_x(\eta) - b^{L,n}_x(\eta)
= c_x^n(\eta) -c_{x+1}^n(\eta).
\]
In addition, suppose there are fixed functions $b_0^R$, $b_0^L$ and
$c_0$ such that configuration-wise
\[
\lim_{n\uparrow\infty} b_0^{R,n}(\eta) =
b_0^R(\eta), \qquad\lim_{n\uparrow\infty}
b_0^{L,n}(\eta)= b_0^L(\eta)\quad \mbox{and} \quad \lim_{n\uparrow\infty} c_0^n(\eta)=
c_0(\eta).
\]
\end{longlist}

In some of the models considered, such as zero-range processes in
Section~\ref{zero-range}, the functions $b_0^{R,n} = b_0^R$,
$b_0^{L,n}=b_0^L$ and $c_0^n=c_0$ are fixed and do not depend on the
parameter $n$. However, for the kinetically constrained exclusion
models in Section~\ref{porous}, the rates do depend on $n$.

\begin{longlist}[(R2)]
\item[(R2)] With respect to a fixed measure $\nu_\rho$ on $\Omega$,
for all $n\geq1$, we have
\[
b_x^{R,n}\bigl(\eta^{x+1,x}\bigr)
\frac{d\nu_\rho^{x+1,x}}{d\nu_\rho}(\eta) = b_x^{L,n}(\eta),
\]
where $\nu_\rho^{x+1,x}$ is the measure of the variable $\zeta= \eta
^{x+1,x}$ under $\nu_\rho$.
\end{longlist}
We also define $b^n_x(\eta) = b_x^{R,n}(\eta) + b^{L,n}_x(\eta)$,
$b^n(\eta) = b^n_0(\eta)$ and $c^n(\eta) = c^n_0(\eta)$ to simplify notation.

We now specify the process generator. For $a\in\R$ and $\gamma>0$, let
\[
p_n = \frac{1}{2} + \frac{a}{2n^\gamma} \quad\mbox{and}\quad
q_n = 1-p_n = \frac{1}{2}- \frac{a}{2n^\gamma}.
\]
Let also $n_0$ be such that $0\leq p_{n_0},q_{n_0}\leq1$, and $T>0$ be
a fixed time.

\begin{longlist}[(M)]
\item[(M)] Suppose, for each $a\in\R$, $\gamma>0$ and $n\geq n_0$,
that $\{\eta^n_t\dvtx t\in[0,T]\}$ is a $L^2(\nu_\rho)$ Markov process
with strongly continuous Markov semigroup $P_t^n$ and Markov generator
$L_n$ (cf. Chapter I; Section IV.4 of \cite{Liggett}) with a core
composed of local $L^2(\nu_\rho)$ functions on which
\end{longlist}
%
%e2.1 #&#
\begin{equation}
\label{generator} L_n f(\eta) = n^2\sum
_{x\in\Z} \bigl\{b^{R,n}_x(
\eta)p_n \nabla _{x,x+1}f(\eta) + b^{L,n}_x(
\eta)q_n\nabla_{x+1,x}f(\eta) \bigr\},
\end{equation}
where
$\nabla_{x, y}f(\eta) = f(\eta^{x,y})-f(\eta)$, and
$\eta^{x,y}$ is the configuration obtained from $\eta$ by moving a
particle from $x$ to $y$:
\[
\eta^{x,y}(z) = \cases{ %
\eta(y)+1, & \quad$\mbox{when } z=y,$
\vspace*{2pt}\cr
\eta(x)-1,& \quad$\mbox{when } z=x,$
\vspace*{2pt}\cr
\eta(z), & \quad$\mbox{otherwise.}$}
\]
The role of $a\in\R$ and $\gamma>0$ is to control the strength of the
``weak asymmetry'' in the model.

\textit{Invariant measure $\nu_\rho$.}
We now specify some technical properties which $\nu_\rho$ should
satisfy. Define for a probability measure $\kappa$, the path measure $\P
_\kappa$ governing the process $\{\eta^n_t\dvtx t\in[0,T]\}$ with initial
configurations $\eta_0$ distributed according to~$\kappa$. Let then
$E_\kappa$ and $\E_\kappa$ denote expectations with respect to $\kappa$
and $\P_\kappa$, respectively.

\begin{longlist}[(IM1)]
\item[(IM1)]
Suppose $\nu_\rho$ is a translation-invariant measure which is
``spatially mixing.'' That is, for local $L^2(\nu_\rho)$ functions $f$
and $h$,
\[
\lim_{|x|\uparrow\infty}E_{\nu_{\rho}}\bigl[f(\eta)\tau_x
h(\eta)\bigr] = E_{\nu_\rho}[f]E_{\nu_\rho}[h].
\]
In addition, suppose the mean $E_{\nu_\rho}[\eta(0)]=\rho$, and
moment-generating function $E_{\nu_\rho}[e^{\lambda\eta(0)}]<\infty$
for $|\lambda|\leq\lambda^*$ for a $\lambda^*>0$.
\end{longlist}
Although product measures $\nu_\rho$ are considered in most of the
examples, we note, in Section~\ref{speed_change_model}, a nonproduct
measure $\nu_\rho$ corresponding to an exponentially mixing ergodic
Markov chain is used.

Now, the measure $\nu_\rho$, by (IM1) and the ``gradient dynamics''
conditions (R1) and (R2), is an invariant measure with respect to $L_n$
for all $a\in\R$ and $\gamma> 0$. Indeed, let $\phi$ be a local
$L^2(\nu_\rho)$ function supported with respect to sites in $\Lambda
_k$. Then, for $\ell>k$, we have
\begin{eqnarray*}
E_{\nu_\rho}[L_n\phi] & = & -n^2E_{\nu_\rho}
\biggl[\sum_{|x|\leq\ell} (p_n-q_n)
\phi(\eta) \bigl[c^n_x(\eta)-c^n_{x+1}(
\eta) \bigr] \biggr]
\\
&=& -n^2(p_n-q_n)E_{\nu_\rho} \bigl[
\phi(\eta) \bigl(c^n_{-\ell}(\eta) - c^n_{\ell+1}(
\eta) \bigr) \bigr].
\end{eqnarray*}
The limit as $\ell\uparrow\infty$ vanishes, by translation-invariance
and the spatial mixing assumption in (IM1).

One can also compute that the $L^2(\nu_\rho)$ adjoint $L^*_n$ is the
generator with parameter $-a$, that is, when the jump probability is
reversed. Define $S_n = (L_n + L^*_n)/2$. Then the Dirichlet form
$D_{\nu_\rho,n}(f):= E_{\nu_\rho}[f(-L_n f)] =\break E_{\nu_\rho}[f(-S_n
f)]$ on local $L^2(\nu_\rho)$ functions is given by
%
%e2.2 #&#
\begin{equation}
\label{dirichlet_form} D_{\nu_\rho}(f) = \frac{1}{2}\sum
_{x\in\Z} E_{\nu_\rho} \bigl[ b^{R,n}_x(
\eta) \bigl(\nabla_{x,x+1}f(\eta) \bigr)^2 \bigr].
\end{equation}
Moreover, when $a=0$, $S_n$ is the generator of the associated process
and $\nu_\rho$ is a reversible measure.

Consider now the empirical measure
\[
\Y^n_0 = \frac{1}{\sqrt{n}}\sum
_{x\in\Z}\bigl(\eta(x)-\rho\bigr)\delta_{x/n}
\]
and its covariance under measure $\kappa$, on compactly supported functions,
\[
\C^n_\kappa(G,H) = E_{\kappa} \bigl[ \bigl(
\Y_0^n(G) - E_\kappa \bigl[\Y
_0^n(G) \bigr] \bigr) \bigl(\Y_0^n(H)
- E_\kappa \bigl[\Y_0^n(H) \bigr] \bigr)
\bigr].
\]

\begin{longlist}[(IM2)]
\item[(IM2)] We assume, starting from $\nu_\rho$, that $\Y^n_0$
converges weakly to a spatial Gaussian process with covariance
$\C_{\nu_\rho}(G,H):= \lim_{n\uparrow\infty} \C^n_{\nu_\rho}(G,H)$
such that $\C_{\nu_\rho}(G,G) \leq C(\rho)\|G\|^2_{L^2(\R)}$.
Also, suppose the moment bound holds
$\sup_{\ell\geq1} \| (\frac{1}{\sqrt{\ell}}\sum_{x=1}^\ell(\eta
(x)-\rho) )^2 \|_{L^4(\nu_\rho)} < \infty$.
\end{longlist}
It will be convenient to define the variances
\[
\sigma^2_n(\rho):= \C^n_{\nu_\rho}(H,H)
= E_{\nu_\rho} \biggl[ \biggl(\frac{1}{\sqrt{2n+1}}\sum
_{x\in\Lambda_n} \bigl(\eta(x)-\rho\bigr) \biggr)^2 \biggr]
\]
and $\sigma^2(\rho) = \C_{\nu_\rho}(H,H) = \lim_{n\uparrow\infty}\sigma
^2_n(\rho)$ when $H(x) = 1_{[-1,1]}(x)$.

When $\nu_\rho$ is sufficiently mixing, the case of our examples, (IM2)
holds with $\C_{\nu_\rho}(G,H) = \sigma^2(\rho)\langle G, H\rangle
_{L^2(\R)}$.

Now, for $\lambda\in(-\lambda^*, \lambda^*)$, consider the tilted
measure $\nu^\lambda_\rho$ with ``tilt'' or ``chemical potential'' $\lambda
$ given by its finite-dimensional projections
%
%e2.3 #&#
\begin{equation}
\label{tilted_measure} \qquad\frac{d\nu^\lambda_\rho}{d\nu_\rho} \bigl(\eta(x) = e(x), x\in\Lambda
_\ell | \eta(x) = \xi(x), x\notin\Lambda_\ell \bigr) =
\frac
{e^{\lambda\sum_{x\in\Lambda_\ell}(e(x) - \rho)}}{Z(\lambda,\ell,\xi)},
\end{equation}
where $e,\xi\in\Omega$ and $Z(\lambda,\ell,\xi)$ is the normalization.

\begin{longlist}[(D1)]
\item[(D1)] We will assume the measures $\{\nu^\lambda_\rho\dvtx  |\lambda
|<\lambda^*\}$ are well defined on $\Omega$, that is a limit of \eqref
{tilted_measure} as $\Lambda_\ell\nearrow\Z$ can be taken, not
depending on $\xi$. Also, we assume that the measures can be indexed by
density, that is, $E_{\nu_\rho^\lambda}[\eta(0)]$ is strictly
increasing in $\lambda$ for $|\lambda|\leq\lambda^*$.
\end{longlist}
These assumptions hold when $\nu_\rho$ is a nontrivial product measure
satisfying (IM1): $\frac{d}{d\lambda}E_{\nu_\rho^\lambda}[\eta(0)] =
E_{\nu_\rho^\lambda} [ (\eta(0)-E_{\nu_\rho^\lambda}[\eta(0)]
)^2 ]>0$. They also hold when $\nu_\rho$
corresponds to the ergodic Markov chain in the case for the exclusion
with speed-change model (cf. details in Section~\ref{speed_change_model}).

The measures $\{\nu^\lambda_\rho\dvtx  |\lambda|<\lambda^*\}$ are
translation-invariant since $\nu_\rho$ is assumed translation-invariant
(IM1). Also, given exponential moments of $\nu_\rho$ (IM1), $E_{\nu
^\lambda_\rho}[\eta(0)]$ is continuous in $\lambda$ for $|\lambda
|<\lambda^*$. Hence, by the strict increasing assumption in (D1), one
can reparameterize $\{\nu^\lambda_\rho\}$ in terms of density: Let
$z\in(\rho_*, \rho^*)$ where $\rho_* = \lim_{\lambda\downarrow-\lambda
^*}E_{\nu^\lambda_\rho}[\eta(0)]$ and $\rho^* = \lim_{\lambda\uparrow
\lambda^*}E_{\nu^\lambda_\rho}[\eta(0)]$. Let $\lambda(z)\in(-\lambda
^*, \lambda^*)$ be the parameter such that $E_{\nu^{\lambda(z)}_\rho
}[\eta(0)]=z$. Then we will
define\vspace*{-1pt} $\nu_z = \nu^{\lambda(z)}_\rho$.\vspace*{1pt}

Define also, for a local $L^2(\nu_\rho)$ function $f$, the ``tilted
mean'' function
$\varphi_f(z) \dvtx  (\rho_*,\rho^*) \rightarrow \R$ where
\[
\varphi_f(z) = E_{\nu_z} \bigl[f(\eta) \bigr].
\]

We consider the derivatives of $\varphi_f(z)$ as the formal limits of
the derivatives of $E_{\nu_z}[f(\eta)|\eta(x) = \xi(x), x\in\Lambda
_\ell]$ as $\ell\uparrow\infty$. Define
\[
\varphi'_f(z) :=  \lambda'(z)E_{\nu_z}
\biggl[\bigl(f(\eta)-E_{\nu_z}[f]\bigr) \biggl(\sum
_{x\in\Z}\bigl(\eta(x) - z\bigr)\biggr)\biggr],
\]
%
%%
%%e2.4 #&#
\begin{eqnarray}
\label{varphi_derivatives} \varphi''_f(z) &:=&
\bigl(\lambda'(z)\bigr)^2 E_{\nu_z}\biggl[
\bigl(f(\eta)-E_{\nu
_z}[f]\bigr) \biggl(\sum_{x\in\Z}
\bigl(\eta(x) - z\bigr)\biggr)^2\biggr]
\nonumber
\\[-8pt]
\\[-8pt]
\nonumber
&& {}+\lambda''(z)E_{\nu_z}\biggl[\bigl(f(
\eta)-E_{\nu_z}[f]\bigr) \biggl(\sum_{x\in\Z
}
\bigl(\eta(x) - z\bigr)\biggr)\biggr].
\end{eqnarray}
For the $0$th derivative, we set $\varphi^{(0)}_f(z):= E_{\nu_z}[f]$.

\begin{longlist}[(D2)]
\item[(D2)] For local $L^2(\nu_\rho)$ functions $f$, suppose the limits
\eqref{varphi_derivatives} are well defined and $|\varphi'_f(\rho)|,
|\varphi''_f(\rho)| \leq C(\rho)\|f\|_{L^2(\nu_\rho)}$; already,
$|\varphi_f(\rho)|\leq\|f\|_{L^2(\nu_\rho)}$. Also, suppose
\[
\lim_{n\uparrow\infty} \varphi'_{f_n}(\rho) =
\varphi'_f(\rho) \quad\mbox{and}\quad \lim_{n\uparrow\infty}
\varphi''_{f_n}(\rho) =
\varphi''_f(\rho)
\]
when $\{f_n\}$ and $f$ are local functions such that $\lim_{n\uparrow
\infty} f_n(\eta) = f(\eta)$ and $f_n(\eta) \leq\hat f(\eta)$
configuration-wise for each $n$ where $\hat f\in L^2(\nu_\rho)$.
\end{longlist}
When $\{\nu_x\}$ are product or rapidly mixing Markov measures, again
the case for our examples, this condition also holds by calculation
with \eqref{tilted_measure}.

\textit{Spectral gap}. We now give a ``spectral gap'' condition.
For $\ell\geq1$, recall $\Lambda_\ell$ is the box of size $2\ell+1$,
namely $\Lambda_\ell:=\{x\in\Z\dvtx |x|\leq\ell\}$. Let also, for $k\geq
0$ and $\xi\in\Omega$, $\G_{k,\ell, \xi}:= \{ \eta\dvtx  \sum_{x\in\Lambda
_\ell} \eta(x) = k, \eta(y)=\xi(y)  \mbox{ for } y\notin\Lambda_\ell
\}$ be the
hyperplane of configurations on $\Lambda_\ell$ with $k$ particles which
equal $\xi$ outside $\Lambda_\ell$. Denote by $\nu_{k,\ell, \xi}$ the
canonical measure on $\G_{k,\ell, \xi}$, namely
\[
\nu_{k,\ell, \xi}(\cdot):= \nu_{\rho} \biggl(\cdot\Big|\sum
_{x\in{\Lambda
_\ell}}\eta(x)=k, \eta(y) = \xi(y) \mbox{ for }y\notin
\Lambda_\ell \biggr).
\]

Consider now the process, restricted to the hyperplane $\G_{k,\ell,\xi
}$ with generator
\[
\mathcal{S}_{n,\G_{k,\ell, \xi}}f(\eta) = \frac{1}{2}\mathop{\sum
_{|x-y|=1}}_{x,y\in\Lambda_{\ell}} b^n_x(\eta)
\nabla_{x,y} f(\eta).
\]
This is a finite-state Markov process with reversible invariant measure
$\nu_{k,\ell, \xi}$. Denote by $\lambda_{k,\ell,\xi,n}$ the spectral
gap, that is the second largest eigenvalue of $-\mathcal{S}_{n,\G
_{k,\ell,\xi}}$ (with $0$ being the largest).
Let $W(k,\ell,\xi,n)$ denote the reciprocal of $\lambda_{k,\ell,\xi
,n}$, which is set to $\infty$ if $\lambda_{k,\ell,\xi,n} = 0$. Then
the associated
Poincar\'e-inequality reads as
%
%e2.5 #&#
\begin{equation}
\label{poincare inequality} \operatorname{Var}(f,\nu_{k,\ell, \xi}) \leq {W(k,\ell, \xi,n)\mathcal
{D}_n(f,\nu_{k,\ell, \xi})},
\end{equation}
where $\operatorname{Var}(f,\nu_{k,\ell, \xi})$ is the variance of $f$ with
respect to $\nu_{k,\ell, \xi}$ and the canonical Dirichlet form
$\mathcal{D}_n(f,\nu_{k,\ell, \xi})$ is given by
\[
\mathcal{D}_n(f,\nu_{k,\ell, \xi}):= \frac{1}{2}\sum
_{x,x+1\in
\Lambda_\ell} E_{\nu_{k,\ell, \xi}} \bigl[ b^{R,n}_x(
\eta) \bigl(\nabla _{x,x+1}f(\eta) \bigr)^2 \bigr].
\]
When $W(k,\ell,\xi,n)<\infty$, the process is ergodic and $\nu_{k,\ell
,\xi}$ is the unique invariant measure.

Denote the ``outside variables'' by $\eta^c_\ell= \{\eta(x)\dvtx x\notin
\Lambda_\ell\}$. We will assume the following condition on $W(k,\ell,
\xi,n)$.
\begin{longlist}[(G)]
\item[(G)] Suppose there is a constant $C=C(\rho)$ such that, for
$n\geq1$, we have
\[
E_{\nu_{\rho}} \biggl[W \biggl(\sum_{x\in\Lambda_\ell}
\eta(x),\ell, \eta ^c_\ell,n \biggr)^2 \biggr]
\leq {C\ell^4}.
\]
\end{longlist}

We remark a sufficient condition to verify (G) would be the uniform
bound $\sup_{k,\xi,n} \ell^{-2}W(k,\ell, \xi, n) <\infty$, which holds
for some types but not all of the specific models discussed.

\textit{Equivalence of ensembles.} We will also assume an ``equivalence of
ensembles'' estimate between the canonical and grand-canonical measures.
Define, for $\ell\geq1$ and $\eta\in\Omega$, the empirical average
\[
\eta^{(\ell)} = \frac{1}{2\ell+1}\sum_{y\in\Lambda_\ell}
\eta(y).
\]

\begin{longlist}[(EE)]
\item[(EE)] For local $L^5(\nu_\rho)$ functions $f$, supported on $\{
\eta(x)\dvtx x\in\Lambda_{\ell_0}\}$, such that $\varphi_f(\rho)=\varphi
'_f(\rho)=0$, and $\ell\geq\ell_0$, there exist constants $\alpha_0>0$
and $C=
C(\rho,\ell_0,\alpha_0)$ where
\[
\biggl\| E_{\nu_\rho} \bigl[f|\eta^{(\ell)}, \eta^c_\ell
\bigr] - \frac{\varphi_f''(\rho)}{2} \biggl[\bigl(\eta^{(\ell)} -\rho
\bigr)^2 - \frac{\sigma
^2_\ell(\rho)}{2\ell+1} \biggr] \biggr\|_{L^4(\nu_\rho)} \leq
\frac{C\|f\|_{L^5(\nu_\rho)}}{\ell^{1+\alpha_0/2}}.
\]

On the other hand, when only $\varphi_f(\rho)=0$ is known,
\[
\bigl\| E_{\nu_\rho} \bigl[f|\eta^{(\ell)}, \eta^c_\ell
\bigr] - \varphi_f'(\rho) \bigl(\eta^{(\ell)} -
\rho \bigr)\bigr \|_{L^4(\nu_\rho)} \leq \frac{C\|f\|_{L^5(\nu_\rho)}}{\ell^{1/2+\alpha_0/2}}.
\]
\end{longlist}

We remark, a weaker version, where the $L^2(\nu_\rho)$ norm, instead
of the $L^4(\nu_\rho)$ norm of the difference, is say less than the
same right-hand side expressions with $\|f\|_{L^3(\nu_\rho)}$ in place
of $\|f\|_{L^5(\nu_\rho)}$ would be sufficient for our purposes if
there is a uniform bound on the inverse gap: $\sup_{k,\xi,n}\ell
^{-2}W(k,\ell,\xi,n)<\infty$.

Usually, such estimates follow from a local central limit theorem. In
Proposition~\ref{2EE}, we show, when $\nu_\rho$ is a nondegenerate
product measure, that (EE) holds with $\alpha_0 =1$. In Proposition~\ref
{Markov_EE}, with respect to a Markovian measure, we prove (EE) holds
with $\alpha_0 = 1-\varepsilon$ for any fixed $0<\varepsilon<1$. These
two propositions cover the examples discussed in the article.

\textit{Initial conditions.} We will start from initial measures $\{\mu^n\}
$ which have bounded relative entropy $H(\mu^n;\nu_\rho)$ with respect
to $\nu_\rho$.
\begin{longlist}[(BE)]
\item[(BE)] Suppose $\{\mu^n\}$ satisfies
\[
\sup_n H\bigl(\mu^n;\nu_\rho\bigr) =
\sup_n E_{\nu_\rho} \biggl[\frac{d\mu
^n}{d\nu_\rho}\log
\frac{d\mu^n}{d\nu_\rho} \biggr] < \infty.
\]
\end{longlist}

In addition, we presume a diffusive initial limit starting from $\{\mu
^n\}$.

\begin{longlist}[(CLT)]
\item[(CLT)] Under initial measures $\{\mu^n\}$, we suppose $\Y^n_0$
converges weakly to a spatial Gaussian process $\bar\Y_0$ with
covariance $\C(G,H) = \lim_{n\uparrow\infty} \C^n_{\mu^n}(G,H)$ for
compactly supported functions $G,H$.
\end{longlist}

Of course, if $\mu^n\equiv\nu_\rho$, (BE) and (CLT) trivially hold
with $\C(G,H) = \C_{\nu_\rho}(G,H)$.
When $\nu_\rho$ is a product measure, a possible way to get nontrivial
examples of measures $\{\mu^n\}$ satisfying (BE) and (CLT) is the
following. For simplicity, we consider the case on which $\nu_\rho$ is
a Bernoulli product measure on $\{0,1\}^{\bb Z}$. Let $\{\kappa_x^n\dvtx x
\in\bb Z\}$ be a given bounded sequence and define $\mu^n$ as the
nonhomogeneous Bernoulli product measure satisfying
\[
\mu_n\bigl(\eta(x)=1\bigr) = \rho+ \frac{\kappa_x^n}{\sqrt n}.
\]
A simple computation shows that
\[
H\bigl(\mu^n; \nu_\rho\bigr) \leq \frac{C(\|\kappa\|_{\ell^\infty})}{n} \sum
_{x \in\bb Z} \bigl(\kappa_x^n
\bigr)^2.
\]
Therefore, taking $\kappa_x^n = \kappa(x/n)$, where $\kappa\dvtx \bb R \to
\bb R$ is bounded and in $L^2(\bb R)$, we see that
$
\sup_n H(\mu^n; \nu_\rho)<\infty$,
and (BE) is satisfied. On the other hand, since the measure $\mu^n$ is
product, a simple computation shows that, under $\{\mu^n\}$, the
process $\mc Y_0^n$ converges in distribution to $\bar{\mc Y}_0 +\kappa
$, where $\bar{\mc Y}_0$ is a white noise with variance $\rho(1-\rho)$.
In \cite{Quastel-Remenik}, the Cole--Hopf solution of KPZ is considered
starting from such initial conditions.

One may relate probabilities of events $A$ under $\mu^n$ with those
under $\nu_\rho$ by an application of the entropy inequality:
%
%e2.6 #&#
\begin{equation}
\label{relative_entropy} \P_{\mu^n}(A) \leq
\frac{\log2 + H(\mu^n;\nu_\rho)}{\log ( 1 + {\mathbb P}_{\nu_\rho}(A)^{-1} )   }.
\end{equation}
For instance, let $r\in L^2(\nu_\rho)$ be a local function. By the
spatial mixing assumption~(IM2), under $\nu_\rho$, we have the
convergence in probability,
%
%e2.7 #&#
\begin{equation}
\label{baby_hydrodynamics} \lim_{n\rightarrow\infty} \int_0^T
\frac{1}{2n+1}\sum_{x\in\Lambda_n} \tau_x r
\bigl(\eta^n_s\bigr)\,ds = E_{\nu_\rho}\bigl[r(\eta)
\bigr].
\end{equation}
Then, by the entropy relation, also under $\{\mu^n\}$, the same limit
also holds in probability.

Of course, given that we begin from nearly the invariant measure $\nu
_\rho$, \eqref{baby_hydrodynamics} is a trivial case of
``hydrodynamics.'' Formally, starting from more general measures, the
hydrodynamic equation for the limiting empirical density $\rho=\rho
(x,t)$ would read
%
%e2.8 #&#
\begin{equation}
\label{hyd_equation} \partial_t \rho(x,t) + \frac{a}{2}\nabla
\varphi_b \bigl(\rho(x,t) \bigr) = \frac{1}{2}\Delta
\varphi_c \bigl(\rho(x,t) \bigr).
\end{equation}

In a sense, the main results of the paper are on the different
fluctuations from the law of large numbers \eqref{baby_hydrodynamics}
which arise for different regimes of the strength asymmetry parameters
$a$ and $\gamma$.

%s2.2 #&#
\subsection{Results}
\label{results}

Denote by $\S(\R)$ the standard Schwarz space of rapidly decreasing
functions equipped with the usual metric, and let $\S'(\mathbb{R})$ be
its dual, namely the set of tempered distributions in $\mathbb{R}$,
endowed with the strong topology.
Denote the density fluctuation field acting on functions $H\in\S
(\mathbb{R})$ as
\[
\mathcal{Y}_{t}^{n}(H) =\frac{1}{\sqrt{n}}\sum
_{x\in{\mathbb{Z}}}H \biggl(\frac{x}{n} \biggr) \bigl(\eta
^n_{t}(x)-\rho\bigr). \label{eq:densfieldinz}
\]
Denote by $D([0,T],\S'(\mathbb{R}))$ and $C([0,T],\S'(\mathbb{R}))$ the
spaces of right continuous functions with left limits and continuous
functions respectively from $[0,T]$ to $\S'(\mathbb{R})$.

We now state a result from the literature which has been proved for
some processes (cf. \cite{Ferrarizrp}, Chapter~11 in \cite{KL} for
zero-range processes with bounded rate, \cite{Ravishankar,Dittrich} for simple exclusion processes, and Section II.2.10 of \cite
{Spohn} for exclusion systems with speed-change), sometimes from more
general initial conditions, when the asymmetry is of order $O(n^{-1})$.

%pr2.1 #&#
\begin{proposition}
\label{OU_oldprop}
For $\gamma=1$, starting from $\{\mu^n\}$, the sequence $\{\mathcal
{Y}_t^{n}; n\geq{1}\}$ converges in the uniform topology on $D([0,T],
\mathcal{S}'(\mathbb{R}))$ to the process $\mathcal{Y}_t$ which solves
the Ornstein--Uhlenbeck equation
%
%e2.9 #&#
\begin{equation}
\label{OU1} \partial_t \mathcal{Y}_t =
\frac{1}{2} \varphi_c'(\rho) \Delta
\mathcal{Y}_t + \frac{a}{2}\varphi_b'(
\rho) \nabla\mathcal{Y}_t +\sqrt{\frac{1}{2}
\varphi_b(\rho)} \nabla \dot\W_t,
\end{equation}
where $\dot\W_t$ is a space--time white noise with unit variance, and $\Y
_0=\bar\Y_0$, the field given in \textup{(CLT)}.
\end{proposition}

The Ornstein--Uhlenbeck equation \eqref{OU1} has a drift term coming
from the weak asymmetry of the jump rates. The drift, as is well known,
can be understood in terms of a characteristic velocity $\upsilon=
(a/2)\varphi'_b(\rho)$ from considering the linearization of the
hydrodynamic equation \eqref{hyd_equation} (cf. Chapter II.2 of \cite
{Spohn}). However, it can be removed from the limit field by observing
the density fluctuation field in the frame of an observer moving along
the process characteristics. Define
\[
\mathcal{Y}_{t}^{n,\gamma}(H) =\frac{1}{\sqrt{n}}\sum
_{x\in{\mathbb{Z}}}H \biggl(\frac{x}{n}-\frac
{1}{n} \biggl\{
\frac{a\varphi_{b^n}'(\rho)tn^2}{2n^\gamma} \biggr\} \biggr) \bigl(\eta ^n_{t}(x)-
\rho\bigr). \label{eq:densfieldinzgamma}
\]
If $\gamma=1$, Proposition~\ref{OU_oldprop} is equivalent to the
statement that $\mathcal{Y}_t^{n,\gamma}$ converges in the uniform
topology on $D([0,T], \S'(\R))$ to $\mathcal{Y}_t$, the unique solution
of the drift-removed Ornstein--Uhlenbeck equation
%
%e2.10 #&#
\begin{equation}
\label{OU2} \partial_t \mathcal{Y}_t =
\tfrac{1}{2} \varphi_c'(\rho) \Delta\mathcal
{Y}_t +\sqrt{\tfrac{1}{2}\varphi_b(\rho)} \nabla
\dot\W_t.
\end{equation}
This equation of course corresponds to \eqref{OU1} with $a=0$, is well
posed and has a unique solution (cf. \cite{Walsh}).

Now we increase the strength of the asymmetry in the jump rates by
decreasing the value of $\gamma$. We show for $1/2<\gamma<1$, starting
from the measures $\{\mu^n\}$, that there is no effect in the
convergence result of the fluctuation field.

%th2.2 #&#
\begin{theorem}[(Crossover fluctuations)]\label{th crossover}
For $1/2<\gamma<1$, starting from initial measures $\{\mu^n\}$, the
sequence $\{\mathcal{Y}_t^{n,\gamma}; n\geq{1}\}$ converges in the
uniform topology on $D([0,T], \S'(\R))$ to the process $\mathcal{Y}_t$
which is the solution of the Ornstein--Uhlenbeck equation \eqref{OU2}
with initial condition $\Y_0= \bar{\Y}_0$ given in \textup{(CLT)}.
\end{theorem}

However, for $\gamma=1/2$, which is a threshold, a much different
qualitative limit behavior is obtained as the strength of the weak
asymmetry in the jump rates is big enough to influence the limit field.
As mentioned in the \hyperref[intro_section]{Introduction}, the limit field $\Y_t$ should
satisfy, in some sense, a stochastic Burgers equation, written in our
framework as
%
%e2.11 #&#
\begin{equation}
\label{BE} \partial_t\mathcal{Y}_t =
\frac{\varphi_c'(\rho)}{2} \Delta\mathcal {Y}_t + \frac{a}{2}
\varphi_b''(\rho) \nabla
\mathcal{Y}_t^2 + \sqrt {\frac{1}{2}
\varphi_b(\rho)} \nabla \dot\W_t,
\end{equation}
although it is ill-posed.

We now detail in what sense we mean to ``solve'' \eqref{BE} in terms of a
martingale formulation.
Let $\iota\dvtx  \bb R \to[0,\infty)$ be the function $\iota(z) =
(1/2)1_{[-1,1]}(z)$. Also, for $0<\varepsilon\leq1$, define
$\iota_\varepsilon(z) = \varepsilon^{-1} \iota(\varepsilon^{-1}z)$
and let $G_\varepsilon\dvtx  \bb R \to[0,\infty)$ be a smooth compactly
supported function in $\S(\R)$ which approximates $\iota_\varepsilon$:
That is, $\|G_\varepsilon\|^2_{L^2(\R)} \leq2\|\iota_\varepsilon\|
^2_{L^2(\R)}=\varepsilon^{-1}$ and
\[
\lim_{\varepsilon\downarrow0}\varepsilon^{-1/2}\|G_\varepsilon-\iota
_\varepsilon\|_{L^2(\R)} = 0.
\]
Such choices can be readily found by convoluting $\iota_\varepsilon$
with smooth kernels.
Also, for $x\in\R$, define the shift $\tau_x$ so that $\tau_x
G_\varepsilon(z) = G_\varepsilon(x+z)$.

Consider now an $\S'(\R)$-valued process $\{\mc Y_t; t \in[0,T]\}$ and
for $0\leq s\leq t\leq T$ let
\[
\mc A_{s,t}^\varepsilon(H) = \int_s^t
\int_{\bb R} \nabla H(x) \bigl[\mc\Y_u(
\tau_{-x} G_\varepsilon) \bigr]^2 \,dx \,du.
\]
We say the process $\Y_\cdot$ satisfies the \emph{probability energy
condition} if
for each $H\in\S(\R)$,
%
%e2.12 #&#
\begin{equation}
\label{en.cond} \bigl\{\A_{s,t}^\varepsilon(H)\bigr\} \mbox{ is
Cauchy in probability as }\varepsilon\downarrow0
\end{equation}
and the limit in probability does not depend on the particular
smoothing family $\{G_\varepsilon\}$.
This limit defines the process
$\{\mc A_{s,t}; 0\leq s\leq t\leq T\}$ given by
\[
\mc A_{s,t}(H):= \lim_{\varepsilon\downarrow0} \mc
A^\varepsilon_{s,t}(H),
\]
which is $\S'(\R)$ valued (cf. pages~364--365; Theorem~6.15
of \cite{Walsh}).

We will say that $\{\mc Y_t; t \in[0,T]\}$ is a \emph{probability
energy solution} of \eqref{BE} if the following conditions hold:
\begin{longlist}[(iii)]
\item[(i)] Initially, $\mc Y_0$ is a spatial Gaussian process with
covariance $\C(G,H)$ for $G,H\in\S(\R)$.
\item[(ii)] The process $\{\mc Y_t; t \in[0,T]\}$ satisfies the
probability energy condition \eqref{en.cond}.
\item[(iii)] Then, the $\S'(\R)$ valued
process $\{\M_t\dvtx t\in[0,T]\}$ where
%
%e2.13 #&#
\begin{equation}
\label{formal_conservative}\qquad \mc M_t(H):= \mc Y_t(H) - \mc
Y_0(H) - \frac{\varphi'_c(\rho
)}{2}\int_0^t
\mc Y_s(\Delta H) \,ds - \frac{a \varphi''_b(\rho)}{2} \mc A_{0,t}(H)
\end{equation}
is a continuous martingale with quadratic variation
\[
\bigl\langle \mc M_t(H)\bigr\rangle = \frac{\varphi_b(\rho) t}{2}\|\nabla H\|_{L^2(\bb R)}^2.
\]
\end{longlist}
In particular, condition (iii) specifies by L\'evy's theorem that $\M
_t(H)$ is a Brownian motion with variance $(\varphi_b(\rho)/2)t\|\nabla
H\|^2_{L^2(\R)}$.

We also define a stronger notion of solution to \eqref{BE} which may be
verified in some cases. We say that $\Y_t$ satisfies the \emph{$L^2$
energy condition} if in \eqref{en.cond}, instead of in the probability
sense, we assert $\{\A^\varepsilon_{s,t}(H)\}$ is Cauchy in $L^2$ with
respect to the underlying probability measure, and $\A_{s,t}(H)$ is its
$L^2$ limit. Then we say $\Y_t$ is an \emph{$L^2$ energy solution} of
\eqref{BE} if (i) holds as before, (ii) the $L^2$ energy condition
holds and (iii) holds with respect to the $L^2$ limit $\A_{s,t}(H)$.

%th2.3 #&#
\begin{theorem}[(KPZ fluctuations)]
\label{th kpz scaling}
For $\gamma= 1/2$, starting from initial measures $\{\mu^n\}$, the
sequence of processes $\{\mathcal{Y}_t^{n,\gamma}\dvtx n\geq n_0\}$ is
tight in the uniform topology on
$D([0,T],\mathbb{S}'(\mathbb R))$. Moreover, any limit point of
$\mathcal{Y}_t^{n,\gamma}$ is a probability energy solution with
respect to \eqref{BE} with initial field $\bar\Y_0$ given in \textup{(CLT)}.

If the initial measure is $\mu^n\equiv\nu_\rho$, any limit point of $\Y
_t^{n,\gamma}$ is an $L^2$ energy solution of \eqref{BE} with initial
field $\bar\Y_0$ given in \textup{(CLT)}.
\end{theorem}

%re2.4 #&#
\begin{remark}
\label{kpz_remark}
We now make the following comments:
\begin{longlist}[1.]
\item[1.] Formally, equation \eqref{formal_conservative} corresponds to the
stochastic Burgers equation~\eqref{BE} where the nonlinear term is
represented by $\mc A_{0,t}$.
We remark, as in \cite{Assing}, by taking a fast subsequence in
$\varepsilon$, one may write $\mc A_{0,t}$ as a function of $\{\Y_u\dvtx u\leq t\}$, and form an equation in which $\Y_t$ satisfies \eqref{BE}
a.s. on a type of negative order Hermite Hilbert space.

\item[2.] We also remark, as alluded to in the \hyperref[intro_section]{Introduction}, if there were a
unique probability or $L^2$ energy solution, that is uniqueness of
process in the associated ``martingale
formulation,'' since with respect to simple exclusion the fluctuation
field limit is known in terms of the ``Cole--Hopf'' solution of the KPZ
equation~\cite{BG}, not only could one conclude a unique fluctuation
field limit in Theorem~\ref{th kpz scaling} in the framework of the
particle systems considered, but also identify it in terms of the
``Cole--Hopf'' apparatus.
What is required to show uniqueness of $\Y_t$ is to determine uniquely
its finite dimensional distributions (cf. Section~4.4 of \cite{EK}),
which the nonlinearity of $\mc A_{0,t}$ makes difficult.

\item[3.] We also note that the statement of Theorem~\ref{th kpz scaling} is
nontrivial when $a\neq0$ and $b$ is such that
%
%e2.14 #&#
\begin{equation}
\label{KPZ condition} \varphi_b''(\rho)
\neq{0}.
\end{equation}
Otherwise, when $\varphi''_b(\rho)=0$, the limit field $\Y_t$ satisfies
the Ornstein--Uhlenbeck equation (\ref{OU2}). Examples, fitting in our
framework, where the second derivative vanishes include types of
zero-range, that is independent particle systems where $\varphi_b(\rho)
= 2\rho$ which are in the EW class.
\end{longlist}
\end{remark}

%s2.3 #&#
\subsection{Model $1$: Zero-range processes}
\label{zero-range}
The one-dimensional weakly asymmetric zero-range process $\eta^n_{t}$,
on the state space $\Omega:=\mathbb{N}_0^{\mathbb{Z}}$, consists of a
collection of random walks which interact in that the jump rate of a
particle at vertex $x$ only depends on the number of particles at $x$.
More precisely, the generator is in form~(\ref{generator}) where
\[
b^{R,n}_x(\eta) = g\bigl(\eta(x)\bigr) \quad\mbox{and}\quad
b^{L,n}_x(\eta) = g\bigl(\eta(x+1)\bigr)
\]
do not depend on $n$ and are fixed with respect to a function $g\dvtx \N_0
\rightarrow\R_+$ such that $g(0)=0$, $g(k)>0$ for $k\geq1$ and $g$ is
Lipschitz,
\begin{longlist}[(LIP)]
\item[(LIP)] $\sup_{k\geq0}|g(k+1)-g(k)|<\infty$.
\end{longlist}
Under this specification, a Markov process $\eta^n_t$ can be
constructed (on a subset of $\Omega$) \cite{Andjel}.
Hence, (R1) holds and we identify the fixed function $c^n \equiv c$ as
\[
c(\eta) = g\bigl(\eta(0)\bigr).
\]

The zero-range process possesses a family of invariant measures which
are fairly explicit product measures.
For $\alpha\geq0$, define
\[
\mathcal{Z}(\alpha):=\sum_{k\geq{0}}\frac{\alpha^k}{g(k)!},
\]
where $g(k)! = g(1)\cdots g(k)$ for $k\geq1$ and $g(0)! = 1$.
Let $\alpha^*$ be the radius of convergence of this power series and
notice that $\mathcal{Z}$ increases on $[0,\alpha^*)$. Fix $0\leq\alpha
<\alpha^*$ and let $\bar{\nu}_{\alpha}$ be the product measure on
$\mathbb{N}^{\mathbb{Z}}$ whose marginal at the site $x$ is given by
\[
\bar{\nu}_{\alpha}\bigl\{\eta\dvtx \eta(x)=k\bigr\}=
\cases{ %
\displaystyle\frac{1}{\mathcal{Z}(\alpha)}\frac{\alpha^k}{g(k)!}, & \quad
$\mbox{when } k\geq1,$
\vspace*{2pt}\cr
\displaystyle\frac{1}{\mathcal{Z}(\alpha)},& \quad$\mbox{when } k=0.$}
\]
We now reparameterize these measures in terms of the ``density.'' Let
$\rho(\alpha):=E_{\bar{\nu}_{\alpha}}[\eta(0)] =\alpha\mathcal
{Z}'(\alpha)/\mathcal{Z}(\alpha)$.
By computing the derivative, we obtain that $\rho(\alpha)$ is strictly
increasing on $[0,\alpha^*)$. Then let $\alpha(\cdot)$ denote its
inverse. Define
\[
\nu_{\rho}(\cdot):=\bar{\nu}_{\alpha(\rho)}(\cdot),
\]
so that $\{\nu_{\rho}\dvtx 0\leq\rho<\rho^*\}$ is a family of invariant
measures parameterized by the density. Here,
$\rho^* = \lim_{\alpha\uparrow\alpha^*} \rho(\alpha)$, which may be
finite or infinite depending on whether $\lim_{\alpha\rightarrow{\alpha
^*}}\mathcal{Z}(\alpha)$ converges or diverges.

Note, since $\nu_\rho$ is a product measure, that $\nu_\rho^{\lambda
(z)} = \nu_z$ for $0\leq z<\rho^*$, and condition (D) holds.
One can readily check that (R2) holds:
\begin{eqnarray*}
g \bigl(\eta^{x+1,x}(x) \bigr)\frac{d\nu_\rho^{x+1,x}}{d\nu_\rho} & = & g\bigl(\eta(x)+1
\bigr) \frac{g(\eta(x))! g(\eta(x+1))!}{g(\eta(x)+1)! g(\eta
(x+1)-1)!}
\\
& = & g\bigl(\eta(x+1)\bigr).
\end{eqnarray*}

Also, by the construction in \cite{Sextremal}, which extends the
construction in \cite{Andjel} to an $L^2(\nu_\rho)$ process, we have
that $L_n$ is a Markov $L^2(\nu_\rho)$ generator whose core can be
taken as the space of all local $L^2(\nu_\rho)$ functions. Indeed, in
\cite{Sextremal}, a core of bounded Lipschitz functions is identified;
however, since any local $L^2(\nu_\rho)$ function is a limit of bounded
Lipschitz functions, and the formula \eqref{generator} is well defined
and $L^2(\nu_\rho)$-bounded for a local $L^2(\nu_\rho)$ function, by
dominated convergence the core can be extended.
It follows that the measures $\{\nu_\rho\dvtx  0\leq\rho<\rho^*\}$ are
invariant for the zero-range process. Also, (IM) holds as $\nu_\rho$ is
a product measure whose marginal has some exponential moments.
In addition, one can check that (EE) holds by Proposition~\ref{2EE}.

We now address the spectral gap properties of the system. Since the
model interactions are range $0$, the gap does not depend on the
outside variables $\xi$. However, the gap depends on $g$, as it should
since $g$ controls the rate of jumps. We identify three types\vspace*{1pt} of rates
for which a spectral gap bound has been proved. Let $\beta= k/(2\ell+1)^d$.

\begin{itemize}
\item If $g$ is not too different from the independent case, for
which the gap is of order $O(\ell^{-2})$ uniform in $k$, one expects
similar behavior as for a single particle. This has been proved for
$d\geq1$ in \cite{LSV} under assumptions (LIP) and
\begin{longlist}[(U)]
\item[(U)] There exists $x_0$ and $\varepsilon_0>0$ such that $g(x+x_0)
- g(x)\geq\varepsilon_0$ for all $x\geq0$.
\end{longlist}

\item If $g$ is sublinear, that is $C^{-1}x^\gamma\leq g(x+1)-g(x)
\leq C x^\gamma$ for $0<\gamma<1$ and $C>0$, then it has been shown
that the spectral gap depends on the number of particles $k$, namely
the gap for $d\geq1$ is $O((1+\beta)^{-\gamma} \ell^{-2})$ \cite{Nagahata}.

\item If $g(x) = 1(x\geq1)$, then it has been shown in $d\geq1$
that the gap is $O((1+\beta)^{-2}\ell^{-2})$ \cite{Morris}. In $d=1$,
this is true because of the connection between the zero-range and
simple exclusion processes for which the gap estimate is well
known~\cite{Quastel}: The number of spaces between consecutive particles in
simple exclusion correspond to the number of particles in the
zero-range process.
\end{itemize}

In all these cases, (G) follows readily by straightforward moment calculations.

%s2.4 #&#
\subsection{Model $2$: Kinetically constrained exclusion systems}
\label{porous}
We consider a type of exclusion process, which may be thought of as a
microscopic model for porous medium behavior, developed in \cite{GLT}
and references therein, in one dimension on $\Omega= \{0,1\}^\Z$ where
particles more likely hop to unoccupied nearest-neighbor sites when at
least $m-1\geq1$ other neighboring sites are full. When $m=2$, the
rates are in the form
\begin{eqnarray*}
b_x^{R,n}(\eta;\theta)&=& \eta(x) \bigl(1-\eta(x+1)\bigr)
\biggl[\eta(x-1)+\eta(x+2) + \frac{\theta}{2n} \biggr],
\\
b_x^{L,n}(\eta;\theta) &=& \eta(x+1) \bigl(1-\eta(x)
\bigr) \biggl[\eta(x-1) + \eta (x+2) + \frac{\theta}{2n} \biggr],
\end{eqnarray*}
with respect to a parameter $\theta>0$.
If $\theta$ would vanish, particles can jump from site $x$ to $x+1$
exactly when there is at least $1$ particle in the vicinity of the bond
$(x,x+1)$. However, with $\theta>{0}$, the jump from $x$ to $x+1$ may
also occur irrespective of the neighboring particle structure with a
small rate $\theta/(2n)$.

When $m\geq2$, the rates generalize to
\begin{eqnarray*}
b_x^{R,n}(\eta;\theta) &=& \eta(x) \bigl(1-\eta(x+1)
\bigr)A_n(\eta;\theta),
\\
b_x^{L,n}(\eta;\theta) &=& \eta(x+1) \bigl(1-\eta(x)
\bigr)A_n(\eta;\theta),
\end{eqnarray*}
where $A_n(\eta;\theta)$ equals
\[
\prod_{j=-(m-1)}^{-1}\eta(x+j)+\mathop{\prod
_{j=-(m-2)}}_{j\neq
{0,1}}^{2}\eta(x+j)+
\cdots+\mathop{\prod_{j=-1}}_{j\neq0,1}^{m-1}
\eta(x+j) + \prod_{j=2}^{m}\eta(x+j)+
\frac{\theta}{2n}.
\]

The role of $\theta>0$ is to make the system ``ergodic.'' If $\theta=0$,
there would be an infinite number of invariant measures, such as Dirac
measures supported on configurations which cannot evolve under the
dynamics. The hydrodynamic limit for this model corresponds to the
porous medium equation, $\partial_t\rho_t(t,u)=\Delta\rho^m(t,u)$.

Now, one may calculate that
$b_x^{R,n}(\eta;\theta)-b_x^{L,n}(\eta;\theta)=c_x^n(\eta
)-c_{x+1}^n(\eta)$ where, for $m\geq2$,
\begin{eqnarray*}
c^n(\eta;\theta) & = & \prod_{j=-(m-1)}^{0}
\eta(j)+\cdots+\prod_{j=0}^{m-1}\eta(j)
\\
&&{} -\mathop{\prod_{j=-(m-1)}}_{j\neq{0}}^{1}
\eta(j)-\cdots -\mathop{\prod_{j=-1}}_{j\neq{0}}^{m-1}
\eta(j)+\frac{\theta}{2n}\eta(0).
\end{eqnarray*}
In the case $m=2$, the last formula reduces to
$
c^n(\eta;\theta)=\eta(-1)\eta(0)+\eta(0)\eta(1)-\eta(-1)\eta(1)+\frac
{\theta}{2n}\eta(0)$.

Of course, uniformly in $\eta$, as $n\uparrow\infty$, the terms
involving $\theta$ vanish,
\begin{eqnarray*}
 b_x^{R,n}(\eta;\theta) &\rightarrow&
b_x^{R}(\eta):= b^{R,1}_x(\eta
;0),\qquad b_x^{L,n}(\eta;\theta)\rightarrow
b_x^L(\eta):= b^{L,1}_x(\eta;0)\quad\mbox{and}
\\
c^n(\eta;\theta)&\rightarrow& c:=c^1(
\eta;0).
\end{eqnarray*}

Consider now the Bernoulli product measure on $\Omega$:
\[
\nu_\rho = \prod_{x\in\Z}
\mu_\rho\qquad \mbox{where } \mu_\rho (1)=1-\mu_\rho(0)
= \rho
\]
for $\rho\in[0,1]$. By the construction in \cite{Liggett}, it is now
standard that $L_n$ is a Markov $L^2(\nu_\rho)$ generator.
One may also inspect that condition (R2) holds with respect to $\nu_\rho$.
Hence, $\nu_\rho$ is
invariant for $\rho\in[0,1]$. Condition (IM) also holds as $\nu_\rho$
supports two-state configurations.
In addition, as $\nu_\rho$ is a product measure, $\nu_\rho^{\lambda
(z)}=\nu_z$ and~(D) holds. Also, by Proposition~\ref{2EE}, (EE) is satisfied.

We now discuss the spectral gap behavior of the process.

%pr2.5 #&#
\begin{proposition} \label{spectralgap KCLG}
For kinetically constrained exclusion processes evolving on $\Lambda
_\ell$, when $m\geq2$, there exists a constant $C$, uniform over $\xi$
and $n$, such that
\[
\label{estimate inv sg for m} W(k,\ell,\xi,n) \leq C\ell^2 \biggl(
\frac{\ell}{k} \biggr)^m1(k\geq1).
\]
\end{proposition}

When $m=2$ and $k\leq\ell/3$, the above spectral gap estimate is
already given in Proposition~6.2 of \cite{GLT}. However, a
straightforward modification of the proof of Proposition~6.2 in \cite
{GLT} yields the more general estimate in Proposition~\ref{spectralgap
KCLG}. Indeed, the difference when $m\geq2$ is that to bound equation
(6.10) in \cite{GLT} in the general case, one uses that there are at
most $Cj^{m-1}$ ways to arrange $m-1$ particles in an interval of width
$j$. Now, a similar optimization on $j$ as given in the proof of
Proposition~6.2 of \cite{GLT} leads to the desired generalized spectral
gap estimate.

%le2.6 #&#
\begin{lemma}\label{kcem_rmk} For the kinetically constrained exclusion model, the spectral gap
condition \textup{(G)} is satisfied.
\end{lemma}

\begin{pf} 
With respect to a constant $C$, which may change line to line,
\begin{eqnarray*}
&&E_{\nu_\rho} \biggl[ \biggl(W\biggl(\sum_{x\in\Lambda_\ell}
\eta(x), \ell, \xi,n\biggr) \biggr)^2 \biggr] \\
&&\qquad\leq C
\ell^4 E_{\nu_\rho} \biggl[ 1 \biggl(\frac{1}{2\ell+1}\leq
\eta^{(\ell)} \biggr) \bigl(\eta^\ell \bigr)^{-2m}
\biggr]
\\
&& \qquad\leq C \ell^4 \biggl\{ \varepsilon^{-2m} +
E_{\nu_\rho} \biggl[1 \biggl(\frac
{1}{2\ell+1}\leq \eta^{(\ell)}<
\varepsilon \biggr) \bigl(\eta^{(\ell)} \bigr)^{-2m} \biggr] \biggr
\}
\\
&& \qquad\leq C\ell^4 \bigl\{\varepsilon^{-2m} +
\ell^{2m}P_{\nu_\rho} \bigl(\eta^{(\ell)}<\varepsilon \bigr)
\bigr\}
\end{eqnarray*}
for a fixed $\varepsilon<\rho$. Then, as $\nu_\rho$ is a Bernoulli
product measure with density $\rho$, by a large deviations estimate say,
$E_{\nu_\rho} [W (\sum_{x\in\Lambda_\ell}\eta(x), \ell, \xi,n
)^2 ]\leq C\ell^4$ for
all $\ell\geq1$.
\end{pf}

%s2.5 #&#
\subsection{Model 3: Gradient exclusion with speed change}
\label{speed_change_model}
In this version of exclusion on $\Omega=\{0,1\}^\Z$, rates are chosen
which correspond to a Hamiltonian with nearest-neighbor interactions,
\[
Q_\beta(\eta) = -\beta\sum_{x\in\Z} \bigl(
\eta(x)-1/2\bigr) \bigl(\eta(x+1)-1/2\bigr)
\]
for $\beta\in\R$, which will be reversible with respect to a
stationary Markovian measure $\nu_{1/2}$.
That is, specify $\nu_{1/2}$ by its finite-dimensional distributions
\[
\nu_{1/2} \bigl(\eta(x) = e(x)\dvtx x\in\Lambda_\ell|
\eta(y)=\xi(y) \mbox{ for } y\notin\Lambda_\ell \bigr) =
\frac{e^{-Q_{\beta,\ell}(e,\xi)} }{\mathcal{Z}},
\]
where
\begin{eqnarray*}
Q_{\beta,\ell}(e,\xi) &=& -\beta\sum_{x,x+1\in\Lambda_\ell
}
\bigl(e(x)-1/2\bigr) \bigl(e(x+1)-1/2\bigr)
\\
&&{} -\beta\bigl(\xi(-\ell-1)-1/2\bigr) \bigl(e(-\ell)-1/2\bigr) \\
&&{}- \beta\bigl(e(
\ell )-1/2\bigr) \bigl(\xi(\ell+1)-1/2\bigr),
\end{eqnarray*}
$e,\xi\in\Omega$ and $\mathcal{Z}=\mathcal{Z}(\ell,\xi)$ is the normalization.
It is not difficult to see that $\nu_{1/2}$ is Markovian with
transition matrix
\[
P = \frac{1}{e^{\beta/4}+e^{-\beta/4}}\left[ %
\matrix{
e^{\beta/4}&e^{-\beta/4}
\vspace*{2pt}\cr
e^{-\beta/4}& e^{\beta/4} }\right],
\]
and marginal distribution $\langle1/2, 1/2\rangle$ so that
$E_{\nu_{1/2}}[\eta(0)]=1/2$.

Then, as discussed in \cite{Spohn}, Section II.2.4, (R2) is ensured if
we take the rates $b^{R,n}_x = b^R_x$ and $b^{L,n}_x = b^L_x$ which do
not depend on $n$ as
\begin{eqnarray*}
b^R_x(\eta) &=& \eta(x) \bigl(1-\eta(x+1)\bigr)
\\
&&{} \times \bigl[\alpha_1\eta(x-1)\eta(x+2) + \alpha _2
\bigl(1-\eta(x-1)\bigr)\eta(x+2)
\\
&&\hspace*{16pt}{} + \alpha_3\eta(x-1) \bigl(1-\eta(x+2)\bigr) + \alpha
_4\bigl(1-\eta(x-1)\bigr) \bigl(1-\eta(x+2)\bigr) \bigr],
\\
b^L_x(\eta)&=& \eta(x+1) \bigl(1-\eta(x)\bigr)
\\
&&{} \times \bigl[\alpha_1\eta(x-1)\eta(x+2) + \alpha _3
\bigl(1-\eta(x-1)\bigr)\eta(x+2)
\\
&&\hspace*{16pt}{} + \alpha_2\eta(x-1) \bigl(1-\eta(x+2)\bigr) + \alpha
_4\bigl(1-\eta(x-1)\bigr) \bigl(1-\eta(x+2)\bigr) \bigr],
\end{eqnarray*}
where $\alpha_1, \alpha_2=e^{\beta}\alpha_3, \alpha_4>0$. The condition
(R1) also follows if we also assume that $\alpha_1-\alpha_2-\alpha
_3+\alpha_4=0$ so that, as can be checked, $c(\eta)$ takes the form
\begin{eqnarray*}
c(\eta) & = & \alpha_4\eta(0) + (\alpha_3-
\alpha_4)\eta(-1)\eta(0) + (\alpha_3-\alpha_4)
\eta(0)\eta(1)
\\
&&{} + (\alpha_4-\alpha_2)\eta(-1)\eta(1) + (
\alpha_2-\alpha_3)\eta (-1)\eta(0)\eta(1).
\end{eqnarray*}

Again, by \cite{Liggett}, $L_n$ is a Markov $L^2(\nu_\rho)$ generator
for the process. We note when $\beta=0$ and $\alpha_i = 1$ for
$i=1,2,3,4$, the model is the simple exclusion process and $\nu_{1/2}$
is the Bernoulli product measure with density $1/2$.

We now introduce a family of stationary, reversible measures by use of
a ``tilt'' or ``chemical potential'' $\lambda$. Define $\nu^\lambda_{1/2}$,
again specified by its finite-dimensional distributions, through the relation
\[
\frac{d\nu^\lambda_{1/2}}{d\nu_{1/2}} \bigl(\eta(x) = e(x)\dvtx x\in\Lambda _\ell|
\eta(y)=\xi(y) \mbox{ for }y\notin\Lambda_\ell \bigr) =
\frac{e^{\lambda\sum_{x\in\Lambda_\ell}(e(x)-1/2)}}{\mathcal{Z}'},
\]
where $e,\xi\in\Omega$ and $\mathcal{Z}'= \mathcal{Z}'(\ell,\xi)$ is
another normalization. These measures are also Markovian with
transition matrix
%
%e2.15 #&#
\begin{equation}
\label{p_formulas} P_\lambda = \left[ \matrix{
1-u_1 &u_1
\vspace*{2pt}\cr
v_1& 1-v_1 }
\right],
\end{equation}
where
\[
u_1 = \frac{r_1(\lambda, \beta) + \sinh(\lambda/2)}{\cosh(\lambda/2) +
r_1(\lambda,\beta)},\qquad v_1 = \frac{r_1(\lambda,\beta) - \sinh(\lambda/2)}{\cosh(\lambda/2) +
r_1(\lambda,\beta)}
\]
and $r_1(\lambda,\beta) = \sqrt{\sinh^2(\lambda/2) + e^{-\beta}}$.
The stationary distribution equals $\pi_\lambda= (v_1+u_1)^{-1}\langle
v_1, u_1\rangle$, which is the marginal distribution of $\nu
_{1/2}^\lambda$. These derivations are performed in \cite{Baxter} and
\cite{MeiYin}.

The measures $\{\nu_{1/2}^\lambda\dvtx  \lambda\in\R\}$ are uniformly
mixing: Indeed, the eigenvalues of $P_\lambda$ are $1$ and $1-u_1-v_1$,
and the spectral gap $u_1+v_1$ is uniformly bounded away from $0$ for
$\lambda\in\R$.

One can calculate $E_{\nu^\lambda_{1/2}}[\eta(0)] = u_1/(u_1+v_1)$
strictly increases in $\lambda$.
To parameterize in terms of ``density,'' recall
$\nu^{\lambda(z)}_{1/2}= \nu_z$ where $\lambda=\lambda(z)$ is chosen so
that $E_{\nu_z}[\eta(0)]= z$. Here, as $z\downarrow0=\rho_*$, $\lambda
(z)\downarrow-\infty$ and, as $z\uparrow1= \rho^*$, $\lambda
(z)\uparrow\infty$; also $\lambda(1/2) = 0$.
Hence, since also $\nu_{z}$ is exponentially mixing, both (IM) and (D) hold.

From the defining relation for $\lambda(z)$, $u_1/(u_1+v_1)=z$, one can
differentiate at $z=1/2$ to find $\lambda'(1/2)[e^{-\beta/2}/4] = 1$.

Also, we note the additive functional variance $\sigma^2(z)$ [cf.
(IM2)] satisfies the formula $\sigma^2(z) = E_{\pi_{\lambda(z)}}[u^2] -
E_{\pi_{\lambda(z)}}[(P_{\lambda(z)} u)^2]$ where $(I-P_{\lambda(z)})u
= f$ and $f = \langle-z, 1-z\rangle$ represents the values of the
function $f(\eta) = \eta(0)-z$; see Section~6.5 of \cite
{Varadhannotes}. In fact, we find $\sigma^2(1/2) = e^{-\beta/2}/4$ and
so $\lambda'(1/2)\sigma^2(1/2) = 1$.

The spectral gap for a more general process, including this one, has
been bounded as follows \cite{LuYau}: Uniformly over $k$ and $\xi$ (it
does not depend on $n$), we have
\[
W (k, \ell, \xi,n ) \leq C\ell^2.
\]
Hence, (G) holds.

Also, in Proposition~\ref{Markov_EE}, we show that (EE) holds.

%s3 #&#
\section{Proofs-outline}
\label{proofs}
The strategies of the proofs for Theorems \ref{th crossover} and \ref
{th kpz scaling} are similar. We consider the stochastic differential
of $\Y^{n,\gamma}_t$ and represent it in terms of corrector and
martingale terms. Tightness is shown for each term in the decomposition
of $\Y^{n,\gamma}_t$. Under the assumption that the initial measure is
the invariant state $\nu_\rho$, limit points are identified using a
Boltzmann--Gibbs principle, and shown to satisfy \eqref{OU2} when
$1/2<\gamma\leq1$ and to be energy solutions of \eqref{BE} when $\gamma
=1/2$. When the initial measures $\{\mu^n\}$ satisfy (BE), the entropy
inequality then allows to characterize the limit points as desired.

In the following Sections~\ref{Assoc_mart_section}--\ref
{tightness_section}, associated martingales, Boltzmann--Gibbs
principles and tightness are discussed. In Section~\ref{identification}, limit points are identified and Theorems \ref{th
crossover} and \ref{th kpz scaling} are proved.

To reduce some of the notation, we will drop the superscript ``$n$'' in
the rate functions and write $b^{R,n}_x = b^R_x$, $b^{L,n}_x=b^L_x$,
$b^n_x=b_x$, $b^n=b$, $c^n_x=c_x$ and $c^n=c$ until Section~\ref{identification}.

%s3.1 #&#
\subsection{Associated martingales}
\label{Assoc_mart_section}

For $H\in\S(\R)$, $x\in\Z$ and $n\geq1$, define
\begin{eqnarray*}
\triangle_x^nH &=& n^2 \biggl\{H \biggl(
\frac{x+1}{n} \biggr) + H \biggl(\frac
{x-1}{n} \biggr) - 2H \biggl(
\frac{x}{n} \biggr) \biggr\},
\\
\nabla_x^nH &=& n \biggl\{H \biggl(\frac{x+1}{n}
\biggr) - H \biggl(\frac{x}{n} \biggr) \biggr\}.
\end{eqnarray*}
Define also, for $\gamma, s\geq0$, the functions
%
%e3.1 #&#
\begin{eqnarray}
\label{shifted_H} H_{\gamma,s}(\cdot) &=& H \biggl( \cdot- \frac{1}{n}
\biggl\lfloor\frac
{a\varphi'_b(\rho)sn^2}{2n^\gamma} \biggr\rfloor \biggr) \quad\mbox{and}
\nonumber
\\[-8pt]
\\[-8pt]
\nonumber
\HW_{\gamma,s}(\cdot) &=& H \biggl( \cdot- \frac{1}{n} \biggl\{
\frac{a\varphi
'_b(\rho)sn^2}{2n^\gamma} \biggr\} \biggr).
\end{eqnarray}
We note, in $H_{\gamma,s}$, the process characteristic shift is along
$n^{-1}\Z$, which helps make tidy some proofs [in applying a
Boltzmann--Gibbs principle (Theorem~\ref{gbg_L2}) in proofs of
Propositions \ref{stationary_tightness} and \ref{stat_lemma1}], instead
of along $\R$ as in $\HW_{\gamma,s}$.

Let $F(s, \eta^n_s;H,n) = \Y^{n,\gamma}_s(H)$, and $F(\eta; H,n) =
n^{-1/2}\sum_{x\in\Z} H(\frac{x}{n})(\eta(x)-\rho)$. Although $F(\eta
;H,n)$ is an $L^2(\nu_\rho)$ function, in general, it is not a local
function. However, by approximation with local functions and noting by
condition (R1) that $|b(\eta)| \leq C\sum_{|x|\leq R}\eta(x)$, one may
conclude $F(\eta;H,n)$ and also $F^2(\eta;H,n)$ belong to the domain of
$L_n$. In particular,
\[
L_nF\bigl(s, \eta^n_s; H,n\bigr) =
\frac{1}{2\sqrt{n}}\sum_{x\in\Z} c_x\bigl(
\eta ^n_s\bigr) \triangle^n_x
\HW_{\gamma,s} + \frac{a}{2n^{\gamma-1/2}}\sum_{x\in
\Z}
b_x\bigl(\eta^n_s\bigr)
\nabla^n_x\HW_{\gamma,s}.
\]
Also,
\begin{eqnarray*}
\frac{\partial}{\partial_s}F\bigl(s,\eta^n_s; H,n\bigr) & = &
\biggl\{\frac
{-a\varphi'_b(\rho)n^2}{2n^\gamma} \biggr\} \frac{1}{n^{3/2}}\sum
_{x\in\Z
} \nabla\HW_{\gamma, s} \biggl(\frac{x}{n}
\biggr) \bigl(\eta^n_s(x)-\rho \bigr).
\end{eqnarray*}

Then
\begin{eqnarray*}
\M^{n,\gamma}_t(H) &:=& F\bigl(t, \eta^n_t;
H,n\bigr) - F\bigl(0, \eta^n_0; H,n\bigr)
\\
&&{} - \int_0^t \frac{\partial}{\partial_s}F\bigl(s,
\eta ^n_s; H,n\bigr) + L_n F\bigl(s,
\eta^n_s; H,n\bigr) \,ds
\end{eqnarray*}
is a martingale.
We may decompose
%
%e3.2 #&#
\begin{equation}
\label{mart_decomposition}\quad \M^{n,\gamma}_t(H) = \Y^{n,\gamma}_t(H)
- \Y^{n,\gamma}_0(H) - \I ^{n,\gamma}_t(H) -
\B^{n,\gamma}_t(H) -\K^{n,\gamma}_t(H),
\end{equation}
where
\begin{eqnarray*}
&&\mathcal{I}^{n,\gamma}_t(H) = \frac{1}{2}\int
_0^t \frac{1}{\sqrt{n}}\sum
_{x\in\Z} \bigl(c_x\bigl(\eta ^n_s
\bigr)-\varphi_c(\rho) \bigr)\triangle^n_xH_{\gamma,s}
\,ds,
\\
&&\mathcal{B}^{n,\gamma}_t(H) = \frac{a}{2n^{\gamma-1/2}}\int
_0^t \sum_{x\in\Z}
\bigl(b_x\bigl(\eta^n_s\bigr)-
\varphi_b(\rho) - \varphi_b'(\rho) \bigl(\eta^n_s(x)-
\rho\bigr) \bigr)\nabla^n_xH_{\gamma,
s} \,ds,
\\
&&\K^{n,\gamma}_t(H) \\
&&\qquad= \int_0^t
\biggl[\frac{1}{\sqrt{n}}\sum_{x\in\Z
}
\kappa^{n,1}_x(H,s) \bigl(c_x\bigl(
\eta^n_s\bigr)-\varphi_c(\rho) \bigr)
\\
&&\quad\qquad\hspace*{17pt}{} + \frac{a}{2n^{\gamma-1/2}}\sum_{x\in\Z}\kappa
^{n,2}_x(H,s) \bigl(b_x\bigl(
\eta^n_s\bigr)-\varphi_b(\rho) -
\varphi_b'(\rho) \bigl(\eta ^n_s(x)-
\rho\bigr) \bigr) \biggr]\,ds.
\end{eqnarray*}
Here, we introduced the centering constants $\varphi_c(\rho)$ and
$\varphi_b(\rho)$ in $\I^{n,\gamma}_t$ and $\B^{n,\gamma}_t$ as
$\triangle^n_xH_{\gamma,s}$ and $\nabla^n_xH_{\gamma,s}$ both sum to
zero. Also,
\begin{eqnarray*}
\kappa^{n,1}_x(H,s) & = & \Delta^n_x
(\HW_{\gamma,s} - H_{\gamma
,s} ) = O\bigl(n^{-1}\bigr)\cdot
\Delta^n_x H'_{\gamma,s} + O
\bigl(n^{-2}\bigr)\cdot H^{(4)}_{\gamma,s}
\bigl(x'/n\bigr),
\\
\kappa^{n,2}_x(H,s) &=& \nabla^n_x
(\HW_{\gamma,s}-H_{\gamma,s} )
\\
& = & O\bigl( n^{-1}\bigr)\cdot\Delta H_{\gamma,s}(x/n) + O
\bigl(n^{-2}\bigr)\cdot H'''_{\gamma,s}
\bigl(x''/n\bigr), %\label{K_bound}
\end{eqnarray*}
where $|x'-x|, |x''-x|\leq2$.

To capture the quadratic variation $\langle\M^{n,\gamma}_t\rangle$, we compute
\begin{eqnarray*}
&&L_n F^2\bigl(s,\eta^n_s; H,n
\bigr) - 2F\bigl(s,\eta^n_s; H,n\bigr)L_nF
\bigl(s,\eta^n_s; H,n\bigr)
\\
&&\qquad= \frac{1}{2n}\sum_{x\in\Z} b_x
\bigl(\eta^n_s\bigr) \bigl(\nabla^n_x
\HW_{\gamma
,s}\bigr)^2 + \frac{a}{2n^{1+\gamma}}\sum
_{x\in\Z} \bigl(c_x\bigl(\eta^n_s
\bigr) - c_{x+1}\bigl(\eta^n_s\bigr) \bigr)
\bigl(\nabla^n_x \HW_{\gamma, s}\bigr)^2
\end{eqnarray*}
so that $(\M^{n,\gamma}_t(H))^2 - \langle\M^{n,\gamma}_t(H)\rangle$ is
a martingale with
\begin{eqnarray*}
\bigl\langle\M^{n,\gamma}_t(H) \bigr\rangle& = & \int
_0^t \frac{1}{2n}\sum
_{x\in\Z} \bigl(\nabla^n_x
\HW_{\gamma,s}\bigr)^2 b_x\bigl(
\eta^n_s\bigr)\,ds
\\
&& {}+ \int_0^t \frac{a}{2n^{1+\gamma}}\sum
_{x\in\Z} \bigl(c_x\bigl(
\eta^n_s\bigr) - c_{x+1}\bigl(
\eta^n_s\bigr) \bigr) \bigl(\nabla^n_x
\HW_{\gamma
,s}\bigr)^2 \,ds.
\end{eqnarray*}
When starting from the invariant measure $\nu_\rho$, noting the bounds
in (R1), we have
%
%e3.3 #&#
\begin{eqnarray}
\label{quad_var_bound}
&&\E_{\nu_\rho} \bigl[ \bigl(\M^{n,\gamma}_t(H)-
\M^{n,\gamma}_s(H) \bigr)^2 \bigr]
\nonumber
\\
&&\qquad \leq \biggl\{\int_s^t \biggl(
\frac{1}{n}\sum_{x\in\Z}\bigl(
\nabla^n_x \HW_{\gamma
,s}\bigr)^2
\biggr)\,ds \biggr\}
\nonumber
\\[-8pt]
\\[-8pt]
\nonumber
&&\qquad\quad{}\times \biggl[\frac{1}{2}E_{\nu_\rho}\bigl[b(\eta)\bigr] +
\frac{a}{2n^\gamma}E_{\nu_\rho}\bigl[ \bigl|c_0(\eta) -
c_1(\eta)\bigr|\bigr] \biggr]
\\
&&\qquad \leq C(a)\|b\|_{L^1(\nu_\rho)}\int_s^t
\biggl(\frac{1}{n} \sum_{x\in\Z} \bigl(
\nabla^n_x \HW_{\gamma,s}\bigr)^2
\biggr)\,ds.\nonumber
\end{eqnarray}

To express an exponential martingale, we now observe for $0\leq\lambda
\leq\lambda(H,n)$ small that $\exp\{\lambda F(\eta; H,n)\}$ is in the
domain of $L_n$. Indeed, if $H$ is a local function, as $\nu_\rho$ is
assumed in (IM) to have small parameter exponential moments, then $\exp
\{\lambda F(\eta;H,n)\}\in L^2(\nu_\rho)$ for all small $\lambda$.
Again, an approximation argument when $H\in\S(\R)$ is not local shows
also $\exp\{\lambda F(\eta;H,n)\}$ belongs to the domain of $L_n$.
We calculate
\begin{eqnarray*}
&&\exp \bigl\{-\lambda F\bigl(u,\eta^n_u;H,n\bigr)
\bigr\} \biggl(\frac{\partial
}{\partial_u} + L_n \biggr)\exp \bigl\{\lambda F
\bigl(u,\eta^n_u; H,n\bigr) \bigr\}
\\
&&\qquad = n^2\sum_{x\in\Z}
\bigl[b_x^R(\eta) p_n \bigl( \exp \bigl\{
\lambda n^{-3/2}\bigl(\nabla^n_x
\HW_{\gamma,u}\bigr) \bigr\} - 1 \bigr)
\\
&&\hspace*{30pt}\qquad\quad{} + b_x^L(\eta)q_n \bigl( \exp \bigl\{-
\lambda n^{-3/2}\bigl(\nabla^n_x \HW
_{\gamma,u}\bigr) \bigr\} - 1 \bigr) \bigr]
\\
&&\qquad\quad{} - \frac{1}{n^{3/2}} \biggl\{\frac{a\lambda\varphi'_b(\rho
)n^2}{2n^\gamma} \biggr\}\sum
_{x\in\Z} \nabla\HW_{\gamma, u}(x/n) \bigl(
\eta^n_u(x)-\rho \bigr),
\end{eqnarray*}
which, given the assumptions on $b$ in (R1) and on moments of $\nu_\rho
$ in (IM), belongs to $L^2(\nu_\rho)$.

Hence, by the proof of Lemma IV.3.2 of \cite{EK},
\[
\mathcal{Z}_{s,t} = \exp \biggl\{ \lambda F\bigl(t,
\eta^n_t\bigr) - \lambda F\bigl(s,\eta^n_s
\bigr) - \int_s^t e^{-\lambda F(u,\eta^n_u)} \biggl(
\frac{\partial
}{\partial_u} + L_n \biggr)e^{\lambda F(u,\eta^n_u)}\,du \biggr\}
\]
is a martingale. We may expand $\mathcal{Z}_{s,t}$ in terms of $\lambda
$ as
\begin{eqnarray*}
\mathcal{Z}_{s,t} &=& \exp \biggl\{ \lambda \bigl(
\M^{n,\gamma}_t(H) - \M ^{n,\gamma}_s(H) \bigr)
\\
&&\hspace*{20pt}{} - \frac{\lambda^2}{2} \bigl\langle\M ^{n,\gamma}_t(H) -
\M^{n,\gamma}_s(H)\bigr\rangle
\\
&&\hspace*{20pt}{} + \frac{\lambda^3}{3!} \int_s^t
\mathcal{R}_1\,du + \frac{\lambda^4}{4!} \int_s^t
\mathcal{R}_2\,du + \lambda^5\int_s^t
\mathcal{R}_3\,du \biggr\},
\end{eqnarray*}
where
\begin{eqnarray*}
\mathcal{R}_1(u) & = & \frac{n^2}{2n^{9/2}}\sum
_{x\in\Z} \bigl(b_x^R(
\eta)-b_x^L(\eta) \bigr) \bigl(\nabla^n_x
\HW_{\gamma,u} \bigr)^{3}
\\
&&{} + \frac{an^2}{2n^{9/2 +(1/2+\gamma)}}\sum_{x\in\Z}
b_x(\eta ) \bigl(\nabla^n_x
\HW_{\gamma,u} \bigr)^{3},
\\
\mathcal{R}_2(u) &=& \frac{n^2}{2n^6}\sum
_{x\in\Z} b_x(\eta) \bigl(\nabla
^n_x \HW_{\gamma,u} \bigr)^4
\\
&&{} + \frac{an^2}{2n^{6+(1/2 + \gamma)}}\sum_{x\in\Z}
\bigl(b_x^R(\eta)-b_x^L(\eta)
\bigr) \bigl(\nabla^n_x \HW_{\gamma,u}
\bigr)^4.
\end{eqnarray*}
By the gradient condition and the bound on $b$ in assumption (R), one
may compute for $i=1,2$ that
%
%e3.4 #&#
\begin{equation}
\label{R_bounds} \bigl\|\mathcal{R}_i(u)\bigr\|_{L^4(\nu_\rho)} \leq
\frac{C(a)}{n^{3/2}} \bigl\|b(\eta)\bigr\|_{L^4(\nu_\rho)} \biggl(\frac{1}{n}\sum
_{x} \bigl|\nabla^n_x
\HW_{\gamma,u}\bigr|^{2+i} \biggr).
\end{equation}

Since $\E_{\nu_\rho}[\mathcal{Z}_{s,t}]=1$, by expanding in powers of
$\lambda$, using Schwarz inequality, the bound on the quadratic
variation \eqref{quad_var_bound}, bounds on $\mathcal{R}_i$ \eqref
{R_bounds} and invariance of~$\nu_\rho$, we obtain a bound for the
fourth moment of $\M^{n,\gamma}_t(H)-\M^{n,\gamma}_s(H)$:
%
%e3.5 #&#
\begin{eqnarray}
\label{fourth_moment_mart} &&\E_{\nu_\rho} \bigl[ \bigl(\M^{n,\gamma}_t(H)
- \M^{n,\gamma}_s(H)\bigr)^4 \bigr]
\nonumber
\\[-8pt]
\\[-8pt]
\nonumber
&& \qquad\leq C(a,H)\|b\|_L^4(\nu_\rho) \bigl(
|t-s|^2 + n^{-3/2}|t-s| \bigr).
\nonumber
\end{eqnarray}

%s3.2 #&#
\subsection{Generalized Boltzmann--Gibbs principles} \label{BG_statement}

To treat the stochastic differential of $\Y^{n,\gamma}_t$, we replace
the spatial terms of form
$\sum_{x\in\Z} h(x) \tau_x f(\eta)$,
where $h$ is a function on $\Z$ and $f$ is a local function,
in terms of the fluctuation field itself to close the evolution
equations. Such replacements fall under the term ``Boltzmann--Gibbs
principles'' coined by Brox--Rost in \cite{BR} which have general
validity. For instance, the following result forms the backbone of the
argument for Proposition~\ref{OU_oldprop}, when starting from the
invariant measure $\nu_\rho$, with respect to the papers cited just
before the proposition statement.

%pr3.1 #&#
\begin{proposition}
\label{usualbg}

Let $f$ be a local $L^2(\nu_\rho)$ function. For $t\geq0$ and $h\in\ell
^2(\Z)$, we have
\[
\lim_{n \to\infty} \E_{\nu_\rho} \biggl[ \biggl(\int
_0^t \frac{1}{\sqrt n} \sum
_{x \in\bb Z} \bigl(\tau_x f\bigl(\eta^n_{s}
\bigr) -\varphi_f(\rho) -\varphi '_f(
\rho) \bigl(\eta^n_{s}(x)-\rho \bigr) \bigr)h(x) \,ds
\biggr)^2 \biggr] =0.
\]
\end{proposition}

We now state a main result of this paper which provides a sharper
estimate, perhaps of independent interest, when starting from $\nu_\rho
$. To simplify expressions, we will use the notation
\[
\bigl(\eta^n_s\bigr)^{(\ell)}(x):=
\frac{1}{2\ell+1}\sum_{y\in\Lambda
_\ell}\eta^n_s(x+y).
\]

%th3.2 #&#
\begin{theorem}[($L^2$ generalized Boltzmann--Gibbs principle)]
\label{gbg_L2}
Let $f$ be a local $L^5(\nu_\rho)$ function supported on sites $\Lambda
_{\ell_0}$ such that $\varphi_f(\rho) =\varphi_f'(\rho)=0$.
There exists a constant $C=C(\rho, \ell_0)$
such that, for $t\geq0$, $\ell\geq\ell^3_0$ and $h\in\ell^1(\Z)\cap
\ell^2(\mathbb{Z})$,
\begin{eqnarray*}
&& \bb E_{\nu_\rho} \biggl[ \biggl( \int_0^t
\sum_{x\in{\mathbb{Z}}} \biggl(\tau _x f\bigl(
\eta^n_s\bigr) - \frac{\varphi_f''(\rho)}{2} \biggl\{ \bigl(
\bigl(\eta^n_s\bigr)^{(\ell
)}(x)-\rho
\bigr)^2-\frac{\sigma^2_\ell(\rho)}{2\ell+1} \biggr\} \biggr) h(x)\,ds
\biggr)^2 \biggr]
\\
& &\qquad\leq C\|f\|^2_{L^5(\nu_\rho)} \biggl(\frac{t \ell}{n} \biggl(
\frac
{1}{n}\sum_{x\in\Z}h^2(x)
\biggr) + \frac{t^2n^2}{\ell^{2+\alpha_0}} \biggl(\frac{1}{n}\sum
_{x\in{\mathbb{Z}}}\bigl|h(x)\bigr| \biggr)^2 \biggr).
\end{eqnarray*}

On the other hand, when only $\varphi_f(\rho)=0$ is known,
\begin{eqnarray*}
& &\bb E_{\nu_\rho} \biggl[ \biggl( \int_0^t
\sum_{x\in{\mathbb{Z}}} \bigl(\tau _x f\bigl(
\eta^n_s\bigr) - \varphi_f'(
\rho) \bigl\{ \bigl(\eta^n_s \bigr)^{(\ell)}(x) -
\rho \bigr\} \bigr) h(x)\,ds \biggr)^2 \biggr]
\\
&&\qquad \leq C\|f\|^2_{L^5(\nu_\rho)} \biggl(\frac{t \ell^2}{n} \biggl(
\frac
{1}{n}\sum_{x\in\Z}h^2(x)
\biggr) + \frac{t^2n^2}{\ell^{1+\alpha_0}} \biggl(\frac{1}{n}\sum
_{x\in{\mathbb{Z}}}\bigl|h(x)\bigr| \biggr)^2 \biggr).
\end{eqnarray*}
Here, $\alpha_0>0$ is the power in assumption \textup{(EE)}.
\end{theorem}

The proof of Theorem~\ref{gbg_L2} is given in Section~\ref{BG}. We
note, if the uniform spectral gap holds, $\sup_{k,\xi,n}\ell
^{-2}W(k,\ell,\xi,n)<\infty$, then the argument shows one can replace
in the right-hand sides above $\|f\|_{L^5(\nu_\rho)}$ with $\|f\|
_{L^3(\nu_\rho)}$.

%s3.3 #&#
\subsection{Tightness}
\label{tightness_section}
We prove tightness of the fluctuation fields, first starting from the
invariant measure $\nu_\rho$, using the $L^2$ generalized
Boltzmann--Gibbs principle. Then by the relative entropy bound (\ref
{relative_entropy}), we deduce tightness when beginning from initial
measures $\{\mu^n\}$.

%pr3.3 #&#
\begin{proposition}
\label{stationary_tightness}
The sequences $\{\Y^{n,\gamma}_t\dvtx t\in[0,T]\}_{n\geq1}$, $\{\M
^{n,\gamma}_t\dvtx t\in\break [0,T]\}_{n\geq1}$, $\{\I^{n,\gamma}_t\dvtx t\in[0,T]\}
_{n\geq1}$, $\{\B^{n,\gamma}_t\dvtx t\in[0,T]\}_{n\geq1}$, $\{\K
^{n,\gamma}_t\dvtx t\in[0,T]\}$ and $\{\langle\M^{n,\gamma}_t\rangle\dvtx
t\in[0,T]\}_{n\geq1}$, when starting from the invariant measure $\nu
_\rho$, are tight in the uniform topology on $D([0,T], \S'(\R))$.
\end{proposition}

\begin{pf}
By Mitoma's criterion \cite{Mitoma},
to prove tightness of the sequences with respect to the uniform
topology on $D([0,T], \S'(\R))$, it is enough to show tightness of $\{\Y
^{n,\gamma}_t(H); t\in[0,T]\}_{n\geq1}$, $\{\M^{n,\gamma}_t(H)\dvtx t\in
[0,T]\}_{n\geq1}$, $\{\I^{n,\gamma}_t(H)\dvtx t\in[0,T]\}_{n\geq1}$, $\{
\B^{n,\gamma}_t(H)\dvtx t\in[0,T]\}_{n\geq1}$, $\{\K^{n,\gamma}_t(H)\dvtx t\in[0,T]\}$ and $\{\langle\M^{n,\gamma}_t(H)\rangle\dvtx\break  t\in[0,T]\}
_{n\geq1}$, with respect to the uniform topology for all $H\in\S(\R
)$. Note that all initial values vanish, except $\Y^{n,\gamma}_0(H)$.

Tightness of $\Y^{n,\gamma}_t(H)$, in view of the decomposition $\Y
^{n,\gamma}_t(H) = \Y^{n,\gamma}_0(H) + \I^{n,\gamma}_t(H) + \B
^{n,\gamma}_t(H) + \K^{n,\gamma}_t(H)+\M^{n,\gamma}_t(H)$, will follow
from tightness of each term. The tightness of $\Y^{n,\gamma}_0(H)$,
given that we begin under $\nu_\rho$, follows from assumption (IM).

For the martingale term, we use Doob's inequality and stationarity to obtain
\begin{eqnarray*}
&&\P_{\nu_\rho} \Bigl( \mathop{\sup_{|t-s|\leq\delta}}_{0\leq s,t\leq
T} \bigl|\M^{n,\gamma}_t(H)
- \M^{n,\gamma}_s(H)\bigr|>\varepsilon \Bigr)
\\[-1pt]
&&\qquad \leq \varepsilon^{-4}\E_{\nu_\rho} \Bigl[ \mathop{\sup
_{|t-s|\leq\delta}}_{0\leq s,t\leq T} \bigl|\M^{n,\gamma}_t(H) - \M
^{n,\gamma}_s(H)\bigr|^4 \Bigr]
\\[-1pt]
&&\qquad \leq C \varepsilon^{-4}\delta^{-1} \E_{\nu_\rho}
\bigl[ \bigl(\M_\delta^{n,\gamma}(H) \bigr)^4 \bigr].
\end{eqnarray*}
Now, by the fourth moment estimate (\ref{fourth_moment_mart}), we have
\[
\delta^{-1} \E_{\nu_\rho} \bigl[ \bigl(\M_\delta^n(H)
\bigr)^4 \bigr] \leq C\|b\|_{L^4(\nu_\rho)}\bigl(\delta+
n^{-3/2}\bigr),
\]
which vanishes as $n\uparrow\infty$ and then $\delta\downarrow0$. This
is enough to conclude that $\{\M^{n,\gamma}_t(H)\dvtx t\in[0,T]\}_{n\geq
1}$ is tight in the uniform topology.

We now prove tightness for $\B^{n,\gamma}_t(H)$ through the
Kolmogorov--Centsov criterion. The argument for $\I^{n,\gamma}_t(H)$ is
similar. Also, the proofs for $\langle\M^{n,\gamma}_t(H)\rangle$ and
$\K^{n,\gamma}_t(H)$, given their forms, are simpler and can be done
using invariance of $\nu_\rho$ by squaring all terms. We focus on the
case $\gamma=1/2$, given that the estimates are analogous and simpler
when $1/2<\gamma\leq1$. Let
\[
V_b(\eta) = b(\eta) - \varphi_b(\rho) -
\varphi_b'(\rho) \bigl(\eta(0) -\rho\bigr).
\]
By assumption (R1), $V_b$ has range $R$. Also, by its form, $\varphi
_{V_b}(\rho) = \varphi'_{V_b}(\rho) = 0$ and also $\varphi''_{V_b}(\rho
)=\varphi''_b(\rho)$.

Then\vspace*{-1pt}
\[
\B^{n,\gamma}_t(H) = \frac{a}{2}\int_0^t
\sum_{x\in\Z} \bigl(\nabla ^n_x
H_{\gamma,s} \bigr) \tau_x V_b(
\eta_s)\,ds.\vspace*{-1pt}
\]
By invoking Theorem~\ref{gbg_L2} and translation-invariance of $\nu_\rho
$ which allows to replace $\nabla^n_x H_{\gamma,s}$ with $\nabla^n_x H$
(which does not depend on time $s$), for $\ell\geq\ell^3_0=R^3$, with
respect to a constant $C=C(a,\rho,R)$, we have
%
%e3.6 #&#
\begin{eqnarray}
\label{tightness_B}
&&\E_{\nu_\rho} \biggl[ \biggl(\B^{n,\gamma}_t(H)
\nonumber\\
&&\hspace*{16pt}\quad{} - \frac{a}{4}\int_0^t \sum
_{x\in\Z} \bigl(\nabla^n_x
H_{\gamma
,s}\bigr)\varphi''_{b}(
\rho) \biggl\{ \bigl( \bigl(\eta_s^n
\bigr)^{(\ell)}(x) - \rho \bigr)^2 - \frac{\sigma^2_\ell(\rho)}{2\ell+1} \biggr\}
\,ds \biggr)^2 \biggr]
\\
&&\qquad \leq C \|b\|^2_{L^4(\nu_\rho)} \biggl\{ \frac{t\ell}{n} +
\frac
{t^2n^2}{\ell^{2+\alpha_0}} \biggr\} \biggl[ \biggl(\frac{1}{n}\sum
_{x\in\Z} \bigl(\nabla^n_xH
\bigr)^2 \biggr) + \biggl(\frac{1}{n}\sum
_{x\in\Z}\bigl|\nabla^n_xH\bigr|
\biggr)^2 \biggr].
\nonumber
\end{eqnarray}

On the other hand, given $\sup_{\ell\geq R}E_{\nu_\rho}[(\sqrt{\ell
}(\eta^\ell- \rho))^4]<\infty$ by assumption (IM) and
$|\varphi_{b}''(\rho)|\leq C\|b\|_{L^2(\nu_\rho)}$ by assumption (D),
and the Schwarz inequality $(\sum_x h(x)r(x))^2 \leq(\sum_x |h(x)|)\sum_x |h(x)|r^2(x)$,
we have for $\ell>R^3$ that
\begin{eqnarray*}
&&\E_{\nu_\rho} \biggl[ \biggl( \int_0^t
\sum_{x\in\Z} \bigl(\nabla^n_xH_{\gamma
,s}
\bigr)\frac{\varphi''_{b}(\rho)}{2} \biggl\{ \bigl( \bigl(\eta^n_s
\bigr)^{(\ell
)}(x) - \rho \bigr)^2 -\frac{\sigma^2_\ell(\rho)}{2\ell+1} \biggr
\}\,ds \biggr)^2 \biggr]
\\
&&\qquad \leq C(\rho) \|b\|^2_{L^2(\nu_\rho
)}\frac{t^2 n^2}{\ell^2} \biggl(
\frac{1}{n}\sum_{x\in\Z} \bigl|\nabla
^n_xH\bigr| \biggr)^2.
\end{eqnarray*}
Hence, for $\ell> R^3$, we have
$\E_{\nu_\rho}  [(\B^{n,\gamma}_t(H))^2 ]  \leq C(a,\rho,R,H)
\|b\|^2_{L^4(\nu_\rho)} [t\ell/n + t^2n^2/\ell^2]$, noting the
domination $n^2/\ell^{2+\alpha_0}\leq n^2/\ell^2$. Then, if $\ell$ is
taken as $\ell= t^{1/3}n>R^3$, we conclude $\E_{\nu_\rho} [(\B
^{n,\gamma}_t(H))^2 ] \leq C(a,\rho,R,H) \|b\|^2_{L^4(\nu_\rho)}t^{4/3}$.

However, when $t^{1/3}n \leq R^3$, we have by the same Schwarz bound
that
\begin{eqnarray*}
\E_{\nu_\rho} \bigl[\bigl(\B^{n,\gamma}_t(H)
\bigr)^2 \bigr] & \leq& C(\rho, a)\|b\| ^2_{L^2(\nu_\rho)}
t^2 n^2 \biggl(\frac{1}{n}\sum
_x \bigl|\nabla_x^n H\bigr|
\biggr)^2
\\
&\leq &C(\rho, a, H,R)\|b\|^2_{L^2(\nu_\rho)} t^{4/3}.
\end{eqnarray*}
This shows tightness of $\B_t^{n,\gamma}(H)$.

Combining these estimates, we conclude the proof of the proposition.
\end{pf}

We now update to when the process begins from the measures $\{\mu^n\}$.

%pr3.4 #&#
\begin{proposition}
\label{nonstat_tightness}
The fluctuation field sequences $\{\Y^{n,\gamma}_t\dvtx t\in[0,T]\}_{n\geq
1}$, $\{\M^{n,\gamma}_t\dvtx t\in[0,T]\}_{n\geq1}$, $\{\I^{n,\gamma}_t\dvtx t\in[0,T]\}_{n\geq1}$,
$\{\B^{n,\gamma}_t\dvtx t\in[0,T]\}_{n\geq1}$, $\{\K^{n,\gamma}_t\dvtx t\in
[0,T]\}_{n\geq1}$ and $\{\langle\M^{n,\gamma}_t\rangle\dvtx  t\in[0,T]\}
_{n\geq1}$ are tight in the uniform topology on $D([0,T], \S'(\R))$
when starting from $\{\mu^n\}$ satisfying assumption \textup{(BE)}.
\end{proposition}

\begin{pf}
As before, all initial values vanish except $\Y_0^{n,\gamma}$ which,
however, is tight by (CLT). Next, by Proposition~\ref
{stationary_tightness}, we have $\lim_{\delta\downarrow0}\lim_{n\uparrow\infty}
\P_{\nu_\rho}(O^n_{\delta,\varepsilon})=0$ where
\[
O^n_{\delta, \varepsilon} = \Bigl\{ \mathop{\sup_{|t-s|\leq\delta
}}_{s,t\in[0,T]}
\bigl\|X^n_t - X^n_s\bigr\|>\varepsilon
\Bigr\},
\]
and $X^n_t$ may be equal to $\Y^{n,\gamma}_t$, $\M^{n,\gamma}_t$, $\I
^{n,\gamma}_t$, $\B^{n,\gamma}_t$, $\K^{n,\gamma}_t$ or $\langle\M
^{n,\gamma}_t\rangle$. Then
we have by the entropy inequality (\ref{relative_entropy}) that also
$\lim_{\delta\downarrow0} \lim_{n\uparrow\infty}\P_{\mu^n}(O_{\delta
,\varepsilon}) = 0$
which allows to conclude.
\end{pf}

%s3.4 #&#
\subsection{Identification of limit points: Proofs of Theorems \texorpdfstring{\protect\ref{th crossover}}{2.2} and \texorpdfstring{\protect\ref{th kpz scaling}}{2.3}}\label{identification}

With tightness (Proposition~\ref{nonstat_tightness}) in hand, we now
identify the limit points of $\{\Y^{n,\gamma}_t\dvtx t\in[0,T]\}_{n\geq
1}$ and its parts in decomposition \eqref{mart_decomposition}. Let
$Q^n$ be the distribution of
\[
\bigl(\Y^{n,\gamma}_t, \M^{n,\gamma}_t,
\I^{n,\gamma}_t, \B^{n,\gamma
}_t,
\K^{n,\gamma}_t, \bigl\langle\M^{n,\gamma}_t\bigr
\rangle\dvtx t\in[0,T] \bigr),
\]
and let $n'$ be a subsequence where $Q^{n'}$ converges to a limit point
$Q$. Let also $\Y_t$, $\M_t$, $\I_t$, $\B_t$, $\K_t$ and $\D_t$ be the
respective limits in distribution of the components.
Since tightness is shown in the uniform topology on $D([0,T], \S'(\R
))$, we have that $\Y_t$, $\M_t$, $\I_t$, $\B_t$, $\K_t$ and $\D_t$
have a.s. continuous paths.

Let now $G_\varepsilon\dvtx \bb R \rightarrow[0,\infty)$ be a smooth
compactly supported function for $0<\varepsilon\leq1$ which
approximates $\iota_\varepsilon(z) = \varepsilon
^{-1}1_{[-1,1]}(z\varepsilon^{-1})$ as in the definition of energy
solution before Theorem~\ref{th kpz scaling}. That is, $\|G_\varepsilon
\|^2_{L^2(\R)}\leq2\|\iota_\varepsilon\|^2_{L^2(\R)}=\varepsilon^{-1}$
and $\lim_{\varepsilon\downarrow0}\varepsilon^{-1/2}\|G_\varepsilon-
\iota_\varepsilon\|_{L^2(\R)}=0$.
Define
\[
\A^{n,\gamma,\varepsilon}_{s,t}(H):= \int_s^t
\frac{1}{n}\sum_{x\in\Z} \bigl(
\nabla^n_x H\bigr) \bigl[\tau_x
\Y^{n,\gamma}_u(G_\varepsilon) \bigr]^2\,du.
\]
Since for fixed $0<\varepsilon\leq1$ the map $\pi_\cdot\mapsto\int_s^t \,du \int \,dx (\nabla H(x) ) \{\pi_u(\tau_{-x} G_\varepsilon
) \}^2$ is continuous in the uniform topology on $D([0,T];\S'(\R))$,
we have subsequentially in distribution that
\[
\lim_{n'\uparrow\infty}\A^{n',\gamma,\varepsilon}_{s,t}(H) = \int
_s^t \,du\int \,dx \bigl(\nabla H(x) \bigr) \bigl
\{\Y_u(\tau_{-x} G_\varepsilon ) \bigr
\}^2 =: \A^{\varepsilon}_{s,t}(H).
\]

%pr3.5 #&#
\begin{proposition}
\label{stat_lemma1} Suppose the initial distribution is the invariant
measure $\nu_\rho$ and $t\in[0,T]$.

When $\gamma= 1/2$, there is a constant $C=C(a,\rho,R)$ such that
\begin{eqnarray*}
&&\lim_{n\uparrow\infty}\E_{\nu_\rho} \biggl[ \biggl|\B^{n,\gamma}_t(H)
- \frac{a\varphi_b''(\rho)}{4}\A^{n,\gamma,\varepsilon}_{0,t}(H) \biggr|^2 \biggr]
\\
&&\qquad \leq Ct \bigl(\varepsilon+ \varepsilon^{-1}\|G_\varepsilon -
\iota_\varepsilon\|^2_{L^2(\R)} \bigr) \|b
\|^2_{L^4(\nu_\rho)} \bigl[\| \nabla H\|^2_{L^2(\R)}
+ \|\nabla H\|^2_{L^1(\R)} \bigr].
\end{eqnarray*}
Then, in $L^2(\P_{\nu_\rho})$,
$A^\varepsilon_{0,t}(H)$ is a Cauchy $\varepsilon$-sequence. Hence,
\[
\frac{a \varphi_b''(\rho)}{4} \A_{0,t}(H):= \lim_{\varepsilon
\downarrow0}
\frac{a \varphi_b''(\rho)}{4}\A^\varepsilon_{0,t}(H) = \B_t(H).
\]
In particular, we conclude $\A_{s,t}(H) \stackrel{d}{=} \A_{0,t-s}(H)$
does not depend on the specific smoothing family $\{G_\varepsilon\}$.
Moreover, when $1/2<\gamma\leq1$, we have $\B_t(H) = 0$.

In addition, when $1/2\leq\gamma\leq1$,
\begin{eqnarray*}
\lim_{n\uparrow\infty} \E_{\nu_\rho} \biggl[ \biggl|
\I^{n,\gamma}_t(H) - \frac{\varphi_c'(\rho)}{2}\int_0^t
\Y^{n,\gamma}_s(\Delta H)\,ds\biggr |^2 \biggr]& =& 0,
\\
\lim_{n\uparrow\infty} \E_{\nu_\rho} \biggl[ \biggl| \bigl\langle
\M^{n,\gamma
}_t(H)\bigr\rangle- \frac{\varphi_b(\rho)}{2}t\|\nabla H
\|^2_{L^2(\R)} \biggr|^2 \biggr] &=& 0,
\\
\lim_{n\uparrow\infty} \E_{\nu_\rho} \bigl[\bigl |
\K^{n,\gamma}_t(H) \bigr|^2 \bigr] &=& 0.
\end{eqnarray*}
Then, in $L^2(\P_{\nu_\rho})$, $\K_t(H)= 0$ and
\[
\I_t(H) = \frac{\varphi_c'(\rho)}{2}\int_0^t
\Y_s(\Delta H)\,ds \quad\mbox{and}\quad D_t(H) = \frac{\varphi_b(\rho)}{2}t
\|\nabla H\| ^2_{L^2(\R)}.
\]
\end{proposition}

\begin{pf} We prove the limit display for $\B_t(H)$ when $\gamma
=1/2$ which shows, by a Fatou's lemma that
$\E_{\nu_\rho} [ |\B_t(H) - (a\varphi''_b(\rho)/4)\A^\varepsilon
_{0,t}(H) |^2 ]\leq C(a,\rho,R, H)\*t [\varepsilon+ \varepsilon
^{-1}\|G_\varepsilon- \iota_\varepsilon\|^2_{L^2(\R)} ]$.
Therefore, $\A^\varepsilon_{0,t}(H)$, as a sequence in $\varepsilon$,
is Cauchy in $L^2(\P_{\nu_\rho})$.
The arguments for $\I_t(H)$, $\D_t(H)$, and $\K_t(H)$, noting their
forms, are similar; for $\D_t(H)$ and $\K_t(H)$, one might also use
spatial mixing assumed in (IM). To simplify notation, we will call $n=n'$.

Note, for $\ell=\varepsilon n$, that
\begin{eqnarray*}
&&\sum_{x\in\Z}\bigl(\nabla^n_x
H_{\gamma,s}\bigr) \bigl( \bigl(\eta^n_s
\bigr)^{(\ell
)}(x) - \rho \bigr)^2
\\
&&\qquad = \sum_{x\in\Z}\bigl(\nabla^n_x
H_{\gamma,s}\bigr) \biggl(\frac
{1}{2n\varepsilon+1}\sum
_{|z|\leq n\varepsilon}\bigl(\eta^n_s(z+x) - \rho
\bigr) \biggr)^2
\\
&& \qquad= \frac{1+O(n^{-1})}{n}\sum_{x\in\Z} \bigl(
\nabla^n_x H\bigr) \bigl[\tau_x
\Y^{n,\gamma}_s(\iota_\varepsilon) \bigr]^2.
\end{eqnarray*}
Here, the shift by $n^{-1}\lfloor a\varphi'_b(\rho)sn^2/(2n^\gamma
)\rfloor$ in $\nabla^n_x H_{\gamma,s}$ [cf. \eqref{shifted_H}] was
transferred to $\tau_x\Y^{n,\gamma}_s(\iota_\varepsilon)$.

Then, with $\ell= \varepsilon n$, by Theorem~\ref{gbg_L2}, as in the
bound \eqref{tightness_B}, we have
\begin{eqnarray*}
&&\lim_{n\uparrow\infty}\E_{\nu_\rho} \biggl[ \biggl(
\B^{n,\gamma}_t(H) - \frac{a\varphi_{b^n}''(\rho)}{4}\int_0^t
\frac{1}{n} \sum_{x\in\Z} \bigl(
\nabla^n_x H\bigr)\tau_x
\Y^{n,\gamma}_s(\iota_\varepsilon)^2 \,ds
\biggr)^2 \biggr]
\\
&&\qquad = \lim_{n\uparrow\infty}\E_{\nu_\rho} \biggl[ \biggl(
\B_t^{n,\gamma
}(H)
\\
&&\hspace*{46pt}\qquad\quad{} - \frac{a\varphi_{b^n}''(\rho)}{4}\int_0^t
\frac
{1}{n} \sum_{x\in\Z} \bigl(
\nabla^n_x H\bigr)\tau_x \biggl\{
\Y^{n,\gamma}_s(\iota _\varepsilon)^2 -
\frac{\sigma^2_\ell(\rho)}{2\varepsilon} \biggr\}\,ds \biggr)^2 \biggr]
\\
&&\qquad \leq \lim_{n\uparrow\infty}C(a,\rho,R)\bigl\|b^n
\bigr\|^2_{L^4(\nu_\rho
)}t \biggl( \varepsilon+ \frac{1}{\varepsilon^{2+\alpha_0} n^{\alpha
_0}} \biggr)
\\
&& \qquad\quad{}\times \biggl[ \biggl(\frac{1}{n}\sum_{x\in\Z}
\bigl(\nabla^n_x H \bigr)^2 \biggr) +
\biggl(\frac{1}{n}\sum_{x\in\Z} \bigl|\nabla
^n_x H \bigr| \biggr)^2 \biggr].
\end{eqnarray*}
Here, as the sum of $\nabla^n_x H_{\gamma,s}$ on $x$ vanishes, we
introduced the centering constant $(2\varepsilon)^{-1}\sigma^2_\ell(\rho
)$ in the second line.

Now,
\[
\Y^{n,\gamma}_s(\iota_\varepsilon)^2 -
\Y^{n,\gamma}_x(G_\varepsilon)^2 = \bigl[
\Y^{n,\gamma}_s(\iota_\varepsilon) - \Y^{n,\gamma
}_s(G_\varepsilon)
\bigr] \bigl[\Y^{n,\gamma}_s(\iota_\varepsilon) + \Y
^{n,\gamma}_s(G_\varepsilon) \bigr]
\]
and by (IM2)
\[
\C_{\nu_\rho} (\iota_\varepsilon- G_\varepsilon,
\iota_\varepsilon - G_\varepsilon )^{1/2} \C_{\nu_\rho} (
\iota_\varepsilon+ G_\varepsilon, \iota_\varepsilon+ G_\varepsilon
)^{1/2} \leq C(\rho)\varepsilon^{-1/2}\|G_\varepsilon-
\iota_\varepsilon\|_{L^2(\R)}.
\]
Hence, by Schwarz inequality,
\begin{eqnarray*}
&&\lim_{n\uparrow\infty} \E_{\nu_\rho} \biggl[ \biggl(\int
_0^t\frac{1}{n} \sum
_{x\in\Z} \bigl(\nabla^n_x H\bigr)
\tau_x\Y^{n,\gamma}_s(\iota_\varepsilon)^2
\,ds - \A^{n,\gamma,\varepsilon}_{0,t}(H) \biggr)^2 \biggr]
\\
&&\qquad \leq C(\rho)\varepsilon^{-1}\|G_\varepsilon-
\iota_\varepsilon\| ^2_{L^2(\R)}t^2 \biggl(
\frac{1}{n}\sum_{x\in\Z} \bigl|\nabla^n_x
H \bigr| \biggr)^2.
\end{eqnarray*}

Finally, combining these estimates with the inequality $(a+b)^2 \leq
2a^2 + 2b^2$, and by assumption (D) that $\lim_{n\uparrow\infty}\varphi
''_{b^n}(\rho) =\varphi''_b(\rho)$, we complete the proof.
\end{pf}

%pr3.6 #&#
\begin{proposition}
\label{nonstat_lemma1}
Suppose the initial measures $\{\mu^n\}$ satisfy assumption \textup{(BE)}, and
$t\in[0,T]$.

When $\gamma=1/2$, we have $\A^\varepsilon_{0,t}(H)$ is a Cauchy
$\varepsilon$-sequence in probability with respect to a limit measure
$Q$, and hence
\[
\frac{a\varphi''_b(\rho)}{4}\A_{0,t}(H):= \lim_{\varepsilon
\downarrow0}
\frac{a \varphi_b''(\rho)}{4}\A_{0,t}^\varepsilon(H) = \B_t(H).
\]
On the other hand, when $1/2<\gamma\leq1$, we have $\B_t(H) \equiv0$.

When $1/2\leq\gamma\leq1$, we have $\K^n_t(H)\equiv0$,
\[
\I_t(H) = \frac{\varphi_c'(\rho)}{2}\int_0^t
\Y_s(\Delta H)\,ds \quad\mbox{and}\quad \D_t(H) =
\frac{\varphi_b(\rho)}{2}t\|\nabla H\|^2_{L^2(\R)}.
\]
\end{proposition}

\begin{pf}
By assumption (BE), and lower semicontinuity of entropy, the limit
measure $Q$ also has bounded entropy with respect to $\P_{\nu_\rho}$,
$H(Q;\P_{\nu_\rho})<\infty$. When $\gamma=1/2$, by the $L^2(\P_{\nu_\rho
})$ statements in Proposition~\ref{stat_lemma1} and the entropy
inequality \eqref{relative_entropy}, we have for $\delta>0$ that
$\lim_{\varepsilon\downarrow0}Q ( |\B_t(H)-(a\varphi''_b(\rho
)/\break 4)\A^\varepsilon_t(H) |>\delta ) = 0$, and so
$\A^\varepsilon_t(H)$ is Cauchy in probability with respect to $Q$.
Therefore, $\lim_{\varepsilon\downarrow0} (a\varphi''_b(\rho)/4)\A
^\varepsilon_t(H) = \B_t(H)$.

The other claims follow similarly.
\end{pf}

\begin{pf*}{Proof of Theorems \ref{th crossover} and \ref{th kpz scaling}}
Let $H\in\S(\R)$, $t\in[0,T]$, and suppose the initial measures are
$\{\mu^n\}$. When $\gamma=1/2$, by the decomposition \eqref
{mart_decomposition}, Proposition~\ref{nonstat_lemma1}, and tightness
of the constituent processes $\Y_t^{n,\gamma}$, $\M^{n,\gamma}_t$, $\I
^{n,\gamma}_t$, $\B^{n,\gamma}_t$, $\K^{n,\gamma}_t$ and $\langle\M
^{n,\gamma}_t\rangle$ in the uniform topology, any limit point of
\[
\bigl(\Y^{n,\gamma}_t, \M^{n,\gamma}_t,
\I^{n,\gamma}_t, \B^{n,\gamma
}_t,
\K^{n,\gamma}_t, \bigl\langle\M^{n,\gamma}_t\bigr
\rangle\dvtx t\in[0,T] \bigr)
\]
satisfies
\[
\M_t(H) = \Y_t(H) - \Y_0(H) -
\frac{\varphi_c'(\rho)}{2}\int_0^t\Y _s(
\triangle H)\,ds -\bigl(a\varphi_b''(
\rho)/2\bigr)\A_t(H).
\]
However, when $1/2<\gamma\leq1$,
%
%e3.7 #&#
\begin{equation}
\label{proof_thm2.4} \M_t(H) = \Y_t(H) - \Y_0(H)
- \frac{\varphi_c'(\rho)}{2}\int_0^t \Y
_s(\triangle H)\,ds.
\end{equation}

Also, in both cases, $\Y_0(H) = \bar\Y_0(H)$ by assumption (CLT).

We also claim in both cases that $\M_t(H)$ is a continuous martingale
with a quadratic variation
\[
\bigl\langle\M_t(H)\bigr\rangle = \frac{\varphi_b(\rho)}{2}t\|\nabla H\|
^2_{L^2(\R)}.
\]
Indeed, by Proposition~\ref{nonstat_lemma1}, any limit point of the
quadratic variation sequence equals $\mc D_t(H) = (\varphi_b(\rho)/2)t\|
\nabla H\|^2_{L^2(\R)}$. Next, $\M_t(H)$ as the limit of martingales
with respect to the uniform topology is a continuous martingale. Also,
by the triangle inequality, Doob's inequality and the quadratic
variation bound \eqref{quad_var_bound},
\begin{eqnarray*}
&&\sup_n \E_{\nu_\rho} \Bigl[\sup_{0\leq s\leq t}
\bigl|\M^{n,\gamma}_s(H) - \M ^{n,\gamma}_{s^-}(H)\bigr|
\Bigr] \\
&&\qquad \leq 2\sup_n \E_{\nu_\rho} \Bigl[ \sup
_{u\in[0,t]}\bigl|\M^{n,\gamma
}_u(H)\bigr|^2
\Bigr]^{{1}/{2}}
\\
&&\qquad\leq 2\sup_n \E_{\nu_\rho} \bigl[\bigl\langle
M^{n,\gamma}_t(H)\bigr\rangle \bigr]^{{1}/{2}}  \leq
C(a,T)\|b\|_{L^1(\nu_\rho)}\|\nabla H\| ^2_{L^2(\R)}.
\end{eqnarray*}
Then, by Corollary VI.6.30 of
\cite{JS}, $(\M_t^{n,\gamma}(H), \langle\M^{n,\gamma}_t(H)\rangle)$
converges on a subsequence in distribution to $(\M_t(H), \langle\M
_t(H)\rangle)$. Since, also $\langle\M^{n,\gamma}_t(H)\rangle$
converges on a subsequence in distribution to $\D_t(H)=(\varphi_b(\rho
)/2)t\|\nabla H\|^2_{L^2(\R)}$, we have $\langle\M_t(H)\rangle=
(\varphi_b(\rho)/2)t\|\nabla H\|^2_{L^2(\R)}$.

By Proposition~\ref{nonstat_lemma1}, when $\gamma=1/2$, $\Y_t$ is a
``probability energy solution'' corresponding to the stochastic Burgers
equation \eqref{BE}. But, if initially $\mu^n\equiv\nu_\rho$, by
Proposition~\ref{stat_lemma1}, $\Y_t$ is an ``$L^2$ energy solution.''
This completes the proof of Theorem~\ref{th kpz scaling}.

However, when $1/2<\gamma\leq1$, by the form of $\M_t(H)$ in \eqref
{proof_thm2.4}, we conclude $\Y_t(H)$ solves the Ornstein--Uhlenbeck
equation \eqref{OU2}. By uniqueness, all subsequences converge to the
same limit, and we obtain Theorem~\ref{th crossover}.
\end{pf*}

%s4 #&#
\section{Proof of the generalized Boltzmann--Gibbs principle}\label{BG}

We start by recalling the notion of $H_{1,n}$ and $H_{-1,n}$ spaces \cite{KV}.
For $n\geq n_0$, recall $S_n = (L_n + L^*_n)/2$ [cf. near \eqref
{dirichlet_form}], and define the $H_{1,n}$ seminorm $\|\cdot\|_{1,n}$
on $L^2(\nu_\rho)$ functions by
\[
\|f\|_{1,n}^2:= E_{\nu_\rho} \bigl[f(-S_n)f
\bigr] = n^2D_{\nu_\rho}(f).
\]
The Hilbert space $H_{1,n}$ is then the completion of functions with
finite $H_{1,n}$ norm modulo norm-zero functions. In particular, local
bounded functions are dense in~$H_{1,n}$.

Correspondingly, one can define the dual seminorm $\|\cdot\|_{-1,n}$
with respect to the $L^2(\nu_\rho)$ inner-product by
\[
\|f\|_{-1,n}:= \sup \biggl\{ \frac{E_{\nu_\rho}[f\phi]}{\|\phi\|
_{1,n}}\dvtx \phi\neq0 \mbox{ local, bounded} \biggr\},
\]
and the Hilbert space $H_{-1,n}$ which is the completion over those
functions with finite $\|\cdot\|_{-1,n}$ norm modulo norm-zero functions.

We now state a helping lemma for the results in this section. Define
the restricted Dirichlet form on local, bounded functions with respect
to the grand canonical measure $\nu_\rho$ as
\[
D_{\nu_\rho, \ell}(\phi) = \sum_{x,x+1\in\Lambda_\ell}E_{\nu_\rho
}
\bigl[ b^{R,n}_x(\eta) \bigl(\nabla_{x,x+1}\phi(
\eta) \bigr)^2 \bigr].
\]

Recall the collection $\eta^c_r:= \{\eta(x)\dvtx x\notin\Lambda_r\}$.

%pr4.1 #&#
\begin{proposition} \label{crucial estimate}
Let $r\dvtx \Omega\to\mathbb R$ be an $L^4(\nu_\rho)$ function and $\ell
_0\geq2$.
Suppose that
$E_{\nu_\rho}[r|\eta^{(\ell_0)}, \eta^c_{\ell_0}]=0$ a.s. Then, for
local, bounded functions $\phi$, we have
\[
\bigl|E_{\nu_{\rho}}\bigl[ r(\eta) \phi(\eta)\bigr] \bigr| \leq E_{\nu_\rho}
\biggl[ W \biggl(\sum_{x\in\Lambda_{\ell_0}}\eta(x),
\ell_0,\eta_{\ell
_0}^c,n \biggr)^2
\biggr]^{1/4}\|r\|_{L^4(\nu_\rho)} D_{\nu_\rho,\ell
_0}^{1/2}(
\phi).
\]
\end{proposition}

\begin{pf}
Recall, from Section~\ref{notation}, for $k\geq0$, $\ell_0\geq2$
and $\xi\in\Omega$, the space
\[
\G_{k,\ell_0,\xi} = \biggl\{\eta: \sum_{x\in\Lambda_{\ell_0}} \eta
(x) = k, \eta(y) = \xi(y) \mbox{ for }y\notin\Lambda_{\ell_0} \biggr\}
\]
and generator $S_{n,\G}:=S_{n,\G_{k,\ell_0,\xi}}$\vspace*{1pt} which governs the
evolution of the symmetrized process on $\G_{k, \ell_0, \xi}$. Suppose
$W(k,\ell_0,\xi,n)<\infty$ so that the measure $\nu_{k,\ell_0,\xi}$ is
the unique invariant measure for the process.

Given $E_{\nu_\rho}[r|\sum_{|x|\leq\ell_0}\eta(x)=k, \eta(y)=\xi(y)
\mbox{ for }y\notin\Lambda_{\ell_0}]=E_{\nu_{k,\ell_0,\xi}}[r]=0$,
we have $r$ restricted to $\G_{k,\ell_0, \xi}$ is orthogonal to
constant functions and therefore belongs to the range of $-S_{n,\G}$,
that is the equation $r = -S_{n,\G} u$ can be solved for some function
$u\dvtx \G_{k,\ell_0,\xi} \rightarrow\R$.

Now, with $k=\sum_{x\in\Lambda_{\ell_0}}\eta(x)$ and $\xi= \eta
^c_{\ell_0}$, $W(k,\ell_0,\eta^c_{\ell_0},n)<\infty$ a.s. by assumption
(G). Hence,
\begin{eqnarray*}
\bigl|E_{\nu_\rho}[r\phi] \bigr| &=& \bigl|E_{\nu_\rho} \bigl[ E_{\nu_\rho
}
\bigl[r\phi|\eta^{(\ell_0)}, \eta_{\ell_0}^c\bigr] \bigr]\bigr |
\\
&=& \bigl|E_{\nu_\rho} \bigl[ E_{\nu_\rho}\bigl[ (-S_{n,\G} u)
\phi|\eta^{(\ell
_0)}, \eta_{\ell_0}^c\bigr] \bigr] \bigr|
\\
&\leq& E_{\nu_\rho} \bigl[ E_{\nu_\rho}\bigl[u (-S_{n,\G} u)|
\eta^{(\ell_0)}, \eta_{\ell_0}^c\bigr]^{1/2}
 E_{\nu_\rho}\bigl[\phi(-S_{n,\G} \phi )|
\eta^{(\ell_0)}, \eta_{\ell_0}^c\bigr]^{1/2}
\bigr].
\end{eqnarray*}
The last line follows as $-S_{n,\G}$ is a nonnegative symmetric
operator and, therefore, has a square root.

Further, since $W(k,\ell_0, \xi,n)$ is the reciprocal of the spectral
gap for $-S_{n,\G}$, we have
\[
E_{\nu_\rho} \bigl[ru|\eta^{(\ell_0)},\eta^c_{\ell_0}
\bigr] \leq W\bigl(k,\ell_0, \eta_{\ell_0}^c,n
\bigr)E_{\nu_\rho} \bigl[r^2|\eta^{(\ell_0)},\eta
^c_{\ell_0} \bigr].
\]
Therefore, we conclude
\[
\bigl|E_{\nu_\rho}[r\phi]\bigr |\leq E_{\nu_\rho} \biggl[ W \biggl(\sum
_{x\in
\Lambda_{\ell_0}}\eta(x),\ell_0,
\eta_{\ell_0}^c,n \biggr) E_{\nu_\rho
}
\bigl[r^2|\eta^{(\ell_0)}, \eta_{\ell_0}^c
\bigr] \biggr]^{1/2} D^{1/2}_{\nu_\rho
,\ell_0}(\phi).
\]
The desired bound now follows from Schwarz inequality.\vadjust{\goodbreak}
\end{pf}

The following bound on the variance of additive functionals is the main
way we control the fluctuations of several quantities in the sequel. A
proof of Proposition~\ref{KV} can be found in Appendix 1.6 of \cite{KL}.

To simplify notation, for the rest of the section, we will drop the
superscript ``$n$'' and write $\eta^n = \eta$.

%
%pr4.2 #&#
\begin{proposition}
\label{KV}
Let $r\dvtx \Omega\to\mathbb R$ be a mean-zero $L^2(\nu_\rho)$ function,
\mbox{$\varphi_r(\rho)=0$}. Then\vspace*{-1pt}
\[
\mathbb{E}_{\nu_\rho} \biggl[ \biggl(\int_0^t
r(\eta_s) \,ds \biggr)^2 \biggr] \leq 20 t \|r
\|_{-1,n}^2.
\]
\end{proposition}

The proof of Theorem~\ref{gbg_L2}, given at the end of the section, is
made through a succession of steps, labeled ``one-block,''
``renormalization step,'' ``two-blocks'' and ``equivalence of ensembles'' estimates.

%le4.3 #&#
\begin{lemma}[(One-block estimate)]
\label{globalone_block}
Let $f\dvtx \Omega\to\mathbb R$ be a local $L^4(\nu_\rho)$ function
supported on sites in $\Lambda_{\ell_0}$ such that
$\varphi_f(\rho)=0$. Then there exists a constant $C=C(\rho)$ such
that for $\ell\geq\ell_0$, $t \geq0$ and $h\in{\ell^1(\Z)\cap\ell
^2(\mathbb{Z})}$:
\begin{eqnarray*}
&&\bb E_{\nu_\rho} \biggl[ \biggl(\int_0^t
\sum_{x\in\mathbb{Z}}h(x)\tau_x \bigl\{f(
\eta_s)-E_{\nu_{\rho}}\bigl[f(\eta_s)|
\eta_s^{(\ell)}, (\eta_s)_\ell
^c\bigr] \bigr\} \,ds \biggr)^2 \biggr]
\\
&&\qquad \leq Ct\frac{\ell^3}{n^2} \|f\|^2_{L^4(\nu
_\rho)}\sum
_{x\in{\mathbb{Z}}}h^2(x).
\end{eqnarray*}
\end{lemma}

\begin{pf}
By Proposition~\ref{KV}, we need only to estimate the $H_{-1,n}$ norm
of the integrand [which is in $L^2(\nu_\rho)$ since $h\in\ell^1(\Z)$].
Bound the $H_{-1,n}$ norm multiplied by $n$, using Proposition~\ref
{crucial estimate}, as follows:\vspace*{-1pt}
%
%e4.1 #&#
\begin{eqnarray}
\label{global_eq_1} &&\sup_\phi D^{-{1}/{2}}_{\nu_\rho}(
\phi) E_{\nu_\rho} \biggl[\sum_{x\in\mathbb{Z}}h(x)
\tau_x \bigl\{f-E_{\nu_{\rho}}\bigl[f|\eta^{(\ell)}, \eta
^c_{\ell}\bigr] \bigr\} \phi \biggr]
\nonumber
\\
&&\qquad = \sup_\phi\sum_{x\in\Z}
D^{-{1}/{2}}_{\nu_\rho}(\phi)E_{\nu
_\rho} \bigl[h(x)
\tau_x \bigl( f - E_{\nu_\rho}\bigl[f|\eta^{(\ell)}, \eta
^c_{\ell}\bigr] \bigr) \phi \bigr]
\\
&&\qquad \leq \sup_\phi D^{-{1}/{2}}_{\nu_\rho}(\phi)
\nonumber\\
&&\qquad\quad{} \times\sum_{x\in\Z} \bigl|h(x)\bigr| E_{\nu_\rho}
\biggl[ W \biggl(\sum_{x\in\Lambda_\ell}\eta(x), \ell,
\eta_\ell^c,n \biggr)^2 \biggr]^{{1}/{4}}
\|f\|_{L^4(\nu_\rho)} D^{{1}/{2}}_{\nu_\rho,\ell}(\tau_{-x} \phi).
\nonumber
\end{eqnarray}

Observe now, by translation-invariance of $\nu_\rho$, that
\[
\sum_{x\in\Z} D_{\nu_\rho,\ell}(\tau_{-x}
\phi) \leq (2\ell+1) D_{\nu_\rho}(\phi).
\]
Then, noting the spectral gap assumption (G), and using the relation
$2ab = \inf_{\kappa>0}[a^2\kappa+ \kappa^{-1}b^2]$, we bound \eqref
{global_eq_1} by
\begin{eqnarray*}
&&\sup_\phi D^{-{1}/{2}}_{\nu_\rho}(\phi)\inf
_{\kappa>0} \biggl\{ \kappa C\ell^2\|f
\|^2_{L^4(\nu_\rho)}\sum_{x\in\Z}
h^2(x) + \kappa ^{-1}C\ell D_{\nu_\rho}(\phi) \biggr\}
\\
&& \qquad\leq \biggl(C\ell^3 \|f\|_{L^4(\nu_\rho)}^2\sum
_{x\in\Z
} h^2(x) \biggr)^{{1}/{2}},
\end{eqnarray*}
where $C=C(\rho)$ is a constant. This completes the proof.
\end{pf}

Now we double the size of the box in the conditional expectation.

%le4.4 #&#
\begin{lemma}[(Renormalization step)]
\label{globalrenormalization}
Let $f\dvtx \Omega\to\mathbb R$ be a local $L^5(\nu_\rho)$ function
supported on sites in $\Lambda_{\ell_0}$
such that $\varphi_f(\rho)=\varphi_f'(\rho)=0$. There exists a constant
$C=C(\rho, \ell_0)$ such that for $\ell\geq\ell_0$, $t \geq0$ and
$h\in{\ell^1(\Z)\cap\ell^2(\mathbb{Z})}$:
\begin{eqnarray*}
&&\bb E_{\nu_\rho} \biggl[ \biggl(\int_0^t
\sum_{x\in{\mathbb{Z}}}\tau_x \bigl\{
E_{\nu_{\rho}}\bigl[f(\eta_s)|\eta_s^{(\ell)},
(\eta_s)_\ell^c\bigr]
\\
&&\hspace*{63pt}\quad{} -E_{\nu_{\rho}}\bigl[f(\eta _s)|\eta_s^{(2\ell)},
(\eta_s)_{2\ell}^c\bigr] \bigr\}h(x) \,ds
\biggr)^2 \biggr]
\\
&&\qquad \leq C\|f\|^2_{L^5(\nu_\rho)}t\frac{\ell
}{n^2}\sum
_{x\in{\mathbb{Z}}}h^2(x).
\end{eqnarray*}

On the other hand, when only $\varphi_f(\rho)=0$ is known,
\begin{eqnarray*}
&&\bb E_{\nu_\rho} \biggl[ \biggl(\int_0^t
\sum_{x\in{\mathbb{Z}}}\tau_x \bigl\{
E_{\nu_{\rho}}\bigl[f(\eta_s)|\eta_s^{(\ell)},
(\eta_s)_\ell^c\bigr]
\\
&&\hspace*{63pt}\quad{} -E_{\nu_{\rho}}\bigl[f(\eta _s)|\eta_s^{(2\ell)},
(\eta_s)_{2\ell}^c\bigr] \bigr\}h(x) \,ds
\biggr)^2 \biggr]
\\
&&\qquad \leq C\|f\|^2_{L^5(\nu_\rho)}t\frac{\ell
^2}{n^2}\sum
_{x\in{\mathbb{Z}}}h^2(x).
\end{eqnarray*}
\end{lemma}

\begin{pf} We prove the first statement as the second is similar. Since
\[
E_{\nu_\rho} \bigl[E_{\nu_\rho} \bigl[f(\eta) |\eta^{(\ell)},
\eta^c_\ell \bigr] |\eta^{(2\ell)},
\eta^c_{2\ell} \bigr] = E_{\nu_\rho} \bigl[f(\eta) |
\eta^{(2\ell)}, \eta^c_{2\ell} \bigr],
\]
we follow now the same steps as in the proof of Lemma~\ref
{globalone_block} to the last line. To finish the proof, we now give an
order $O(\|f\|^2_{L^5(\nu_\rho)}\ell^{-2})$ bound on the variance
\begin{eqnarray*}
&&\bigl\|E_{\nu_{\rho}}\bigl[f(\eta)|\eta^{(\ell)},\eta^c_\ell
\bigr]-E_{\nu_{\rho
}}\bigl[f(\eta)|\eta^{(2\ell)}, \eta^c_{2\ell}
\bigr]\bigr\|_{L^4(\nu_\rho)}^2.
\end{eqnarray*}
Adding and subtracting terms, and the inequality $(a+b+c)^2 \leq
3a^2+3b^2 +3c^2$, the variance is bounded by
\begin{eqnarray*}
& \leq& 3\biggl \|E_{\nu_{\rho}} \biggl[f(\eta) - \frac{\varphi_f''(\rho
)}{2} \biggl\{\bigl(
\eta^{(\ell)} - \rho\bigr)^2 -\frac{\sigma^2_\ell(\rho)}{2\ell
+1} \biggr\} \Big|
\eta^{(\ell)}, \eta^c_\ell \biggr]
\biggr\|^2_{L^4(\nu_\rho)}
\\
&&{} +3 \biggl\|E_{\nu_{\rho}} \biggl[f(\eta)- \frac{\varphi_f''(\rho
)}{2} \biggl\{\bigl(
\eta^{(2\ell)} - \rho\bigr)^2 -\frac{\sigma^2_{2\ell}(\rho
)}{2(2\ell+1)} \biggr\} \Big|
\eta^{(2\ell)}, \eta^c_{2\ell} \biggr] \biggr\|
^2_{L^4(\nu_\rho)}
\\
&&{} + 3 \biggl\|\frac{\varphi_f''(\rho)}{2} \biggl\{E_{\nu_\rho} \biggl[\bigl(
\eta^{(\ell)} - \rho\bigr)^2 -\frac{\sigma^2_\ell(\rho)}{2\ell+1}\Big |\eta
^{(\ell)}, \eta^c_\ell \biggr]
\\
&&\hspace*{59pt}{} + E_{\nu_\rho} \biggl[\bigl(\eta^{(2\ell)} - \rho
\bigr)^2 -\frac{\sigma^2_{2\ell
}(\rho)}{2(2\ell+1)} \Big|\eta^{(2\ell)},
\eta^c_{2\ell} \biggr] \biggr\} \biggr\|_{L^4(\nu_\rho)}^2.
\end{eqnarray*}
The last term, by the fourth moment bound of $(\eta^{(k)}-\rho)^2$ in
(IM2) with $k=\ell$ and $k=2\ell$ and that $|\varphi_f''(\rho)|\leq
C(\rho)\|f\|_{L^2(\nu_\rho)}$ in (D), is of order $O(\|f\|^2_{L^2(\nu
_\rho)}\ell^{-2})$. But the first two terms are of order $O(\|f\|
^2_{L^5(\nu_\rho)}\ell^{-2+\alpha_0})$ by applying the equivalence
of ensembles assumption (EE).
\end{pf}

%le4.5 #&#
\begin{lemma}[(Two-blocks estimate)]
\label{globaltwo-blocks}
Let $f\dvtx \Omega\to\mathbb R$ be a local $L^5(\nu_\rho)$ function
supported on sites in $\Lambda_{\ell_0}$
such that $\varphi_f(\rho)=\varphi'_f(\rho)=0$. Then, there exists a
constant $C = C(\rho,\ell_0)$ such that for $\ell\geq \ell_0$, $t
\geq0$ and
$h\in{\ell^1(\Z)\cap\ell^2(\mathbb{Z})}$:
\begin{eqnarray*}
&&\bb E_{\nu_\rho} \biggl[ \biggl(\int_0^t
\sum_{x\in{\mathbb{Z}}}\tau_x \bigl\{
E_{\nu_{\rho}}\bigl[f(\eta)|\eta^{(\ell_0)}, \eta_{\ell_0}^c
\bigr]-E_{\nu_{\rho
}}\bigl[f(\eta)|\eta^{(\ell)}, \eta_\ell^c
\bigr] \bigr\} h(x) \,ds \biggr)^2 \biggr]
\\
&& \qquad\leq C\|f\|^2_{L^5(\nu_\rho)}t\frac{\ell
}{n^2}\sum
_{x\in{\mathbb{Z}}}h^2(x).
\end{eqnarray*}

On the other hand, when only $\varphi_f(\rho)=0$ is known,
\begin{eqnarray*}
&&\bb E_{\nu_\rho} \biggl[ \biggl(\int_0^t
\sum_{x\in{\mathbb{Z}}}\tau_x \bigl\{
E_{\nu_{\rho}}\bigl[f(\eta)|\eta^{(\ell_0)}, \eta_{\ell_0}^c
\bigr]-E_{\nu_{\rho
}}\bigl[f(\eta)|\eta^{(\ell)}, \eta_\ell^c
\bigr] \bigr\} h(x) \,ds \biggr)^2 \biggr]
\\
&&\qquad \leq C\|f\|^2_{L^5(\nu_\rho)}t\frac{\ell
^2}{n^2}\sum
_{x\in{\mathbb{Z}}}h^2(x).
\end{eqnarray*}
\end{lemma}

\begin{pf} We prove the first display as the second is analogous.
Again, we invoke Proposition~\ref{KV} and bound the square of the
$H_{-1,n}$ norm. To this end,
write $\ell= 2^{m+1}\ell_0 + r$ where $0\leq r\leq2^{m+1}\ell_0 -1$. Then
\begin{eqnarray*}
&&E_{\nu_{\rho}}\bigl[f(\eta)|\eta^{(\ell_0)}, \eta_{\ell_0}^c
\bigr]-E_{\nu_{\rho
}}\bigl[f(\eta)|\eta^{(\ell)}, \eta_\ell^c
\bigr]
\\
&&\qquad = E_{\nu_\rho}\bigl[f(\eta)|\eta^{(2^{m+1}\ell_0)}, \eta
_{2^{m+1}\ell_0}^c\bigr] - E_{\nu_\rho}\bigl[f(\eta)|
\eta^{(\ell)}, \eta_\ell^c\bigr]
\\
&& \qquad\quad{}+ \sum_{i=0}^{m} \bigl
\{E_{\nu_{\rho}}\bigl[f(\eta)|\eta ^{(2^i \ell_0)}, \eta_{2^i\ell_0}^c
\bigr]-E_{\nu_{\rho}}\bigl[f(\eta)|\eta ^{(2^{i+1}\ell_0)}, \eta_{2^{i+1}\ell_0}^c
\bigr] \bigr\}.
\end{eqnarray*}

Now, by Minkowski's inequality, with respect to the $H_{-1,n}$ norm,
over the $m+2$ terms, and Lemma~\ref{globalrenormalization}, we obtain
that the left-hand side of the display in the lemma statement is
bounded by
\begin{eqnarray*}
&& \Biggl\{ \biggl(\frac{Ct 2^{m+1}\ell_0}{n^2} \biggr)^{1/2} + \sum
_{i=0}^m \biggl(\frac{Ct 2^{i}\ell_0}{n^2}
\biggr)^{1/2} \Biggr\}^2 \|f\|^2_{L^5(\nu
_\rho)}
\sum_{x\in\Z} h^2(x)
\\
&&\qquad \leq \frac{C\|f\|^2_{L^5(\nu_\rho)} t \ell
}{n^2}\sum_{x\in\Z}
h^2(x)
\end{eqnarray*}
to finish the proof.
\end{pf}

%le4.6 #&#
\begin{lemma}[(Equivalence of ensembles estimate)]
\label{EE_1block}
Let $f\dvtx \Omega\to\mathbb R$ be a local $L^5(\nu_\rho)$ function
supported on sites in $\Lambda_{\ell_0}$ such that $\varphi_f(\rho
)=\varphi'_f(\rho)=0$. Then, there exists a constant $C = C(\rho,\ell_0)$
such that for $\ell\geq \ell_0$, $t \geq0$ and $h\in{\ell^1(\mathbb{Z})}$:
\begin{eqnarray*}
&&\bb E_{\nu_\rho} \biggl[ \biggl(\int_0^t
\sum_{x\in{\mathbb{Z}}}\tau_x \biggl\{
E_{\nu_{\rho}}\bigl[f(\eta_s)|\eta_s^{(\ell)},
(\eta_s)_\ell^c\bigr]
\\
&&\hspace*{66pt}\quad -\frac{\varphi_f''(\rho)}{2} \biggl(\bigl(\eta_s^{(\ell)} - \rho
\bigr)^2 - \frac
{\sigma^2_\ell(\rho)}{2\ell+1} \biggr) \biggr\} h(x) \,ds
\biggr)^2 \biggr]
\\
&&\qquad \leq C\|f\|^2_{L^5(\nu_\rho)}t^2\frac{n^2}{\ell
^{2+\alpha_0}}
\biggl(\frac{1}{n}\sum_{x\in\Z}\bigl|h(x)\bigr|
\biggr)^2.
\end{eqnarray*}

On the other hand, when only $\varphi_f(\rho)=0$ is known,
\begin{eqnarray*}
&&\bb E_{\nu_\rho} \biggl[ \biggl(\int_0^t
\sum_{x\in{\mathbb{Z}}}\tau_x \bigl\{
E_{\nu_{\rho}}\bigl[f(\eta_s)|\eta_s^{(\ell)},
(\eta_s)_\ell^c\bigr] -\varphi_f'(
\rho) \bigl(\eta_s^{(\ell)}-\rho \bigr) \bigr\}h(x)\,ds
\biggr)^2 \biggr]
\\
&&\qquad \leq C\|f\|^2_{L^5(\nu_\rho)}t^2\frac{n^2}{\ell
^{1+\alpha_0}}
\biggl(\frac{1}{n}\sum_{x\in\Z}\bigl|h(x)\bigr|
\biggr)^2.
\end{eqnarray*}
Here, $\alpha_0>0$ is the power mentioned in assumption \textup{(EE)}.
\end{lemma}

\begin{pf} By squaring and using invariance of $\nu_\rho$, the
left-hand side of the display is bounded by
\[
2t^2\E_{\nu_\rho} \biggl[ \biggl(\sum
_{x\in\Z}\bigl |h(x)\bigr|\bigl|r(x)\bigr| \biggr)^2 \biggr],
\]
where $r(x)$ is the $\tau_x$-shifted expression in curly braces in the
display of Lemma~\ref{EE_1block}.
Now, by Schwarz inequality,
\[
\biggl(\sum_{x\in\Z} \bigl|h(x)\bigr|r(x) \biggr)^2
\leq \biggl(\sum_{x\in\Z} \bigl|h(x)\bigr| \biggr)\sum
_{x\in\Z} \bigl|h(x)\bigr|r^2(x).
\]
Since $\nu_\rho$ is translation-invariant, the desired bound is now
obtained by noting the form of $r(x)$ and the equivalence of ensembles
assumption (EE).
\end{pf}

\begin{pf*}{Proof of Theorem~\ref{gbg_L2}} By combining Lemma~\ref
{globalone_block} with $\ell= \ell_0$, and Lemmas~\ref
{globaltwo-blocks} and~\ref{EE_1block}, we straightforwardly
obtain the result.
\end{pf*}

%s5 #&#
\section{Equivalence of ensembles}\label{EE_section}

We prove, as a consequence of Proposition~\ref{2EE}, that condition
(EE) holds for a large class of systems with product invariant
measures. In this case, $\nu_{k,\ell, \xi}$ does not depend on $\xi$,
which simplifies the conditional expectation in the statement of (EE).

Next, we show in Proposition~\ref{Markov_EE} that (EE) also holds for
the Markov chain measure $\nu_{1/2}$ defined in Section~\ref{speed_change_model}. Some parts of the proofs of these statements are
similar to those in \cite{SX}.

Define $\Lambda^+_{m} = \{x\dvtx 1\leq x\leq m\}$.

%
%pr5.1 #&#
\begin{proposition}\label{2EE}
Let $\nu_\rho$ be a product measure on $\Omega$ such that \textup{(IM)} holds,
and $0<\nu_\rho(\eta(0)=j)<1$ for $j=0,1$. Let also $f$ be a local
$L^5(\nu_\rho)$ function, supported on sites $\Lambda^+_{\ell_0}$, such
that $\varphi_f(\rho)= \varphi'_f(\rho)=0$. Then
there exists a constant $C= C(\rho,\ell_0)$, such that for $n\geq\ell
_0$ we have
\[
\biggl\| E_{\nu_\rho}\bigl[f(\eta)| y\bigr] - \biggl\{y^2-
\frac{\sigma^2(\rho
)}{n} \biggr\} \frac{\varphi_f''(\rho)}{2} \biggr\|_{L^4(\nu_\rho)} \leq
\frac{C\|f\|_{L^5(\nu_\rho)}}{n^{3/2}}.
\]

On the other hand, when only $\varphi_f(\rho)=0$ is known,
\[
\bigl\| E_{\nu_\rho}\bigl[f(\eta)| y\bigr] - y\varphi_f'(
\rho) \bigr\|_{L^4(\nu_\rho)} \leq \frac{C\|f\|_{L^5(\nu_\rho)}}{n}.
\]
Here, $y:= \frac{1}{n}\sum_{x\in\Lambda^+_n} \eta(x)-\rho$.
\end{proposition}

\begin{pf} We prove the first display as the second statement,
following the same scheme, has a simpler argument. At the expense of
the constant, we need only to consider all large $n>\ell_0$. To
simplify notation, we will call $\ell=\ell_0$. The proof follows in
several steps.

\textit{Step} 1.
Recall the tilted measures $\{\nu_z\dvtx \rho_*<z<\rho^*\}$ given after
assumption (D1) which are well defined as $\nu_\rho$ is a product
measure. Let $\sigma^2(z) = E_{\nu_z}[(\eta(0)-z)^2]$.
Note also the canonical expectation $E_{\nu_z}[f|y]$ does not depend
on $z$, and that we are free to choose it as desired.

Develop
\begin{eqnarray*}
E_{\nu_{\rho}}\bigl[f(\eta)|y\bigr] &=& E_{\nu_{y+\rho}} \biggl[ f(\eta) \Big|
\frac
{1}{n} \sum_{x\in\Lambda^+_n}\eta(x)-\rho= y \biggr]
\\
&=& \frac{E_{\nu_{y+\rho}} [f(\eta)1(({1}/{n})\sum_{x\in\Lambda
^+_n} \eta(x)-\rho= y) ]}{\nu_{y+\rho} (({1}/{n})\sum_{x\in
\Lambda^+_n}\eta(x)-\rho= y )}.
\end{eqnarray*}

Define $\theta_m(z) = \sqrt{m}\nu_{y+\rho}(\sum_{x\in\Lambda^+_m}\eta
(x) - \rho-y=z)$, and write
the last expression as
\[
E_{\nu_{y+\rho}} \biggl[ f(\eta) \frac{\sqrt{n}\theta_{n-\ell}(-\sum_{x\in\Lambda^+_\ell} (\eta(x)-y-\rho))}{\sqrt{n-\ell}\theta
_n(0)} \biggr].
\]
The goal will be now to expand $\theta_{n-\ell}(z)$ to recover the main
terms approximating $E_{\nu_\rho}[f|y]$
when $|y|$ is small. We will treat the case when $|y|$ is bounded away
from~$0$ afterward.

\textit{Step} 2. To expand $\theta_m(z)$,
let $\psi_y(t)= E_{\nu_{y+\rho}}[e^{it(\eta(x)-\rho-y)}]$ be the
characteristic function. Then one can write
\begin{eqnarray*}
\theta_m(x) &=& \frac{\sqrt{m}}{2\pi} \int_{-\pi}^\pi
e^{-itx}\psi_y^m(t)\,dt
\\
&=& \frac{1}{2\pi}\int
_{-\pi\sqrt{m}}^{\pi\sqrt{m}} e^{-itx/\sqrt
{m}}\psi^m_y(t/
\sqrt{m})\,dt.
\end{eqnarray*}

By Taylor expansion,
%
%e5.1 #&#
\begin{eqnarray}
\label{Taylor} 2\pi\theta_m(x) &=& \int_{-\pi\sqrt{m}}^{\pi\sqrt{m}}
\psi _y^m(t/\sqrt{m}) \,dt - \int_{-\pi\sqrt{m}}^{\pi\sqrt{m}}
\frac
{ixt}{\sqrt{m}} \psi^m_y(t/\sqrt{m}) \,dt
\nonumber
\\
&&{} - \frac{1}{2}\int_{-\pi\sqrt{m}}^{\pi\sqrt{m}}
\frac{x^2
t^2}{m} \psi_y^m(t/\sqrt{m}) \,dt \\
&&{}+ O \biggl(
\frac{|x|^3}{m^{3/2}} \biggr) \int_{-\pi\sqrt{m}}^{\pi\sqrt{m}}
|t|^3 \bigl|\psi_y^m(t/\sqrt{m})\bigr|\,dt.\nonumber
\end{eqnarray}

\textit{Step} 3.
Let $\delta>0$ be such that $(\rho-\delta, \rho+\delta)\subset(\rho_*,
\rho^*)$ and sufficiently small in the following estimates. Let also
$0<\varepsilon\leq\pi$.

First,
$\sup_{|y|\leq\delta, \varepsilon\leq|t|\leq\pi}|\psi_y^m(t)|<
C_0^m$ where $C_0<1$: Write
\begin{eqnarray*}
\bigl|\psi_y(t)\bigr| &\leq &\bigl|\nu_{y+\rho}\bigl(\eta(0)=0\bigr) +
e^{it}\nu_{y+\rho
}\bigl(\eta(0)=1\bigr)\bigr| + \sum
_{k\geq2}\nu_{y+\rho}\bigl(\eta(0)=k\bigr)
\\
&\leq& \bigl(A^2 -2\nu_{y+\rho}\bigl(\eta(0)=0\bigr)
\nu_{y+\rho}\bigl(\eta(0)=1\bigr)\bigl[1-\cos (t)\bigr]
\bigr)^{1/2} + 1-A,
\end{eqnarray*}
where $A = \nu_{y+\rho}(\eta(0)=0) + \nu_{y+\rho}(\eta(0)=1)$.
By the proposition assumptions and continuity of $\nu_{y+\rho}(\eta
(0)=k)$ in $y$, $0<\nu_{y+\rho}(\eta(0)=j)<1$ for $j=0,1$ uniformly for
$|y|\leq\delta$. Hence,
uniformly over $\varepsilon\leq|t|\leq\pi$, $|y|\leq\delta$, the
right-hand side of the display above is strictly bounded by a constant $C_0<1$.

Second, for $0\leq|t/\sqrt{m}|<\varepsilon$ and $|y|\leq\delta$,
\[
\psi_y^m(t/\sqrt{m}) = \bigl[1-\bigl(t^2
\sigma^2(y+\rho)/(2m) \bigr) + O \bigl(C(\delta)|t|^3m^{-3/2}
\bigr) \bigr]^m
\]
so that
$|\psi_y^m(t/\sqrt{m})|\leq e^{-C_1(y,\varepsilon)t^2}$. Here, by
continuity in $y$ and $\sigma^2(\rho)>0$, when $\varepsilon$ is small,
$\inf_{|y|\leq\delta}C_1(y,\varepsilon)>0$. Similarly, we note $\sup_{|y|\leq\delta}\sigma^2(y+\rho)<\infty$, and $\inf_{|y|\leq\delta
}\sigma^2(y+\rho)>0$.

Last, by the classical local limit theorem, $\lim_{m\uparrow\infty
}\theta_m(0) = (2\pi\sigma^2(y+\rho))^{-1/2}$.

\textit{Step} 4.
We now observe, for $|y|\leq\delta$ and $m\geq1$, as a consequence of
the estimates in step 3, the integral in the last term in \eqref
{Taylor} is uniformly bounded: Split the integral over the ranges
$|t/\sqrt{m}|<\varepsilon$ and $|t/\sqrt{m}|\geq\varepsilon$ and bound
each part separately.

Also, similarly, we split the second integral in \eqref{Taylor}, when
$|y|\leq\delta$, over ranges $|t/\sqrt{m}|\geq\varepsilon$ and
$|t/\sqrt{m}|< \varepsilon$. On the first range, the restricted
integral exponentially decays, and on the range $|t/\sqrt {m}|<\varepsilon$, the restricted integral is almost the integral of an
odd function since here
\[
\psi_y^m(t/\sqrt{m}) = \biggl(1-\frac{t^2\sigma^2(y+\rho)}{2m}
\biggr)^m \bigl[1+ O\bigl(C(\delta)|t|^3m^{-1/2}
\bigr) \bigr].
\]
Therefore, we conclude that the second integral in \eqref{Taylor} is of
order $O(m^{-1/2})$.

\textit{Step} 5.
Then, for $|y|\leq\delta$, we have
\begin{eqnarray*}
E_{\nu_{\rho}}\bigl[f(\eta)|y\bigr] &=& \kappa_0
E_{\nu_{y+\rho}}\bigl[f(\eta)\bigr] + \frac
{\kappa_1}{\sqrt{n-\ell}}E_{\nu_{y+\rho}}
\biggl[ f(\eta) \biggl(\sum_{x\in
\Lambda^+_\ell}\eta(x)-\rho-y
\biggr) \biggr]
\\
&&{} + \frac{\kappa_2}{n-\ell}E_{\nu_{y+\rho}} \biggl[ f(\eta) \biggl(\sum
_{x\in\Lambda^+_\ell}\eta(x)-\rho-y \biggr)^2 \biggr] +
\varepsilon_f(n),
\end{eqnarray*}
where $|\varepsilon_f(n)| \leq C(\rho,\ell,\delta)\|f\|_{L^2(\nu_\rho
)}n^{-3/2}$ and $\kappa_i=\kappa_i(n)$ for $i=0,1,2$ are explicit
expressions. Indeed, one observes
\begin{eqnarray*}
\kappa_0(n) &=& \frac{\sqrt{n}}{\sqrt{n-\ell}}\frac{\theta_{n-\ell
}(0)}{\theta_n(0)} = 1 + O
\bigl(n^{-1/2}\bigr),
\\
\kappa_1(n) &=& \frac{\sqrt{n}}{\theta_n(0)\sqrt{n-\ell}}\frac{1}{2\pi
}\int
_{-\pi\sqrt{n-\ell}}^{\pi\sqrt{n-\ell}} it\psi_y^{n-\ell}
\biggl(\frac
{t}{\sqrt{n-\ell}} \biggr)\,dt = O\bigl(n^{-1/2}\bigr),
\\
\kappa_2(n) &=& \frac{-\sqrt{n}}{2\theta_n(0)\sqrt{n-\ell}}\frac{1}{2\pi
}\int
_{-\pi\sqrt{n-\ell}}^{\pi\sqrt{n-\ell}} t^2\psi_y^{n-\ell}
\biggl(\frac{t}{\sqrt{n-\ell}} \biggr)\,dt
\\
& = &\frac{-1}{2\sigma^{2}(y+\rho)} + O\bigl(n^{-1/2}\bigr).
\end{eqnarray*}

\textit{Step} 6. We now develop expansions of $E_{\nu_{y+\rho}}[h]$ for a
local $L^2(\nu_\rho)$ function $h$ supported on coordinates in $\Lambda
^+_\ell$.
The ``tilting'' given in the \hyperref[intro_section]{Introduction}, \eqref{tilted_measure}
reduces to
\[
E_{\nu_{y+\rho}}[h] = E_{\nu_\rho} \biggl[ h(\eta) \frac{e^{\lambda(y+\rho)\sum_{x\in
\Lambda^+_\ell}(\eta(x)-\rho)}}{M^\ell(\lambda(y+\rho))}
\biggr],
\]
where $\lambda(y+\rho)$ is the ``tilt'' chosen to change the density to
$y+\rho$ and $M(\lambda) = E_{\nu_\rho}[e^{\lambda(\eta(x)-\rho)}]$.
Note that
$z-\rho= M'(\lambda(z))/M(\lambda(z))$ and
\[
\lambda'(z) = \biggl[\frac{M''(\lambda(z))}{M(\lambda(z))} - \biggl(
\frac{M'(\lambda
(z))}{M(\lambda(z))} \biggr)^2 \biggr]^{-1} =
\frac{1}{\sigma^{2}(z)}.
\]

Consider the first and second derivatives of $E_{\nu_{y+\rho}}[h]$
given exactly in \eqref{varphi_derivatives} as $\nu_{y+\rho}$ is a
product measure. The third derivative takes the form
\begin{eqnarray*}
\frac{d^3}{dy^3} E_{\nu_{y+\rho}}\bigl[h(\eta)\bigr]& =&
\lambda'''(y+\rho)E_{\nu_{y+\rho}}
\biggl[\bar h(\eta) \biggl(\sum_{x\in\Lambda^+_\ell}\eta(x)-y-\rho
\biggr) \biggr]
\\
&&{} +3\lambda'(y+\rho)\lambda''(y+
\rho)E_{\nu_{y+\rho}} \biggl[\bar h(\eta ) \biggl(\sum
_{x\in\Lambda^+_\ell}\eta(x)-y-\rho \biggr)^2 \biggr]
\\
&& {}+\bigl(\lambda'(y+\rho)\bigr)^3E_{\nu_{y+\rho}}
\biggl[\bar h(\eta) \biggl(\sum_{x\in\Lambda^+_\ell}\eta(x)-y-\rho
\biggr)^3 \biggr]
\\
&&{} -3\bigl(\lambda'(y+\rho)\bigr)^3E_{\nu_{y+\rho}}
\biggl[\bar h(\eta) \biggl(\sum_{x\in\Lambda^+_\ell}\eta(x)-y-\rho
\biggr) \biggr]
\\
&&\quad{} \times E_{\nu
_{y+\rho}} \biggl[ \biggl(\sum_{x\in\Lambda^+_\ell}
\eta(x)-y-\rho \biggr)^2 \biggr],
\end{eqnarray*}
where $\bar h(\eta) = h(\eta) - E_{\nu_{y+\rho}}[h]$.

Then, for $|y|\leq\delta$, when $\varphi_h(\rho)=\varphi_h'(\rho)=0$,
we may expand around $y=0$:
\[
E_{\nu_{y+\rho}}\bigl[h(\eta)\bigr] = \bigl(\lambda'(\rho)
\bigr)^2\frac{y^2}{2}E_{\nu_\rho}\biggl[h(\eta) \biggl(\sum
_{x\in
\Lambda^+_\ell}\eta(x)-\rho \biggr)^2 \biggr] +
|y|^3r(\rho,\delta,h).
\]
When only $\varphi_h(\rho)=0$ is known,
\begin{eqnarray*}
E_{\nu_{y+\rho}}\bigl[h(\eta)\bigr]&=& \lambda'(
\rho)yE_{\nu_\rho} \biggl[h(\eta ) \biggl(\sum_{x\in
\Lambda^+_\ell}
\eta(x)-\rho \biggr) \biggr] + |y|^2r(\rho,\delta,h).
\end{eqnarray*}
When possibly $\varphi_h(\rho)\neq0$,
\[
E_{\nu_{y+\rho}}\bigl[h(\eta)\bigr]= E_{\nu_\rho}\bigl[h(\eta)\bigr] +
|y|r(\rho, \delta, h).
\]
Here, as the first and second derivatives in \eqref{varphi_derivatives}
and the third derivative above are bounded for $|y|\leq\delta$, we may
conclude that the remainders $|r(\rho,\delta,h)| \leq C(\rho, \delta)\|
h\|_{L^2(\nu_\rho)}$.

We now relate the terms $E_{\nu_\rho} [h(\eta) (\sum_{x\in\Lambda
_\ell^+} (\eta(x)-\rho) )^k ]$ to derivatives $\varphi
^{(k)}_h(\rho)$:
From \eqref{varphi_derivatives}, for $k=1,2$, when $\varphi
_h^{(k-1)}(\rho) = \varphi_h(\rho) = 0$, we have
%
%e5.2 #&#
\begin{equation}
\label{step6_line} \varphi_h^{(k)}(\rho) = \bigl(
\lambda'(\rho)\bigr)^kE_{\nu_\rho} \biggl[h(\eta )
\biggl(\sum_{x\in\Lambda^+_\ell}\bigl(\eta(x)-\rho\bigr)
\biggr)^k \biggr].
\end{equation}

\textit{Step} 7.
Consider the expansion of $E_{\nu_\rho}[f|y]$ in step 5 when $|y|\leq
\delta$. With $h$ equal to variously $f$, $f(\eta) (\sum_{x\in\Lambda
_\ell^+}(\eta(x)-\rho) )$, and $f(\eta) (\sum_{x\in\Lambda_\ell
^+}(\eta(x)-\rho) )^2$, we may write
\begin{eqnarray*}
E_{\nu_\rho}[f|y] & = & \frac{\kappa_0}{2}\bigl(\lambda'(
\rho)\bigr)^2y^2E_{\nu
_\rho} \biggl[f(\eta) \biggl(
\sum_{\Lambda_\ell^+}\bigl(\eta(x)-\rho\bigr)
\biggr)^2 \biggr] + \kappa_0|y|^3r(f)
\\
&&{} + \frac{\kappa_1\lambda'(\rho)y}{\sqrt{n-\ell}}E_{\nu_\rho} \biggl[f(\eta) \biggl(\sum
_{\Lambda_\ell^+}\bigl(\eta(x)-\rho\bigr) \biggr)^2 \biggr] +
\frac
{\kappa_1}{\sqrt{n-\ell}}|y|^2r(f)
\\
&&{} + \frac{\kappa_2}{n-\ell}E_{\nu_\rho} \biggl[f(\eta) \biggl(\sum
_{\Lambda
^+_\ell}\bigl(\eta(x)-\rho\bigr) \biggr)^2 \biggr] +
\frac{\kappa_2}{n-\ell}|y|r(f) + \varepsilon_f(n),
\end{eqnarray*}
where $|r(f)|\leq C(\rho, \ell, \delta)\|f\|^2_{L^2(\nu_\rho)}$.

Hence, noting the assumptions on $\varphi_f(\rho)$, \eqref{step6_line}, and
$E_{\nu_\rho}[y^{2p}] = O(n^{-p})$ so that each $y$ factor is
$O(n^{-1/2})$, we can group the dominant terms so that
\begin{eqnarray*}
&& E_{\nu_\rho} \biggl[ 1\bigl(|y|\leq\delta\bigr)
\\
&&\hspace*{12pt}\quad \times \biggl(E_{\nu_{\rho}}\bigl[f(\eta)|y\bigr] - \biggl\{
\frac{\kappa_0
y^2}{2} + \frac{1}{\lambda'(\rho)}\frac{\kappa_1 y}{\sqrt{n}} + \frac
{1}{(\lambda'(\rho))^2}
\frac{\kappa_2}{n} \biggr\}\varphi_f''(
\rho) \biggr)^4 \biggr]
\\
&&\qquad \leq C(\rho, \delta )\|f\|^4_{L^2(\nu_\rho)}n^{-6}.
\end{eqnarray*}
Noting $\kappa_0(n) = 1 + O(n^{-1/2}), \kappa_1(n)= O(n^{-1/2})$,
formula $\lambda'(\rho) = \sigma^{-2}(\rho)$ in step~6, $|\varphi''(\rho
)|\leq C\|f\|_{L^2(\nu_\rho)}$ and, by Taylor expansion of $\sigma
^2(y+\rho)$ around $y=0$, $\kappa_2(n) = -2^{-1}\sigma^{-2}(\rho) +
O(n^{-1/2})$, we have further
\[
E_{\nu_\rho} \biggl[ 1\bigl(|y|\leq\delta\bigr) \biggl(E_{\nu_\rho}\bigl[f(
\eta)|y\bigr] - \biggl\{y^2 - \frac{\sigma^2(\rho)}{n} \biggr\}
\frac{\varphi_f''(\rho)}{2} \biggr)^4 \biggr] \leq C\|f\|^4_{L^2(\nu_\rho)}n^{-6}.
\]

\textit{Step} 8.
On the other hand,
by say large deviations estimates, we bound
\begin{eqnarray*}
&&E_{\nu_\rho} \biggl[ 1\bigl(|y|> \delta\bigr) \biggl(E_{\nu_\rho}\bigl[f(
\eta)|y\bigr] - \biggl\{ y^2 - \frac{\sigma^2(\rho)}{n} \biggr\}
\frac{\varphi_f''(\rho)}{2} \biggr)^4 \biggr]
\\
&&\qquad \leq C\|f\|^4_{L^5(\nu_\rho)}O\bigl(n^{-6}\bigr)
\end{eqnarray*}
to complete the proof.
\end{pf}

We now prove the equivalence ensembles estimate (EE) with respect to a
Markovian measure.
Recall the Gibbs measures $\nu_{1/2}$ and $\nu_z=\nu^{\lambda
(z)}_{1/2}$, and transition matrix $P$ defined in Section~\ref{speed_change_model}.
To see how the next proposition can be used to satisfy assumption (EE),
we note (1) the estimate is uniform in the ``outside variables'' $\eta
^c_\ell$, and (2) since the transition matrix $P$ is positive, the
$L^\infty$ norm of any local function supported on sites $\Lambda_{\ell
_0}$ can be bounded $\|f\|_{L^\infty} \leq C(\ell_0,\beta)\|f\|_{L^p(\nu
_{1/2})}$ for $p>0$. Recall also the definitions of $\varphi_f(\rho)$
and its derivatives in~\eqref{varphi_derivatives}.

%
%pr5.2 #&#
\begin{proposition}
\label{Markov_EE}
Let $f$ be a local function, supported on sites indexed by $\Lambda
_{\ell_0}$, such that $\varphi_f(1/2)=\varphi'_f(1/2)=0$. Then, for each
$0<\varepsilon<1$, there is a constant $C=C(\ell_0,\varepsilon)$ such
that for $a,b\in\{0,1\}$ and $n\geq\ell_0$,
\begin{eqnarray*}
&& \biggl\| E_{\nu_{{1}/{2}}}\bigl[f|y, \eta(-n-1)=a,\eta(n+1)=b\bigr] -
\frac
{\varphi_f''({1}/{2})}{2} \biggl[y^2 - \frac{\sigma^2_n(
{1}/{2})}{2n+1} \biggr]
\biggr\|_{L^4(\nu_{{1}/{2}})}
\\
&&\qquad \leq \frac{C\|f\|_{L^\infty
}}{n^{3/2-\varepsilon}}.
\end{eqnarray*}

On the other hand, when only $\varphi_f(1/2)=0$ is known,
\[
\biggl\| E_{\nu_{{1}/{2}}}\bigl[f|y, \eta(-n-1)=a,\eta(n+1)=b\bigr] - y\varphi
'_f\biggl(\frac{1}{2}\biggr) \biggr\|_{L^4(\nu_{{1}/{2}})}
\leq \frac{C\|f\|
_{L^\infty}}{n^{1-\varepsilon}}.
\]
Here, $y = (2n+1)^{-1}\sum_{x\in\Lambda_n}(\eta(x) -\frac{1}{2})$.
\end{proposition}

\begin{pf}
The argument has the same structure as for
Proposition~\ref{2EE}. We will concentrate on the first display for all
large $n$; the second statement has a similar argument. Since $\nu
_{1/2}$ corresponds to an ergodic finite-state Markov chain with
uniform invariant measure, it is exponentially mixing and allows
standard approximations, which are used in many steps.

\textit{Step} 1.
Let $\chi>0$ be small and $n' = n - n^\chi$. Write
\begin{eqnarray*}
&&E_{\nu_{1/2}}\bigl[f|y, \eta(-n-1)=a,\eta(n+1)=b\bigr]
\\
&& \qquad= E_{\nu_{y+1/2}}\bigl[f(\eta)|y, \eta(-n-1)=a,\eta(n+1)=b\bigr]
\\
&&\qquad = E_{\nu_{y+1/2}} \biggl[ f(\eta) \frac{\sqrt{2n+1}\theta^\chi
_{n,y,a,b}(-\sum_{x\in\Lambda_{n^\chi}} (\eta(x)-y-1/2))}{\sqrt {2n'}
\theta_{n,y,a,b}(0)}\Big|
\\
&&\hspace*{126pt}\qquad\qquad{} \eta(-n-1)=a, \eta (n+1)=b \biggr],
\end{eqnarray*}
where
\begin{eqnarray*}
\theta^\chi_{n,y,a,b}(z) &=& \sqrt{2n'}
\nu_{y+1/2} \biggl(\sum_{n^\chi<|x|\leq n}\eta (x)-y-1/2=z\Big|
\\
&&\hspace*{61pt}{}  \eta\bigl(n^\chi \bigr), \eta\bigl(-n^\chi\bigr),
\eta(-n-1)=a,\\
&&\hspace*{144pt}\eta(n+1)=b \biggr),
\\
\theta_{n,y,a,b}(z) &=& \sqrt{2n+1}\nu_{y+1/2} \biggl(\sum
_{x\in\Lambda
_n}\eta(x) - y-1/2 =z\Big|
\\
&&\hspace*{76pt}{} \eta (-n-1)=a,\eta(n+1)=b \biggr).
\end{eqnarray*}

\textit{Step} 2.
Let the characteristic function $\psi_{n,y,\chi,a,b}(t)$ for $|t|\leq
\pi$ be defined by
\[
E_{\nu_{y+1/2}} \bigl[e^{it\sum_{n^\chi<|x|\leq n}(\eta(x)-y-{1}/{2})} |\eta\bigl(n^\chi\bigr),
\eta\bigl(-n^\chi\bigr), \eta(-n-1)=a, \eta(n+1)=b \bigr].
\]

Let $\delta>0$ be such that $(\rho-\delta,\rho+\delta)\subset(0,1)$ and
sufficiently small in the following estimates. Suppose $|y|\leq\delta
$. Let also $r>0$ be a small number. We now state a few relations, and
then argue them.
First, for $|t|<r$, we claim
%
%e5.3 #&#
\begin{equation}
\label{claim1_step2} \psi_{n,y,\chi,a,b} \biggl(\frac{t}{\sqrt{2n'}} \biggr) =
\biggl(1-\frac
{t^2\sigma^2(y+1/2)}{2(2n')} \biggr)^{2n'} \bigl[1+O
\bigl(n'^{-1/2}\bigr) \bigr],
\end{equation}
where $\sigma^2(z) = \lim_{n\uparrow\infty} (2n+1)^{-1}E_{\nu_{z}}[(\sum_{x\in\Lambda_n}\eta(x)-z)^2]$ is the limiting variance of the
additive functional $n^{-1/2}\sum_{x=1}^n \eta(x)-z$ with respect to
measure $\nu_z$ (cf. formula in Section~\ref{speed_change_model}).

Therefore, for $|t|<r$ and $C = C(\delta,r)>0$, we have $|\psi_{n,y,\chi
,a,b}(t/\sqrt{2n'})| < \exp\{-Ct^2\}$.

Next, we claim, for $r\leq|t|\leq\pi$ that
%
%e5.4 #&#
\begin{equation}
\label{claim2_step2} \bigl|\psi_{n,y,\chi,a,b}(t)\bigr|< A^{2n'},
\end{equation}
where $A=A(\delta, r)<1$.

We also state a case of a local central limit theorem for ergodic
Markov chains~\cite{Kolmogorov},
\[
\lim_{n\uparrow\infty} \theta_{n,y,a,b}(0) = \frac{1}{\sqrt{2\pi
\sigma^2(y+1/2)}}.
\]

We now give an argument for the above claims which may be skipped on
first reading.
Recall $u_1$ and $v_1$ near \eqref{p_formulas} with $\lambda= \lambda
(y+1/2)$, and consider the transfer matrix:
\[
\widetilde P(s) = \left[ \matrix{
(1-u_1)e^{s(-1/2-y)}& u_1e^{s(1/2-y)}
\vspace*{2pt}\cr
v_1e^{s(-1/2-y)}& (1-v_1)e^{s(1/2-y)}}
\right].
\]
By the Markov property, one writes
%
%e5.5 #&#
\begin{equation}
\label{eigenvalue_decomp} \psi_{n,y,\chi,a,b}(t) = \frac{\widetilde P(it)^{n'}(\eta(-n^\chi
),a)\widetilde P(it)^{n'}(\eta(n^\chi),b)} {
\widetilde P(0)^{n'}(\eta(-n^\chi),a)\widetilde P(0)^{n'}(\eta(n^\chi),b)}.
\end{equation}

We may diagonalize $\widetilde P(it/\sqrt{m})^m = Q(it/\sqrt {m})D^m(it/\sqrt{m})Q^{-1}(it/\sqrt{m})$, for large $m$, where $D(t)$
is a diagonal matrix with eigenvalues $\d_1(t)$ and $\d_2(t)$ and
$Q(t)$ is the matrix of the corresponding eigenvectors. Of course, when
$t=0$, $1=\d_1(0)>\d_2(0) = 1-u_1-v_1$ with corresponding eigenvectors
$\langle1,1\rangle$ and $\langle u_1,-v_1\rangle$. For large $m$, $\d
_1(it/\sqrt{m})$ is the eigenvalue with maximum absolute value and is
expressed as
\begin{eqnarray*}
&&\frac{e^{-ity/\sqrt{m}}}{2} \bigl[(1-u_1)e^{-it/(2\sqrt{m})} +
(1-v_1)e^{it/(2\sqrt{m})}
\\
&&\hspace*{25pt}\qquad{} + \bigl\{ \bigl((1-u_1)e^{-it/(2\sqrt{m})} - (1-v_1)e^{it/(2\sqrt{m})}
\bigr)^2 + 4u_1v_1 \bigr\}^{1/2}
\bigr].
\end{eqnarray*}
It is not difficult to check that
\begin{eqnarray*}
\d_1'(0) & =& -ity/\sqrt{m} + it/(2\sqrt{m})
\bigl[(u_1-v_1)/(u_1+v_1)
\bigr]
\\
& =& (it/\sqrt{m})E_{\pi_1(y+1/2)}\bigl[\eta(0)-1/2-y\bigr] = 0,
\end{eqnarray*}
where $\pi_1(y+1/2)= (u_1+v_1)^{-1}\langle v_1,u_1\rangle$ is the
marginal of $\nu_{y+1/2}$ (cf. Section~\ref{speed_change_model}).
One now expands, as all quantities are smooth,
$\d_1(it/m) = 1 -(t^2/2m) \d_1''(0) + O(m^{-3/2})$
where the error is uniform for $|y|\leq\delta$ and $|t|\leq\pi$.
Similarly, $Q(it/\sqrt{m}) = Q(0) + O(m^{-1/2})$.
One can identify $\d_1''(0)$ as the variance $\sigma^2(y+1/2)$ since we know
\begin{eqnarray*}
&&E_{\nu_{y+1/2}} \bigl[e^{(it/\sqrt{m})\sum_{x=1}^m\eta(x)-y-1/2} \bigr]\\
&&\qquad=\pi_1(y+1/2)P(it/
\sqrt{m})^m\bf1
\\
&& \qquad=\bigl(1 - t^2\d_1''(0)/(2m)
\bigr)^m \bigl(1 + O\bigl(m^{-1/2}\bigr) \bigr)
\end{eqnarray*}
must converge to $e^{-t^2\sigma^2(y+1/2)/2}$. Here, $\pi_1(y+1/2)$ is
thought of as a row vector, and $\bf1$ is the column vector with
entries equal to $1$.

Putting these estimates together, we may conclude \eqref
{claim1_step2}. To verify \eqref{claim2_step2}, from equation \eqref
{eigenvalue_decomp}, we need only show the moduli $|\d_1(it)|, |\d
_2(it)|<1$ uniformly for $r\leq|t|\leq\pi$ and $|y|\leq\delta$. One
way is the following. Suppose $y=0$ and note that the moduli are less
than $1$ at $|t|=\pi$. For $r\leq|t|\leq\pi$, from the determinant of
$\widetilde P(it)$, $\d_1(it)\d_2(it) = 1-u_1-v_1$. In particular, if
$\d_1(it)$ say is of the form $e^{i\theta}$ with $|\theta|\leq\pi$,
then $\d_2(it) = e^{-i\theta}(1-u_1-v_1)$. From the trace, we obtain equation
$e^{i\theta} + e^{-i\theta}(1-u_1-v_1) = (1-u_1)e^{-it/2} +
(1-v_1)e^{it/2}$ which is absurd: The real part is $\cos(\theta
)(2-u_1-v_2) = \cos(t/2)(2-u_1-v_2)$ which yields $\theta= t/2$. But
the imaginary part is $\sin(\theta)(u_1+v_1) = \sin(t/2)(u_1-v_1)$
which is a contradiction as $r/2<|\theta| = |t|/2\leq\pi/2$ and
$v_1\neq0$ for $y=0$. Hence, by continuity, for $|y|\leq\delta$
small, we conclude the claim.

\textit{Step} 3.
Now, write
\begin{eqnarray*}
\theta^\chi_{n, y, a,b}(x) &=& \frac{\sqrt{2n'}}{2\pi} \int
_{-\pi}^\pi e^{-itx}\psi_{n,y,\chi,a,b}(t)\,dt
\\
&=& \frac{1}{2\pi}\int_{-\pi\sqrt{2n'}}^{\pi\sqrt{2n'}}
e^{-itx/\sqrt
{2n'}}\psi_{n,y,\chi,a,b}\bigl(t/\sqrt{2n'}\bigr)\,dt.
\end{eqnarray*}
The last expression is rewritten as
\begin{eqnarray*}
&&\frac{1}{2\pi}\int_{-\pi\sqrt{2n'}}^{\pi\sqrt{2n'}}
\psi_{n,y,\chi
,a,b}\bigl(t/\sqrt{2n'}\bigr)\,dt - \frac{ix}{2\pi\sqrt{2n'}}
\int_{-\pi\sqrt
{2n'}}^{\pi\sqrt{2n'}} t\psi_{n,y,\chi,a,b}\bigl(t/
\sqrt{2n'}\bigr)\,dt
\\
&&\qquad{} - \frac{x^2}{4\pi n'}\int_{-\pi\sqrt{2n'}}^{\pi\sqrt{2n'}}
t^2\psi_{n,y,\chi,a,b}\bigl(t/\sqrt{2n'}\bigr)\,dt +
r_0(x)n^{-3/2}
\end{eqnarray*}
in terms of error $r_0(x)$ which, by the estimates in step 2, is of
order $O(|x|^3)$.

The second integral in the last display is also estimated of order
$O(n'^{-1/2})$ by the same argument as given in step 3 of the proof of
Proposition~\ref{2EE}.

\textit{Step} 4.
Then, for $|y|\leq\delta$, we have
\begin{eqnarray*}
&&E_{\nu_{1/2}} \bigl[ f|y, \eta(-n-1)=a,\eta(n+1)=b \bigr]
\\
&&\qquad = \kappa_0 E_{\nu_{y+1/2}} \bigl[f(\eta)|\eta(-n-1)=a,
\eta(n+1)=b \bigr]
\\
&&\qquad\quad{} + \frac{\kappa_1}{\sqrt{2n'}}E_{\nu_{y+1/2}} \biggl[ f(\eta) \biggl(\sum
_{|x|\leq n^\chi} \eta(x)-\frac{1}{2}-y \biggr) \Big|\eta(-n-1)=a,\\
&&\hspace*{247pt}{}\eta
(n+1)=b \biggr]
\\
&&\qquad\quad{} + \frac{\kappa_2}{2n'}E_{\nu_{y+1/2}} \biggl[ f(\eta) \biggl(\sum
_{|x|\leq n^\chi}\eta(x)-\frac{1}{2}-y \biggr)^2 \Big|
\eta(-n-1)=a,\\
&&\hspace*{242pt}{}\eta (n+1)=b \biggr]
\\
&&\qquad\quad{} + \varepsilon_f(n),
\end{eqnarray*}
where $|\varepsilon_f(n)|\leq C\|f\|_{L^\infty(\nu_\rho)}n^{-3/2 +3\chi
}$ and $\kappa_i=\kappa_i(n)$ for $i=0,1,2$ have the same asymptotics
as in step 5 of the proof of Proposition~\ref{2EE}.

\textit{Step} 5. Recall the tilted measures and the formula for the tilt
$\lambda(z)$ in Section~\ref{speed_change_model}. Recall also the
definitions of $\varphi^{(i)}_f(\rho)$ \eqref{varphi_derivatives}.
For $|y|\leq\delta$ and $i=0,1,2$, using the uniform exponentially
mixing property of the measures $\{\nu_{y+1/2}\dvtx |y|\leq\delta\}$ and
$\varphi_f(1/2)=\varphi_f'(1/2)=0$,
we claim
%
%e5.6 #&#
\begin{eqnarray}
\label{mixing_taylor} &&E_{\nu_{y+1/2}} \biggl[f(\eta) \biggl(\sum
_{|x|\leq n^{\chi}}\eta (x)-1/2-y \biggr)^i \Big|\eta(-n-1)=a,
\eta(n+1)=b \biggr]
\nonumber
\\
&&\qquad = \frac{\lambda'(1/2)^{2-i} y^{2-i}}{(2-i)!}E_{\nu_{1/2}} \biggl[f(\eta) \biggl(\sum
_{|x|\leq n^{2\chi}}\eta(x)-1/2 \biggr)^2 \biggr]
\\
&&\qquad\quad {}+ |y|^{3-i}r_1(f,n) + r_2(f,n).\nonumber
\end{eqnarray}
Here, the error $r_1(f,n)$ stands for the error made first in Taylor
approximation around $y=0$ with respect to the conditioned measure:
Using that $\nu_{y+1/2}$ is exponentially mixing, one can bound the
first, second and third derivatives below~\eqref{first_derivative},
\eqref{second_derivative} and \eqref{third_derivative}, uniformly in
$a,b$ and $|y|\leq\delta$ after a calculation so that $|r_1(f,n)| \leq
C(\delta)n^{4\chi}\|f\|_{L^\infty}$. The error $r_2(f,n)$ represents
other errors made by exponential approximations and $|r_2(h,n)| \leq C\|
f\|_{L^\infty}n^{-3/2}$. The reader, on first reading, may like to skip
now to step 6.

Indeed, in more detail, when $i=1$,
%
%e5.7 #&#
\begin{eqnarray}
\label{mixing_1} &&E_{\nu_{y+1/2}} \biggl[f(\eta) \biggl(\sum
_{|x|\leq n^{\chi}}\eta (x)-1/2-y \biggr) \Big|\eta(-n-1)=a,\eta(n+1)=b \biggr]
\nonumber
\\
&&\qquad= E_{\nu_{1/2}} \biggl[f(\eta) \biggl(\sum_{|x|\leq n^{\chi}}
\eta (x)-1/2 \biggr) \Big|\eta(-n-1)=a,\eta(n+1)=b \biggr]
\\
&&\qquad\quad{} + By + y^2r_1(f,n),\nonumber
\end{eqnarray}
where, referring to the first derivative expression \eqref
{first_derivative}, $B$ equals
%
%e5.8 #&#
\begin{eqnarray}
\label{B_expression} && \lambda' \biggl(\frac{1}{2} \biggr)
E_{\nu_{1/2}} \biggl[f(\eta) \biggl(\sum_{|x|\leq n^{\chi}}
\eta(x)-\frac{1}{2} \biggr) \biggl(\sum_{|x|\leq
n}
\widetilde\eta(x) \biggr)\Big|
\nonumber
\\
&&\hspace*{69pt}\qquad{} \eta(-n-1)=a,\eta(n+1)=b \biggr]
\\
&&\qquad{} - 2n^\chi E_{\nu_{1/2}} \bigl[f |\eta(-n-1)=a,\eta (n+1)=b
\bigr]\nonumber
\end{eqnarray}
and $\widetilde\eta(x) = \eta(x)- E_{\nu_{1/2}}[\eta(x)|\eta
(-n-1)=a,\eta(n+1)=b]$.
The error $r_1(f,n)$ is~less than the bound on the second derivative
\eqref{second_derivative} with $h = f(\eta)\* (\sum_{|x|\leq n^{\chi
}}\eta(x)-1/2 )$ plus $2n^\chi$ times the bound on the first
derivative \eqref{first_derivative} with $h = f$.
We now bound the second derivative; estimating the first derivative is
similar. See notation $\bar h$ and $\bar\eta(x)$ above \eqref{first_derivative}.\vadjust{\goodbreak}

For $|y|\leq\delta$, from the formula for the tilt $\lambda(z)$ in
Section~\ref{speed_change_model}, the derivatives $\lambda
^{(k)}(y+1/2)$ for $k=1,2,3$ are uniformly bounded.
The expectation
$E_{\nu_{y+1/2}} [\bar h(\eta) (\sum_{|x|\leq n}\bar\eta(x)
)|\eta(-n-1)=a,\eta(n+1)=b ]$ in \eqref{second_derivative}
is handled as follows. By splitting the sum $\sum_{|x|\leq n}\bar\eta
(x)$ over indices $|x|\leq n^{2\chi}$, $n^{2\chi}<|x|\leq n-n^\chi$ and
$|x|>n-n^\chi$, and using the uniform exponentially mixing property of
$\nu_{y+1/2}$, for $|y|\leq\delta$, and that all variables $|\eta
(x)|\leq1$, one bounds this term as $O(\|f\|_{L^\infty}n^{3\chi})$.

Consider now the other term $E_{\nu_{y+1/2}} [\bar h(\eta) (\sum_{|x|\leq n}\bar\eta(x) )^2|\eta(-n-1)=a$, $\eta(n+1)=b ]$ in \eqref
{second_derivative}. Split the sum over $|x|\leq n$ into sums over
$|x|\leq n^{(1+u)\chi}$ and $|x|>n^{(1+u)\chi}$, and square to yield
three terms. Bounding the cross term is the most involved, the other
two being straightforward. The cross term is
\begin{eqnarray*}
&&2E_{\nu_{y+1/2}} \biggl[\bar h(\eta) \biggl(\sum
_{|x|\leq n^{(1+u)\chi}}\bar \eta(x) \biggr)
 \biggl(\sum_{n^{(1+u)\chi
}<|x|\leq n}\bar\eta(x) \biggr)\Big|\\
&&\hspace*{104pt}\eta(-n-1)=a,\eta(n+1)=b \biggr].
\end{eqnarray*}
By\vspace*{1pt} splitting the sum over $n^{(1+u)\chi}<|x|\leq n$ into sums on
$n^{(1+u)\chi}<\break |x|<n^{(1+2u)\chi}$, $n^{(1+2u)\chi}
\leq|x| \leq n-n^{u\chi}$ and $|x|> n-n^{u\chi}$, and using the
exponentially mixing property of $\nu_{y+1/2}$, one can bound the cross
term\break $O(\|f\|_{L^\infty}n^{(3+3u)\chi})$ which for $u<1/3$ gives the
desired error bound.

We now
relate terms in \eqref{mixing_1} to $\varphi_f'(1/2)$ and $\varphi
_f''(1/2)$ [cf. \eqref{varphi_derivatives}], using the exponentially
mixing property. It is straightforward that the difference between the
first conditional expectation on the right-hand side of \eqref
{mixing_1} and $\lambda'(1/2)^{-1}\varphi_f'(1/2)=0$ is exponentially
close. Also, as $\varphi'_f(1/2)=0$, the first conditional expectation
in the expression $B$ in \eqref{B_expression} is exponentially close to
$(\lambda'(1/2))^{-1}\varphi_f''(1/2)$, which in turn is exponentially
close to the expectation on the right-hand side of \eqref
{mixing_taylor}. The other expectation in $B$ is exponentially small.

The cases $i=0,2$ with respect to equation \eqref{mixing_taylor},
are argued analogously.

Here, for functions $h$ supported on sites in $\Lambda_{n^\chi}$, and notation
\begin{eqnarray*}
\bar h(\eta) &=& h(\eta) - E_{\nu_{y+1/2}}\bigl[h|\eta(-n-1)=a,\eta(n+1)=b
\bigr] \quad\mbox{and }
\\
\bar\eta(x) &=& \eta(x) - E_{\nu_{y+1/2}}\bigl[\eta(x)|\eta(-n-1)=a,
\eta(n+1)=b\bigr],
\end{eqnarray*}
the first derivative is
%
%e5.9 #&#
\begin{eqnarray}
\label{first_derivative} && \frac{d}{dy}E_{\nu_{y+{1}/{2}}}\bigl[h(\eta)|
\eta(-n-1)=a, \eta (n+1)=b\bigr]
\nonumber
\\
&&\qquad = \lambda'\biggl(y+\frac{1}{2}\biggr) \\
&&\qquad\quad{}\times E_{\nu_{y+{1}/{2}}}
\biggl[\bar h(\eta ) \biggl(\sum_{|x|\leq n}\bar\eta(x)
\biggr) \Big| \eta(-n-1)=a, \eta (n+1)=b \biggr].\nonumber
\end{eqnarray}
The second derivative is
%
%e5.10 #&#
\begin{eqnarray}
\label{second_derivative} &&\frac{d^2}{dy^2} E_{\nu_{y+{1}/{2}}} \bigl[h(\eta) |
\eta(-n-1)=a, \eta(n+1)=b \bigr]
\nonumber
\\
&&\qquad = \lambda''\biggl(y+\frac{1}{2}
\biggr)E_{\nu_{y+{1}/{2}}} \biggl[\bar h(\eta) \biggl(\sum
_{|x|\leq n}\bar\eta(x) \biggr) \Big| \eta(-n-1)=a, \eta (n+1)=b \biggr]
\nonumber
\\[-8pt]
\\[-8pt]
\nonumber
&&\qquad\quad{} + \biggl(\lambda'\biggl(y+\frac{1}{2}\biggr)
\biggr)^2 \\
&&\qquad\qquad{}\times E_{\nu_{y+{1}/{2}}} \biggl[\bar h(\eta) \biggl(\sum
_{|x|\leq n} \bar\eta(x) \biggr)^2 \Big| \eta(-n-1)=a, \eta(n+1)=b \biggr].\nonumber
\end{eqnarray}
The third derivative is
%
%e5.11 #&#
\begin{eqnarray}
\label{third_derivative}
&&\frac{d^3}{dy^3} E_{\nu_{y+{1}/{2}}} \bigl[h(\eta)|
\eta(-n-1)=a,\eta (n+1)=b \bigr]
\nonumber
\\
&&\qquad = \lambda'''\biggl(y+
\frac{1}{2}\biggr)E_{\nu_{y+{1}/{2}}} \biggl[\bar h(\eta ) \biggl(\sum
_{|x|\leq n}\bar\eta(x) \biggr)\Big |\eta(-n-1)=a,\eta(n+1)=b
\biggr]
\nonumber
\\
&&\qquad\quad{} +3\lambda'\biggl(y+\frac{1}{2}\biggr)
\lambda''\biggl(y+\frac{1}{2}\biggr)
\nonumber
\\
&&\qquad\qquad{} \times E_{\nu_{y+{1}/{2}}} \biggl[\bar h(\eta ) \biggl(\sum
_{|x|\leq n}\bar\eta(x) \biggr)^2 \Big|\eta(-n-1)=a,\eta
(n+1)=b \biggr]\nonumber
\\
&& \qquad\quad{}+\biggl(\lambda'\biggl(y+\frac{1}{2}\biggr)
\biggr)^3 \\
&&\qquad\qquad{}\times E_{\nu_{y+{1}/{2}}} \biggl[\bar h(\eta) \biggl(\sum
_{|x|\leq n}\bar\eta(x) \biggr)^3 \Big|\eta(-n-1)=a,\eta
(n+1)=b \biggr]
\nonumber
\\
&&\qquad\quad{} -3\biggl(\lambda'\biggl(y+\frac{1}{2}\biggr)
\biggr)^3\nonumber\\
&&\qquad\qquad{}\times E_{\nu_{y+{1}/{2}}} \biggl[\bar h(\eta) \biggl(\sum
_{|x|\leq n}\bar\eta(x) \biggr)\Big |\eta(-n-1)=a,\eta (n+1)=b \biggr]
\nonumber
\\
&&\qquad\qquad{}\times E_{\nu_{y+{1}/{2}}} \biggl[ \biggl(\sum_{|x|\leq n}
\bar\eta(x) \biggr)^2\Big |\eta(-n-1)=a,\eta (n+1)=b \biggr].
\nonumber
\end{eqnarray}

\textit{Step} 6.
By the exponentially mixing property of $\nu_{1/2}$ and the assumption
$\varphi_f(1/2)=\varphi'_f(1/2)=0$ [cf. \eqref{varphi_derivatives}], we have
\[
\lambda'(1/2)^2E_{\nu_{1/2}} \biggl[f(\eta)
\biggl(\sum_{|x|\leq n^{2\chi
}}\eta(x)-1/2 \biggr)^2
\biggr] = \varphi_f''(1/2) + O
\bigl(n^{-3/2}\bigr).
\]

Also, note the relation $\lambda'(1/2)\sigma^2(1/2) = 1$ (cf.
Section~\ref{speed_change_model}), and by exponential mixing that
$|\sigma^2(1/2) - \sigma^2_n(1/2)| = O(n^{-1})$. Recall the asymptotic
behaviors of $\kappa_0$, $\kappa_1$ and $\kappa_2$ (cf. step 5 of proof
of Proposition~\ref{2EE}). In addition, a factor $n^{4\chi} y$ is of
order $O(n^{-(1/2-4\chi)})$ in $L^4(\nu_{1/2})$. Hence, with the
parameter $\chi$ chosen small enough, dominant terms may be gathered,
as done in the proof of Proposition~\ref{2EE}, to obtain for all large
$n$ that
\begin{eqnarray*}
&& \biggl\|1\bigl(|y|\leq\delta\bigr) \biggl(E_{\nu_{1/2}}\bigl[f|y,\eta(-n-1)=a,\eta
(n+1)=b\bigr]
\\
&&\hspace*{89pt}\qquad{} - \frac{\varphi_f''(1/2)}{2} \biggl[y^2 - \frac
{\sigma^2_n(1/2)}{2n+1} \biggr]
\biggr) \biggr\|_{L^4(\nu_{1/2})} \\
&&\qquad\leq C\|f\| _{L^\infty} n^{-3/2 + \varepsilon}.
\end{eqnarray*}

On the other hand, large deviation estimates yield
\begin{eqnarray*}
&& \biggl\|1\bigl(|y|>\delta\bigr) \biggl(E_{\nu_{1/2}}\bigl[f|y,\eta(-n-1)=a,\eta(n+1)=b
\bigr]
\\
&&\hspace*{90pt}\qquad{} - \frac{\varphi_f''(1/2)}{2} \biggl[y^2 - \frac{\sigma
^2_n(1/2)}{2n+1} \biggr]
\biggr)\biggr \|_{L^4(\nu_{1/2})} \\
&&\qquad\leq C\|f\| _{L^\infty}n^{-3/2}.
\end{eqnarray*}
\upqed\end{pf}

\section*{Acknowledgments}
P. Gon\c{c}alves and M. Jara are grateful to Capes\break (Brazil) and FCT (Portugal)
for the
research project ``Non-Equilibrium Statistical Mechanics of Stochastic
Lattice Systems'' no. FCT 291/11.
Also, P. Gon\c{c}alves thanks the Research Centre of
Mathematics, University of Minho, and ``Feder'' for support of the
``Programa Operacional Factores de Competiti\-vidade---COMPETE.''
Thanks also to the Editors and referees for their constructive comments.

% imsref loaded by akundreckaite, 2014-02-04 14:45:19

%


%suskaldyti doi

\printaddresses


\begin{thebibliography}{63}
% Style name=ims, version=2.7, label_style=nolabel, sorting_style=complex, cfg=None, language=None.
%b1 ###
%b1 #&#
\bibitem{ACQ}
\begin{barticle}[mr]
\bauthor{\bsnm{Amir},~\bfnm{Gideon}\binits{G.}},
\bauthor{\bsnm{Corwin},~\bfnm{Ivan}\binits{I.}} \AND
\bauthor{\bsnm{Quastel},~\bfnm{Jeremy}\binits{J.}}
(\byear{2011}).
\btitle{Probability distribution of the free energy of the continuum directed random polymer in {$1+1$} dimensions}.
\bjournal{Comm. Pure Appl. Math.}
\bvolume{64}
\bpages{466--537}.
\bid{doi={10.1002/cpa.20347}, issn={0010-3640}, mr={2796514}}
\end{barticle}\vadjust{\goodbreak}
\bptok{imsref}%
% NOT OUTPUTED:
%   issn = 0010-3640
%   url = http://dx.doi.org/10.1002/cpa.20347
%   number = 4
%   coden = CPAMA
%   fjournal = Communications on Pure and Applied Mathematics
\endbibitem

%b2 ###
%b2 #&#
\bibitem{Andjel}
\begin{barticle}[mr]
\bauthor{\bsnm{Andjel},~\bfnm{Enrique~Daniel}\binits{E.~D.}}
(\byear{1982}).
\btitle{Invariant measures for the zero range processes}.
\bjournal{Ann. Probab.}
\bvolume{10}
\bpages{525--547}.
\bid{issn={0091-1798}, mr={0659526}}
\end{barticle}
\bptok{imsref}%
% NOT OUTPUTED:
%   issn = 0091-1798
%   url = http://links.jstor.org/sici?sici=0091-1798(198208)10:3<525:IMFTZR>2.0.CO;2-Q&origin=MSN
%   number = 3
%   coden = APBYAE
%   fjournal = The Annals of Probability
\endbibitem

%b3 ###
%b3 #&#
\bibitem{Assing}
\begin{bmisc}[auto:STB|2014/01/06|10:16:28]
\bauthor{\bsnm{Assing},~\bfnm{S.}\binits{S.}}
(\byear{2011}).
\bhowpublished{A rigorous equation for the Cole--Hopf solution
of the conservative KPZ dynamics.
Available at \arxivurl{arXiv:1109.2886}.}
\end{bmisc}
\bptok{imsref}%
% NOT OUTPUTED:
%   sortkey = Assing(2011
\endbibitem

%b4 ###
%b4 #&#
\bibitem{Assingmartingale}
\begin{barticle}[mr]
\bauthor{\bsnm{Assing},~\bfnm{Sigurd}\binits{S.}}
(\byear{2002}).
\btitle{A pregenerator for {B}urgers equation forced by conservative noise}.
\bjournal{Comm. Math. Phys.}
\bvolume{225}
\bpages{611--632}.
\bid{doi={10.1007/s002200100606}, issn={0010-3616}, mr={1888875}}
\end{barticle}
\bptok{imsref}%
% NOT OUTPUTED:
%   issn = 0010-3616
%   url = http://dx.doi.org/10.1007/s002200100606
%   number = 3
%   coden = CMPHAY
%   fjournal = Communications in Mathematical Physics
\endbibitem

%b5 ###
%b5 #&#
\bibitem{Assingreplacement}
\begin{barticle}[mr]
\bauthor{\bsnm{Assing},~\bfnm{Sigurd}\binits{S.}}
(\byear{2007}).
\btitle{A limit theorem for quadratic fluctuations in symmetric simple exclusion}.
\bjournal{Stochastic Process. Appl.}
\bvolume{117}
\bpages{766--790}.
\bid{doi={10.1016/j.spa.2006.10.005}, issn={0304-4149}, mr={2327838}}
\end{barticle}
\bptok{imsref}%
% NOT OUTPUTED:
%   issn = 0304-4149
%   url = http://dx.doi.org/10.1016/j.spa.2006.10.005
%   number = 6
%   coden = STOPB7
%   fjournal = Stochastic Processes and their Applications
\endbibitem

%b6 ###
%b6 #&#
\bibitem{BDJ}
\begin{barticle}[mr]
\bauthor{\bsnm{Baik},~\bfnm{Jinho}\binits{J.}},
\bauthor{\bsnm{Deift},~\bfnm{Percy}\binits{P.}} \AND
\bauthor{\bsnm{Johansson},~\bfnm{Kurt}\binits{K.}}
(\byear{1999}).
\btitle{On the distribution of the length of the longest increasing subsequence of random permutations}.
\bjournal{J. Amer. Math. Soc.}
\bvolume{12}
\bpages{1119--1178}.
\bid{doi={10.1090/S0894-0347-99-00307-0}, issn={0894-0347}, mr={1682248}}
\end{barticle}
\bptok{imsref}%
% NOT OUTPUTED:
%   issn = 0894-0347
%   url = http://dx.doi.org/10.1090/S0894-0347-99-00307-0
%   number = 4
%   fjournal = Journal of the American Mathematical Society
\endbibitem

%b7 ###
%b7 #&#
\bibitem{Baik-Rains}
\begin{barticle}[mr]
\bauthor{\bsnm{Baik},~\bfnm{Jinho}\binits{J.}} \AND
\bauthor{\bsnm{Rains},~\bfnm{Eric~M.}\binits{E.~M.}}
(\byear{2000}).
\btitle{Limiting distributions for a polynuclear growth model with external sources}.
\bjournal{J. Stat. Phys.}
\bvolume{100}
\bpages{523--541}.
\bid{doi={10.1023/A:1018615306992}, issn={0022-4715}, mr={1788477}}
\end{barticle}
\bptok{imsref}%
% NOT OUTPUTED:
%   issn = 0022-4715
%   url = http://dx.doi.org/10.1023/A:1018615306992
%   number = 3-4
%   coden = JSTPSB
%   fjournal = Journal of Statistical Physics
\endbibitem

%b8 ###
%b8 #&#
\bibitem{BQS}
\begin{barticle}[mr]
\bauthor{\bsnm{Bal{\'a}zs},~\bfnm{M.}\binits{M.}},
\bauthor{\bsnm{Quastel},~\bfnm{J.}\binits{J.}} \AND
\bauthor{\bsnm{Sepp{\"a}l{\"a}inen},~\bfnm{T.}\binits{T.}}
(\byear{2011}).
\btitle{Fluctuation exponent of the {KPZ}/stochastic {B}urgers equation}.
\bjournal{J. Amer. Math. Soc.}
\bvolume{24}
\bpages{683--708}.
\bid{doi={10.1090/S0894-0347-2011-00692-9}, issn={0894-0347}, mr={2784327}}
\end{barticle}
\bptok{imsref}%
% NOT OUTPUTED:
%   issn = 0894-0347
%   url = http://dx.doi.org/10.1090/S0894-0347-2011-00692-9
%   number = 3
%   fjournal = Journal of the American Mathematical Society
\endbibitem

%b9 ###
%b9 #&#
\bibitem{Balazs-R-S}
\begin{barticle}[mr]
\bauthor{\bsnm{Bal{\'a}zs},~\bfnm{M{\'a}rton}\binits{M.}},
\bauthor{\bsnm{Rassoul-Agha},~\bfnm{Firas}\binits{F.}} \AND
\bauthor{\bsnm{Sepp{\"a}l{\"a}inen},~\bfnm{Timo}\binits{T.}}
(\byear{2006}).
\btitle{The random average process and random walk in a space--time random environment in one dimension}.
\bjournal{Comm. Math. Phys.}
\bvolume{266}
\bpages{499--545}.
\bid{doi={10.1007/s00220-006-0036-y}, issn={0010-3616}, mr={2238887}}
\end{barticle}
\bptok{imsref}%
% NOT OUTPUTED:
%   issn = 0010-3616
%   url = http://dx.doi.org/10.1007/s00220-006-0036-y
%   number = 2
%   coden = CMPHAY
%   fjournal = Communications in Mathematical Physics
\endbibitem

%b10 ###
%b10 #&#
\bibitem{Balazs-Seppalainen}
\begin{barticle}[mr]
\bauthor{\bsnm{Bal{\'a}zs},~\bfnm{M{\'a}rton}\binits{M.}} \AND
\bauthor{\bsnm{Sepp{\"a}l{\"a}inen},~\bfnm{Timo}\binits{T.}}
(\byear{2009}).
\btitle{Fluctuation bounds for the asymmetric simple exclusion process}.
\bjournal{ALEA Lat. Am. J. Probab. Math. Stat.}
\bvolume{6}
\bpages{1--24}.
\bid{issn={1980-0436}, mr={2485877}}
\end{barticle}
\bptok{imsref}%
% NOT OUTPUTED:
%   issn = 1980-0436
%   fjournal = ALEA. Latin American Journal of Probability and Mathematical Statistics
\endbibitem

%b11 ###
%b11 #&#
\bibitem{Balazs-Seppalainen1}
\begin{barticle}[mr]
\bauthor{\bsnm{Bal{\'a}zs},~\bfnm{M{\'a}rton}\binits{M.}} \AND
\bauthor{\bsnm{Sepp{\"a}l{\"a}inen},~\bfnm{Timo}\binits{T.}}
(\byear{2010}).
\btitle{Order of current variance and diffusivity in the asymmetric simple exclusion process}.
\bjournal{Ann. of Math. (2)}
\bvolume{171}
\bpages{1237--1265}.
\bid{doi={10.4007/annals.2010.171.1237}, issn={0003-486X}, mr={2630064}}
\end{barticle}
\bptok{imsref}%
% NOT OUTPUTED:
%   issn = 0003-486X
%   url = http://dx.doi.org/10.4007/annals.2010.171.1237
%   number = 2
%   coden = ANMAAH
%   fjournal = Annals of Mathematics. Second Series
\endbibitem

%b12 ###
%b12 #&#
\bibitem{Baxter}
\begin{bbook}[mr]
\bauthor{\bsnm{Baxter},~\bfnm{Rodney~J.}\binits{R.~J.}}
(\byear{1982}).
\btitle{Exactly Solved Models in Statistical Mechanics}.
\bpublisher{Academic Press},
\blocation{London}.
\bid{mr={0690578}}
\end{bbook}
\bptok{imsref}%
% NOT OUTPUTED:
%   isbn = 0-12-083180-5
%   fpage = xii+486
\endbibitem

%b13 ###
%b13 #&#
\bibitem{BC}
\begin{barticle}[mr]
\bauthor{\bsnm{Bertini},~\bfnm{Lorenzo}\binits{L.}} \AND
\bauthor{\bsnm{Cancrini},~\bfnm{Nicoletta}\binits{N.}}
(\byear{1995}).
\btitle{The stochastic heat equation: {F}eynman--{K}ac formula and intermittence}.
\bjournal{J. Stat. Phys.}
\bvolume{78}
\bpages{1377--1401}.
\bid{doi={10.1007/BF02180136}, issn={0022-4715}, mr={1316109}}
\end{barticle}
\bptok{imsref}%
% NOT OUTPUTED:
%   issn = 0022-4715
%   url = http://dx.doi.org/10.1007/BF02180136
%   number = 5-6
%   coden = JSTPSB
%   fjournal = Journal of Statistical Physics
\endbibitem

%b14 ###
%b14 #&#
\bibitem{BG}
\begin{barticle}[mr]
\bauthor{\bsnm{Bertini},~\bfnm{Lorenzo}\binits{L.}} \AND
\bauthor{\bsnm{Giacomin},~\bfnm{Giambattista}\binits{G.}}
(\byear{1997}).
\btitle{Stochastic {B}urgers and {KPZ} equations from particle systems}.
\bjournal{Comm. Math. Phys.}
\bvolume{183}
\bpages{571--607}.
\bid{doi={10.1007/s002200050044}, issn={0010-3616}, mr={1462228}}
\end{barticle}
\bptok{imsref}%
% NOT OUTPUTED:
%   issn = 0010-3616
%   url = http://dx.doi.org/10.1007/s002200050044
%   number = 3
%   coden = CMPHAY
%   fjournal = Communications in Mathematical Physics
\endbibitem

%b15 ###
%b15 #&#
\bibitem{Borodin-Corwin}
\begin{bmisc}[auto:STB|2014/01/06|10:16:28]
\bauthor{\bsnm{Borodin},~\bfnm{A.}\binits{A.}} \AND
\bauthor{\bsnm{Corwin},~\bfnm{I.}\binits{I.}}
(\byear{2012}).
\bhowpublished{Macdonald processes. Preprint.
Available at \arxivurl{arXiv:1111.4408}.}
\end{bmisc}
\bptok{imsref}%
% NOT OUTPUTED:
%   sortkey = Borodin(2012
\endbibitem

%b16 ###
%b16 #&#
\bibitem{BCS}
\begin{bmisc}[auto:STB|2014/01/06|10:16:28]
\bauthor{\bsnm{Borodin},~\bfnm{A.}\binits{A.}},
\bauthor{\bsnm{Corwin},~\bfnm{I.}\binits{I.}} \AND
\bauthor{\bsnm{Sasamoto},~\bfnm{T.}\binits{T.}}
\bhowpublished{Duality to determinants for q-TASEP and
ASEP.  Preprint. Available at \arxivurl{arXiv:1207.5035}.}
\end{bmisc}
\bptok{imsref}%
% NOT OUTPUTED:
%   sortkey = Borodin
\endbibitem

%b17 ###
%b17 #&#
\bibitem{BFS}
\begin{barticle}[mr]
\bauthor{\bsnm{Borodin},~\bfnm{Alexei}\binits{A.}},
\bauthor{\bsnm{Ferrari},~\bfnm{Patrik~L.}\binits{P.~L.}},
\bauthor{\bsnm{Pr{\"a}hofer},~\bfnm{Michael}\binits{M.}} \AND
\bauthor{\bsnm{Sasamoto},~\bfnm{Tomohiro}\binits{T.}}
(\byear{2007}).
\btitle{Fluctuation properties of the {TASEP} with periodic initial configuration}.
\bjournal{J. Stat. Phys.}
\bvolume{129}
\bpages{1055--1080}.
\bid{doi={10.1007/s10955-007-9383-0}, issn={0022-4715}, mr={2363389}}
\end{barticle}
\bptok{imsref}%
% NOT OUTPUTED:
%   issn = 0022-4715
%   url = http://dx.doi.org/10.1007/s10955-007-9383-0
%   number = 5-6
%   fjournal = Journal of Statistical Physics
\endbibitem

%b18 ###
%b18 #&#
\bibitem{BR}
\begin{barticle}[mr]
\bauthor{\bsnm{Brox},~\bfnm{T.}\binits{T.}} \AND
\bauthor{\bsnm{Rost},~\bfnm{H.}\binits{H.}}
(\byear{1984}).
\btitle{Equilibrium fluctuations of stochastic particle systems: The role of conserved quantities}.
\bjournal{Ann. Probab.}
\bvolume{12}
\bpages{742--759}.
\bid{issn={0091-1798}, mr={0744231}}
\end{barticle}
\bptok{imsref}%
% NOT OUTPUTED:
%   issn = 0091-1798
%   url = http://links.jstor.org/sici?sici=0091-1798(198408)12:3<742:EFOSPS>2.0.CO;2-2&origin=MSN
%   number = 3
%   coden = APBYAE
%   fjournal = The Annals of Probability
\endbibitem

%b19 ###
%b19 #&#
\bibitem{Corwinreview}
\begin{barticle}[mr]
\bauthor{\bsnm{Corwin},~\bfnm{Ivan}\binits{I.}}
(\byear{2012}).
\btitle{The {K}ardar--{P}arisi--{Z}hang equation and universality class}.
\bjournal{Random Matrices Theory Appl.}
\bvolume{1}
\bpages{1130001, 76}.
\bid{doi={10.1142/S2010326311300014}, issn={2010-3263}, mr={2930377}}
\end{barticle}
\bptok{imsref}%
% NOT OUTPUTED:
%   issn = 2010-3263
%   url = http://dx.doi.org/10.1142/S2010326311300014
%   number = 1
%   fjournal = Random Matrices. Theory and Applications
\endbibitem

%b20 ###
%b20 #&#
\bibitem{Dittrich}
\begin{barticle}[mr]
\bauthor{\bsnm{Dittrich},~\bfnm{Peter}\binits{P.}} \AND
\bauthor{\bsnm{G{\"a}rtner},~\bfnm{J{\"u}rgen}\binits{J.}}
(\byear{1991}).
\btitle{A central limit theorem for the weakly asymmetric simple exclusion process}.
\bjournal{Math. Nachr.}
\bvolume{151}
\bpages{75--93}.
\bid{doi={10.1002/mana.19911510107}, issn={0025-584X}, mr={1121199}}
\end{barticle}
\bptok{imsref}%
% NOT OUTPUTED:
%   issn = 0025-584X
%   url = http://dx.doi.org/10.1002/mana.19911510107
%   coden = MTMNAQ
%   fjournal = Mathematische Nachrichten
\endbibitem

%b21 ###
%b21 #&#
\bibitem{EK}
\begin{bbook}[mr]
\bauthor{\bsnm{Ethier},~\bfnm{Stewart~N.}\binits{S.~N.}} \AND
\bauthor{\bsnm{Kurtz},~\bfnm{Thomas~G.}\binits{T.~G.}}
(\byear{1986}).
\btitle{Markov Processes: Characterization and Convergence}.
%Probability and Mathematical Statistics}.
\bpublisher{Wiley},
\blocation{New York}.
\bid{doi={10.1002/9780470316658}, mr={0838085}}
\end{bbook}
\bptok{imsref}%
% NOT OUTPUTED:
%   isbn = 0-471-08186-8
%   url = http://dx.doi.org/10.1002/9780470316658
%   fpage = x+534
\endbibitem

%b22 ###
%b22 #&#
\bibitem{Ferrari-Fontes}
\begin{barticle}[mr]
\bauthor{\bsnm{Ferrari},~\bfnm{P.~A.}\binits{P.~A.}} \AND
\bauthor{\bsnm{Fontes},~\bfnm{L.~R.~G.}\binits{L.~R.~G.}}
(\byear{1994}).
\btitle{Current fluctuations for the asymmetric simple exclusion process}.
\bjournal{Ann. Probab.}
\bvolume{22}
\bpages{820--832}.
\bid{issn={0091-1798}, mr={1288133}}
\end{barticle}\vadjust{\eject}
\bptok{imsref}%
% NOT OUTPUTED:
%   issn = 0091-1798
%   url = http://links.jstor.org/sici?sici=0091-1798(199404)22:2<820:CFFTAS>2.0.CO;2-4&origin=MSN
%   number = 2
%   coden = APBYAE
%   fjournal = The Annals of Probability
\endbibitem

%b23 ###
%b23 #&#
\bibitem{Ferrarizrp}
\begin{barticle}[mr]
\bauthor{\bsnm{Ferrari},~\bfnm{P.~A.}\binits{P.~A.}},
\bauthor{\bsnm{Presutti},~\bfnm{E.}\binits{E.}} \AND
\bauthor{\bsnm{Vares},~\bfnm{M.~E.}\binits{M.~E.}}
(\byear{1988}).
\btitle{Nonequilibrium fluctuations for a zero range process}.
\bjournal{Ann. Inst. Henri Poincar\'e Probab. Stat.}
\bvolume{24}
\bpages{237--268}.
\bid{issn={0246-0203}, mr={0953119}}
\end{barticle}
\bptok{imsref}%
% NOT OUTPUTED:
%   issn = 0246-0203
%   url = http://www.numdam.org/item?id=AIHPB_1988__24_2_237_0
%   number = 2
%   coden = AHPBAR
%   fjournal = Annales de l'Institut Henri Poincar\'e. Probabilit\'es et Statistique
\endbibitem

%b24 ###
%b24 #&#
\bibitem{Ferrari-Spohn}
\begin{barticle}[mr]
\bauthor{\bsnm{Ferrari},~\bfnm{Patrik~L.}\binits{P.~L.}} \AND
\bauthor{\bsnm{Spohn},~\bfnm{Herbert}\binits{H.}}
(\byear{2006}).
\btitle{Scaling limit for the space--time covariance of the stationary totally asymmetric simple exclusion process}.
\bjournal{Comm. Math. Phys.}
\bvolume{265}
\bpages{1--44}.
\bid{doi={10.1007/s00220-006-1549-0}, issn={0010-3616}, mr={2217295}}
\end{barticle}
\bptok{imsref}%
% NOT OUTPUTED:
%   issn = 0010-3616
%   url = http://dx.doi.org/10.1007/s00220-006-1549-0
%   number = 1
%   coden = CMPHAY
%   fjournal = Communications in Mathematical Physics
\endbibitem

%b25 ###
%b25 #&#
\bibitem{Gartner}
\begin{barticle}[mr]
\bauthor{\bsnm{G{\"a}rtner},~\bfnm{J{\"u}rgen}\binits{J.}}
(\byear{1988}).
\btitle{Convergence towards {B}urgers' equation and propagation of chaos for weakly asymmetric exclusion processes}.
\bjournal{Stochastic Process. Appl.}
\bvolume{27}
\bpages{233--260}.
\bid{doi={10.1016/0304-4149(87)90040-8}, issn={0304-4149}, mr={0931030}}
\end{barticle}
\bptok{imsref}%
% NOT OUTPUTED:
%   issn = 0304-4149
%   url = http://dx.doi.org/10.1016/0304-4149(87)90040-8
%   number = 2
%   coden = STOPB7
%   fjournal = Stochastic Processes and their Applications
\endbibitem

%b26 ###
%b26 #&#
\bibitem{G.}
\begin{barticle}[mr]
\bauthor{\bsnm{Gon{\c{c}}alves},~\bfnm{Patr{\'{\i}}cia}\binits{P.}}
(\byear{2008}).
\btitle{Central limit theorem for a tagged particle in asymmetric simple exclusion}.
\bjournal{Stochastic Process. Appl.}
\bvolume{118}
\bpages{474--502}.
\bid{doi={10.1016/j.spa.2007.05.002}, issn={0304-4149}, mr={2389054}}
\end{barticle}
\bptok{imsref}%
% NOT OUTPUTED:
%   issn = 0304-4149
%   url = http://dx.doi.org/10.1016/j.spa.2007.05.002
%   number = 3
%   coden = STOPB7
%   fjournal = Stochastic Processes and their Applications
\endbibitem

%b27 ###
%b27 #&#
\bibitem{GJ6}
\begin{bmisc}[auto:STB|2014/01/06|10:16:28]
\bauthor{\bsnm{Gon{\c{c}}alves},~\bfnm{P.}\binits{P.}} \AND
\bauthor{\bsnm{Jara},~\bfnm{M.}\binits{M.}}
(\byear{2010}).
\bhowpublished{Universality of KPZ equation. Available at
\arxivurl{arXiv:1003.4478}.}
\end{bmisc}
\bptok{imsref}%
% NOT OUTPUTED:
%   sortkey = Gonccalves(2010
\endbibitem

%b28 ###
%b28 #&#
\bibitem{GJ5}
\begin{barticle}[mr]
\bauthor{\bsnm{Gon{\c{c}}alves},~\bfnm{Patr{\'{\i}}cia}\binits{P.}} \AND
\bauthor{\bsnm{Jara},~\bfnm{Milton}\binits{M.}}
(\byear{2012}).
\btitle{Crossover to the {KPZ} equation}.
\bjournal{Ann. Henri Poincar\'e}
\bvolume{13}
\bpages{813--826}.
\bid{doi={10.1007/s00023-011-0147-7}, issn={1424-0637}, mr={2913622}}
\end{barticle}
\bptok{imsref}%
% NOT OUTPUTED:
%   issn = 1424-0637
%   url = http://dx.doi.org/10.1007/s00023-011-0147-7
%   number = 4
%   fjournal = Annales Henri Poincar\'e. A Journal of Theoretical and Mathematical Physics
\endbibitem

%b29 ###
%b29 #&#
\bibitem{GLT}
\begin{barticle}[mr]
\bauthor{\bsnm{Gon{\c{c}}alves},~\bfnm{P.}\binits{P.}},
\bauthor{\bsnm{Landim},~\bfnm{C.}\binits{C.}} \AND
\bauthor{\bsnm{Toninelli},~\bfnm{C.}\binits{C.}}
(\byear{2009}).
\btitle{Hydrodynamic limit for a particle system with degenerate rates}.
\bjournal{Ann. Inst. Henri Poincar\'e Probab. Stat.}
\bvolume{45}
\bpages{887--909}.
\bid{doi={10.1214/09-AIHP210}, issn={0246-0203}, mr={2572156}}
\end{barticle}
\bptok{imsref}%
% NOT OUTPUTED:
%   issn = 0246-0203
%   url = http://dx.doi.org/10.1214/09-AIHP210
%   number = 4
%   fjournal = Annales de l'Institut Henri Poincar\'e Probabilit\'es et Statistiques
\endbibitem

%b30 ###
%b30 #&#
\bibitem{Hairer}
\begin{barticle}[mr]
\bauthor{\bsnm{Hairer},~\bfnm{Martin}\binits{M.}}
(\byear{2013}).
\btitle{Solving the {KPZ} equation}.
\bjournal{Ann. of Math. (2)}
\bvolume{178}
\bpages{559--664}.
\bid{doi={10.4007/annals.2013.178.2.4}, issn={0003-486X}, mr={3071506}}
\end{barticle}
\bptok{imsref}%
% NOT OUTPUTED:
%   issn = 0003-486X
%   url = http://dx.doi.org/10.4007/annals.2013.178.2.4
%   number = 2
%   fjournal = Annals of Mathematics. Second Series
\endbibitem

%b31 ###
%b31 #&#
\bibitem{JS}
\begin{bbook}[mr]
\bauthor{\bsnm{Jacod},~\bfnm{Jean}\binits{J.}} \AND
\bauthor{\bsnm{Shiryaev},~\bfnm{Albert~N.}\binits{A.~N.}}
(\byear{2003}).
\btitle{Limit Theorems for Stochastic Processes},
\bedition{2nd} ed.
\bseries{Grundlehren der Mathematischen Wissenschaften}
\bvolume{288}.
\bpublisher{Springer},
\blocation{Berlin}.
\bid{mr={1943877}}
\end{bbook}
\bptok{imsref}%
% NOT OUTPUTED:
%   isbn = 3-540-43932-3
%   fpage = xx+661
\endbibitem

%b32 ###
%b32 #&#
\bibitem{Jara-Landim}
\begin{barticle}[mr]
\bauthor{\bsnm{Jara},~\bfnm{M.~D.}\binits{M.~D.}} \AND
\bauthor{\bsnm{Landim},~\bfnm{C.}\binits{C.}}
(\byear{2006}).
\btitle{Nonequilibrium central limit theorem for a tagged particle in symmetric simple exclusion}.
\bjournal{Ann. Inst. Henri Poincar\'e Probab. Stat.}
\bvolume{42}
\bpages{567--577}.
\bid{doi={10.1016/j.anihpb.2005.04.007}, issn={0246-0203}, mr={2259975}}
\end{barticle}
\bptok{imsref}%
% NOT OUTPUTED:
%   issn = 0246-0203
%   url = http://dx.doi.org/10.1016/j.anihpb.2005.04.007
%   number = 5
%   coden = AHPBAR
%   fjournal = Annales de l'Institut Henri Poincar\'e. Probabilit\'es et Statistiques
\endbibitem

%b33 ###
%b33 #&#
\bibitem{KPZ}
\begin{barticle}[auto:STB|2014/01/06|10:16:28]
\bauthor{\bsnm{Kardar},~\bfnm{M.}\binits{M.}},
\bauthor{\bsnm{Parisi},~\bfnm{G.}\binits{G.}} \AND
\bauthor{\bsnm{Zhang},~\bfnm{Y.~C.}\binits{Y.~C.}}
(\byear{1986}).
\btitle{Dynamic scaling of growing interfaces}.
\bjournal{Phys. Rev. Lett.}
\bvolume{56}
\bpages{889--892}.
\end{barticle}
\bptok{imsref}%
\endbibitem

%b34 ###
%b34 #&#
\bibitem{KL}
\begin{bbook}[mr]
\bauthor{\bsnm{Kipnis},~\bfnm{Claude}\binits{C.}} \AND
\bauthor{\bsnm{Landim},~\bfnm{Claudio}\binits{C.}}
(\byear{1999}).
\btitle{Scaling Limits of Interacting Particle Systems}.
\bseries{Grundlehren der Mathematischen Wissenschaften}
\bvolume{320}.
\bpublisher{Springer},
\blocation{Berlin}.
\bid{mr={1707314}}
\end{bbook}
\bptok{imsref}%
% NOT OUTPUTED:
%   isbn = 3-540-64913-1
%   fpage = xvi+442
\endbibitem

%b35 ###
%b35 #&#
\bibitem{KV}
\begin{barticle}[mr]
\bauthor{\bsnm{Kipnis},~\bfnm{C.}\binits{C.}} \AND
\bauthor{\bsnm{Varadhan},~\bfnm{S.~R.~S.}\binits{S.~R.~S.}}
(\byear{1986}).
\btitle{Central limit theorem for additive functionals of reversible {M}arkov processes and applications to simple exclusions}.
\bjournal{Comm. Math. Phys.}
\bvolume{104}
\bpages{1--19}.
\bid{issn={0010-3616}, mr={0834478}}
\end{barticle}
\bptok{imsref}%
% NOT OUTPUTED:
%   issn = 0010-3616
%   url = http://projecteuclid.org/getRecord?id=euclid.cmp/1104114929
%   number = 1
%   coden = CMPHAY
%   fjournal = Communications in Mathematical Physics
\endbibitem

%b36 ###
%b36 #&#
\bibitem{Kolmogorov}
\begin{bincollection}[mr]
\bauthor{\bsnm{Kolmogorov},~\bfnm{A.~N.}\binits{A.~N.}}
(\byear{1962}).
\btitle{A local limit theorem for {M}arkov chains}.
In \bbooktitle{Select. {T}ransl. {M}ath. {S}tatist. and {P}robability, {V}ol. 2}
\bpages{109--129}.
\bpublisher{Amer. Math. Soc.},
\blocation{Providence, RI}.
\bid{mr={0150810}}
\end{bincollection}
\bptok{imsref}%
\endbibitem

%b37 ###
%b37 #&#
\bibitem{Kumar}
\begin{barticle}[mr]
\bauthor{\bsnm{Kumar},~\bfnm{Rohini}\binits{R.}}
(\byear{2011}).
\btitle{Current fluctuations for independent random walks in multiple dimensions}.
\bjournal{J. Theoret. Probab.}
\bvolume{24}
\bpages{1170--1195}.
\bid{doi={10.1007/s10959-010-0317-4}, issn={0894-9840}, mr={2851250}}
\end{barticle}
\bptok{imsref}%
% NOT OUTPUTED:
%   issn = 0894-9840
%   url = http://dx.doi.org/10.1007/s10959-010-0317-4
%   number = 4
%   coden = JTPREO
%   fjournal = Journal of Theoretical Probability
\endbibitem

%b38 ###
%b38 #&#
\bibitem{LSV}
\begin{barticle}[mr]
\bauthor{\bsnm{Landim},~\bfnm{C.}\binits{C.}},
\bauthor{\bsnm{Sethuraman},~\bfnm{S.}\binits{S.}} \AND
\bauthor{\bsnm{Varadhan},~\bfnm{S.}\binits{S.}}
(\byear{1996}).
\btitle{Spectral gap for zero-range dynamics}.
\bjournal{Ann. Probab.}
\bvolume{24}
\bpages{1871--1902}.
\bid{doi={10.1214/aop/1041903209}, issn={0091-1798}, mr={1415232}}
\end{barticle}
\bptok{imsref}%
% NOT OUTPUTED:
%   issn = 0091-1798
%   url = http://dx.doi.org/10.1214/aop/1041903209
%   number = 4
%   coden = APBYAE
%   fjournal = The Annals of Probability
\endbibitem

%b39 ###
%b39 #&#
\bibitem{Lee}
\begin{barticle}[mr]
\bauthor{\bsnm{Lee},~\bfnm{Eunghyun}\binits{E.}}
(\byear{2010}).
\btitle{Distribution of a particle's position in the {ASEP} with the alternating initial condition}.
\bjournal{J. Stat. Phys.}
\bvolume{140}
\bpages{635--647}.
\bid{doi={10.1007/s10955-010-0014-9}, issn={0022-4715}, mr={2670734}}
\end{barticle}
\bptok{imsref}%
% NOT OUTPUTED:
%   issn = 0022-4715
%   url = http://dx.doi.org/10.1007/s10955-010-0014-9
%   number = 4
%   fjournal = Journal of Statistical Physics
\endbibitem

%b40 ###
%b40 #&#
\bibitem{Liggett}
\begin{bbook}[mr]
\bauthor{\bsnm{Liggett},~\bfnm{Thomas~M.}\binits{T.~M.}}
(\byear{1985}).
\btitle{Interacting Particle Systems}.
\bseries{Grundlehren der Mathematischen Wissenschaften}
\bvolume{276}.
\bpublisher{Springer},
\blocation{New York}.
\bid{doi={10.1007/978-1-4613-8542-4}, mr={0776231}}
\end{bbook}
\bptok{imsref}%
% NOT OUTPUTED:
%   isbn = 0-387-96069-4
%   url = http://dx.doi.org/10.1007/978-1-4613-8542-4
%   fpage = xv+488
\endbibitem

%b41 ###
%b41 #&#
\bibitem{LuYau}
\begin{barticle}[mr]
\bauthor{\bsnm{Lu},~\bfnm{Sheng~Lin}\binits{S.~L.}} \AND
\bauthor{\bsnm{Yau},~\bfnm{Horng-Tzer}\binits{H.-T.}}
(\byear{1993}).
\btitle{Spectral gap and logarithmic {S}obolev inequality for {K}awasaki and {G}lauber dynamics}.
\bjournal{Comm. Math. Phys.}
\bvolume{156}
\bpages{399--433}.
\bid{issn={0010-3616}, mr={1233852}}
\end{barticle}\vadjust{\eject}
\bptok{imsref}%
% NOT OUTPUTED:
%   issn = 0010-3616
%   url = http://projecteuclid.org/getRecord?id=euclid.cmp/1104253633
%   number = 2
%   coden = CMPHAY
%   fjournal = Communications in Mathematical Physics
\endbibitem

%b42 ###
%b42 #&#
\bibitem{Mitoma}
\begin{barticle}[mr]
\bauthor{\bsnm{Mitoma},~\bfnm{Itaru}\binits{I.}}
(\byear{1983}).
\btitle{Tightness of probabilities on {$C([0,1];{\mathcal Y}^{\prime})$} and
{$D([0,1];{\mathcal Y}^{\prime} )$}}.
\bjournal{Ann. Probab.}
\bvolume{11}
\bpages{989--999}.
\bid{issn={0091-1798}, mr={0714961}}
\end{barticle}
\bptok{imsref}%
% NOT OUTPUTED:
%   issn = 0091-1798
%   url = http://links.jstor.org/sici?sici=0091-1798(198311)11:4<989:TOPOA>2.0.CO;2-P&origin=MSN
%   number = 4
%   coden = APBYAE
%   fjournal = The Annals of Probability
\endbibitem

%b43 ###
%b43 #&#
\bibitem{Morris}
\begin{barticle}[mr]
\bauthor{\bsnm{Morris},~\bfnm{Ben}\binits{B.}}
(\byear{2006}).
\btitle{Spectral gap for the zero range process with constant rate}.
\bjournal{Ann. Probab.}
\bvolume{34}
\bpages{1645--1664}.
\bid{doi={10.1214/009117906000000304}, issn={0091-1798}, mr={2271475}}
\end{barticle}
\bptok{imsref}%
% NOT OUTPUTED:
%   issn = 0091-1798
%   url = http://dx.doi.org/10.1214/009117906000000304
%   number = 5
%   coden = APBYAE
%   fjournal = The Annals of Probability
\endbibitem

%b44 ###
%b44 #&#
\bibitem{Mueller}
\begin{barticle}[mr]
\bauthor{\bsnm{Mueller},~\bfnm{Carl}\binits{C.}}
(\byear{1991}).
\btitle{On the support of solutions to the heat equation with noise}.
\bjournal{Stochastics Stochastics Rep.}
\bvolume{37}
\bpages{225--245}.
\bid{issn={1045-1129}, mr={1149348}}
\end{barticle}
\bptok{imsref}%
% NOT OUTPUTED:
%   issn = 1045-1129
%   number = 4
%   coden = STOCBS
%   fjournal = Stochastics and Stochastics Reports
\endbibitem

%b45 ###
%b45 #&#
\bibitem{Nagahata}
\begin{barticle}[mr]
\bauthor{\bsnm{Nagahata},~\bfnm{Yukio}\binits{Y.}}
(\byear{2010}).
\btitle{Spectral gap for zero-range processes with jump rate {$g(x)=x^\gamma$}}.
\bjournal{Stochastic Process. Appl.}
\bvolume{120}
\bpages{949--958}.
\bid{doi={10.1016/j.spa.2010.01.019}, issn={0304-4149}, mr={2610333}}
\end{barticle}
\bptok{imsref}%
% NOT OUTPUTED:
%   issn = 0304-4149
%   url = http://dx.doi.org/10.1016/j.spa.2010.01.019
%   number = 6
%   coden = STOPB7
%   fjournal = Stochastic Processes and their Applications
\endbibitem

%b46 ###
%b46 #&#
\bibitem{Praehofer-Spohn}
\begin{bincollection}[mr]
\bauthor{\bsnm{Pr{\"a}hofer},~\bfnm{Michael}\binits{M.}} \AND
\bauthor{\bsnm{Spohn},~\bfnm{Herbert}\binits{H.}}
(\byear{2002}).
\btitle{Current fluctuations for the totally asymmetric simple exclusion process}.
In \bbooktitle{In and Out of Equilibrium ({M}ambucaba, 2000)}
(\beditor{\binits{V.}\bfnm{V.} \bsnm{Sidoravi\v{c}ius}}, ed.).
\bseries{Progress in Probability}
\bvolume{51}
\bpages{185--204}.
\bpublisher{Birkh\"auser},
\blocation{Boston, MA}.
\bid{mr={1901953}}
\end{bincollection}
\bptok{imsref}%
\endbibitem

%b47 ###
%b47 #&#
\bibitem{Quastel}
\begin{barticle}[mr]
\bauthor{\bsnm{Quastel},~\bfnm{Jeremy}\binits{J.}}
(\byear{1992}).
\btitle{Diffusion of color in the simple exclusion process}.
\bjournal{Comm. Pure Appl. Math.}
\bvolume{45}
\bpages{623--679}.
\bid{doi={10.1002/cpa.3160450602}, issn={0010-3640}, mr={1162368}}
\end{barticle}
\bptok{imsref}%
% NOT OUTPUTED:
%   issn = 0010-3640
%   url = http://dx.doi.org/10.1002/cpa.3160450602
%   number = 6
%   coden = CPAMA
%   fjournal = Communications on Pure and Applied Mathematics
\endbibitem

%b48 ###
%b48 #&#
\bibitem{Quastel-Remenik}
\begin{barticle}[mr]
\bauthor{\bsnm{Quastel},~\bfnm{Jeremy}\binits{J.}} \AND
\bauthor{\bsnm{Remenik},~\bfnm{Daniel}\binits{D.}}
(\byear{2011}).
\btitle{Local {B}rownian property of the narrow wedge solution of the {KPZ} equation}.
\bjournal{Electron. Commun. Probab.}
\bvolume{16}
\bpages{712--719}.
\bid{doi={10.1214/ECP.v16-1678}, issn={1083-589X}, mr={2861435}}
\end{barticle}
\bptok{imsref}%
% NOT OUTPUTED:
%   issn = 1083-589X
%   url = http://dx.doi.org/10.1214/ECP.v16-1678
%   fjournal = Electronic Communications in Probability
\endbibitem

%b49 ###
%b49 #&#
\bibitem{Quastel-Valko}
\begin{barticle}[mr]
\bauthor{\bsnm{Quastel},~\bfnm{Jeremy}\binits{J.}} \AND
\bauthor{\bsnm{Valko},~\bfnm{Benedek}\binits{B.}}
(\byear{2007}).
\btitle{{$t^{1/3}$} {s}uperdiffusivity of finite-range asymmetric
exclusion processes on {$\mathbb{Z}$}}.
\bjournal{Comm. Math. Phys.}
\bvolume{273}
\bpages{379--394}.
\bid{doi={10.1007/s00220-007-0242-2}, issn={0010-3616}, mr={2318311}}
\end{barticle}
\bptok{imsref}%
% NOT OUTPUTED:
%   issn = 0010-3616
%   url = http://dx.doi.org/10.1007/s00220-007-0242-2
%   number = 2
%   coden = CMPHAY
%   fjournal = Communications in Mathematical Physics
\endbibitem

%b50 ###
%b50 #&#
\bibitem{Ravishankar}
\begin{barticle}[mr]
\bauthor{\bsnm{Ravishankar},~\bfnm{K.}\binits{K.}}
(\byear{1992}).
\btitle{Fluctuations from the hydrodynamical limit for the symmetric
simple exclusion in {${\mathbf Z}^d$}}.
\bjournal{Stochastic Process. Appl.}
\bvolume{42}
\bpages{31--37}.
\bid{doi={10.1016/0304-4149(92)90024-K}, issn={0304-4149}, mr={1172505}}
\end{barticle}
\bptok{imsref}%
% NOT OUTPUTED:
%   issn = 0304-4149
%   url = http://dx.doi.org/10.1016/0304-4149(92)90024-K
%   number = 1
%   coden = STOPB7
%   fjournal = Stochastic Processes and their Applications
\endbibitem

%b51 ###
%b51 #&#
\bibitem{Rost-Vares}
\begin{bincollection}[mr]
\bauthor{\bsnm{Rost},~\bfnm{Hermann}\binits{H.}} \AND
\bauthor{\bsnm{Vares},~\bfnm{Maria~Eul{\'a}lia}\binits{M.~E.}}
(\byear{1985}).
\btitle{Hydrodynamics of a one-dimensional nearest neighbor model}.
In \bbooktitle{Particle Systems, Random Media and Large Deviations ({B}runswick, {M}aine, 1984)}.
\bseries{Contemp. Math.}
\bvolume{41}
\bpages{329--342}.
\bpublisher{Amer. Math. Soc.},
\blocation{Providence, RI}.
\bid{doi={10.1090/conm/041/814722}, mr={0814722}}
\end{bincollection}
\bptok{imsref}%
% NOT OUTPUTED:
%   url = http://dx.doi.org/10.1090/conm/041/814722
\endbibitem

%b52 ###
%b52 #&#
\bibitem{Sasamoto-Spohn}
\begin{barticle}[auto:STB|2014/01/06|10:16:28]
\bauthor{\bsnm{Sasamoto},~\bfnm{T.}\binits{T.}} \AND
\bauthor{\bsnm{Spohn},~\bfnm{H.}\binits{H.}}
(\byear{2010}).
\btitle{One-dimensional KPZ equation: An exact solution and its universality}.
\bjournal{Phys. Rev. Lett.}
\bvolume{104}
\bpages{230602}.
\end{barticle}
\bptok{imsref}%
\endbibitem

%b53 ###
%b53 #&#
\bibitem{Sasamoto-Spohn-crossover}
\begin{barticle}[mr]
\bauthor{\bsnm{Sasamoto},~\bfnm{Tomohiro}\binits{T.}} \AND
\bauthor{\bsnm{Spohn},~\bfnm{Herbert}\binits{H.}}
(\byear{2010}).
\btitle{The crossover regime for the weakly asymmetric simple exclusion process}.
\bjournal{J. Stat. Phys.}
\bvolume{140}
\bpages{209--231}.
\bid{doi={10.1007/s10955-010-9990-z}, issn={0022-4715}, mr={2659278}}
\end{barticle}
\bptok{imsref}%
% NOT OUTPUTED:
%   issn = 0022-4715
%   url = http://dx.doi.org/10.1007/s10955-010-9990-z
%   number = 2
%   fjournal = Journal of Statistical Physics
\endbibitem

%b54 ###
%b54 #&#
\bibitem{Seppalainen}
\begin{barticle}[mr]
\bauthor{\bsnm{Sepp{\"a}l{\"a}inen},~\bfnm{Timo}\binits{T.}}
(\byear{2005}).
\btitle{Second-order fluctuations and current across characteristic for a one-dimensional growth model of independent random walks}.
\bjournal{Ann. Probab.}
\bvolume{33}
\bpages{759--797}.
\bid{doi={10.1214/009117904000000946}, issn={0091-1798}, mr={2123209}}
\end{barticle}
\bptok{imsref}%
% NOT OUTPUTED:
%   issn = 0091-1798
%   url = http://dx.doi.org/10.1214/009117904000000946
%   number = 2
%   coden = APBYAE
%   fjournal = The Annals of Probability
\endbibitem

%b55 ###
%b55 #&#
\bibitem{Sextremal}
\begin{barticle}[mr]
\bauthor{\bsnm{Sethuraman},~\bfnm{Sunder}\binits{S.}}
(\byear{2001}).
\btitle{On extremal measures for conservative particle systems}.
\bjournal{Ann. Inst. Henri Poincar\'e Probab. Stat.}
\bvolume{37}
\bpages{139--154}.
\bid{doi={10.1016/S0246-0203(00)01062-1}, issn={0246-0203}, mr={1819121}}
\end{barticle}
\bptok{imsref}%
% NOT OUTPUTED:
%   issn = 0246-0203
%   url = http://dx.doi.org/10.1016/S0246-0203(00)01062-1
%   number = 2
%   coden = AHPBAR
%   fjournal = Annales de l'Institut Henri Poincar\'e. Probabilit\'es et Statistiques
\endbibitem
%
%%b56 ###
%%b56 #&#
%(\byear{2000}).
%% NOT OUTPUTED:
%%   issn = 0010-3640
%%   url = http://dx.doi.org/10.1002/1097-0312(200008)53:8<972::AID-CPA2>3.0.CO;2-#
%%   number = 8
%%   coden = CPAMA
%%   fjournal = Communications on Pure and Applied Mathematics

%b57 ###
%b57 #&#
\bibitem{SX}
\begin{barticle}[mr]
\bauthor{\bsnm{Sethuraman},~\bfnm{Sunder}\binits{S.}} \AND
\bauthor{\bsnm{Xu},~\bfnm{Lin}\binits{L.}}
(\byear{1996}).
\btitle{A central limit theorem for reversible exclusion and zero-range particle systems}.
\bjournal{Ann. Probab.}
\bvolume{24}
\bpages{1842--1870}.
\bid{doi={10.1214/aop/1041903208}, issn={0091-1798}, mr={1415231}}
\end{barticle}
\bptok{imsref}%
% NOT OUTPUTED:
%   issn = 0091-1798
%   url = http://dx.doi.org/10.1214/aop/1041903208
%   number = 4
%   coden = APBYAE
%   fjournal = The Annals of Probability
\endbibitem

%b58 ###
%b58 #&#
\bibitem{Spohn}
\begin{bbook}[auto:STB|2014/01/06|10:16:28]
\bauthor{\bsnm{Spohn},~\bfnm{H.}\binits{H.}}
(\byear{1991}).
\btitle{Large Scale Dynamics of Interacting Particles}.
\bpublisher{Springer},
\blocation{Berlin}.
\end{bbook}
\bptok{imsref}%
\endbibitem

%b59 ###
%b59 #&#
\bibitem{TWreview}
\begin{barticle}[mr]
\bauthor{\bsnm{Tracy},~\bfnm{Craig~A.}\binits{C.~A.}} \AND
\bauthor{\bsnm{Widom},~\bfnm{Harold}\binits{H.}}
(\byear{2011}).
\btitle{Formulas and asymptotics for the asymmetric simple exclusion process}.
\bjournal{Math. Phys. Anal. Geom.}
\bvolume{14}
\bpages{211--235}.
\bid{doi={10.1007/s11040-011-9095-1}, issn={1385-0172}, mr={2824604}}
\end{barticle}
\bptok{imsref}%
% NOT OUTPUTED:
%   issn = 1385-0172
%   url = http://dx.doi.org/10.1007/s11040-011-9095-1
%   number = 3
%   coden = MPAGFO
%   fjournal = Mathematical Physics, Analysis and Geometry. An International Journal Devoted to the Theory and Applications of Analysis and Geometry to Physics
\endbibitem

%b60 ###
%b60 #&#
\bibitem{BKS}
\begin{barticle}[mr]
\bauthor{\bparticle{van} \bsnm{Beijeren},~\bfnm{H.}\binits{H.}},
\bauthor{\bsnm{Kutner},~\bfnm{R.}\binits{R.}} \AND
\bauthor{\bsnm{Spohn},~\bfnm{H.}\binits{H.}}
(\byear{1985}).
\btitle{Excess noise for driven diffusive systems}.
\bjournal{Phys. Rev. Lett.}
\bvolume{54}
\bpages{2026--2029}.
\bid{doi={10.1103/PhysRevLett.54.2026}, issn={0031-9007}, mr={0789756}}
\end{barticle}
\bptok{imsref}%
% NOT OUTPUTED:
%   issn = 0031-9007
%   url = http://dx.doi.org/10.1103/PhysRevLett.54.2026
%   number = 18
%   coden = PRLTAO
%   fjournal = Physical Review Letters
\endbibitem

%b61 ###
%b61 #&#
\bibitem{Varadhannotes}
\begin{bbook}[mr]
\bauthor{\bsnm{Varadhan},~\bfnm{S.~R.~S.}\binits{S.~R.~S.}}
(\byear{2001}).
\btitle{Probability Theory}.
\bseries{Courant Lecture Notes in Mathematics}
\bvolume{7}.
\bpublisher{New York Univ. Courant Institute of Mathematical Sciences},
\blocation{New York}.
\bid{mr={1852999}}
\end{bbook}
\bptok{imsref}%
% NOT OUTPUTED:
%   isbn = 0-8218-2852-5
%   fpage = viii+167
\endbibitem

%b62 ###
%b62 #&#
\bibitem{Walsh}
\begin{bincollection}[mr]
\bauthor{\bsnm{Walsh},~\bfnm{John~B.}\binits{J.~B.}}
(\byear{1986}).
\btitle{An introduction to stochastic partial differential equations}.
In \bbooktitle{\'{E}cole D'\'et\'e de Probabilit\'es de {S}aint-{F}lour, {XIV}---1984}.
\bseries{Lecture Notes in Math.}
\bvolume{1180}
\bpages{265--439}.
\bpublisher{Springer},
\blocation{Berlin}.
\bid{doi={10.1007/BFb0074920}, mr={0876085}}
\end{bincollection}\vadjust{\eject}
\bptok{imsref}%
% NOT OUTPUTED:
%   url = http://dx.doi.org/10.1007/BFb0074920
\endbibitem

%b63 ###
%b63 #&#
\bibitem{MeiYin}
\begin{barticle}[mr]
\bauthor{\bsnm{Yin},~\bfnm{Mei}\binits{M.}}
(\byear{2013}).
\btitle{A {M}arkov chain approach to renormalization group transformations}.
\bjournal{Phys. A}
\bvolume{392}
\bpages{1347--1354}.
\bid{doi={10.1016/j.physa.2012.12.005}, issn={0378-4371}, mr={3019864}}
\end{barticle}
\bptok{imsref}%
% NOT OUTPUTED:
%   issn = 0378-4371
%   url = http://dx.doi.org/10.1016/j.physa.2012.12.005
%   number = 6
%   fjournal = Physica A. Statistical Mechanics and its Applications
\endbibitem

\end{thebibliography}
\end{document}